\documentclass{siamart190516}
\usepackage{braket,amsfonts}
\usepackage{array}
\usepackage[caption=false]{subfig}
\usepackage{graphicx,epstopdf}
\usepackage{amsopn}
\usepackage{enumitem}
\usepackage{amsmath}
\usepackage{amsfonts}
\usepackage{amssymb}
\usepackage{graphicx}
\usepackage{multirow}
\usepackage{enumitem,array}
\usepackage{algorithm,algpseudocode}
\usepackage{graphicx}
\newsiamthm{claim}{Claim}
\newsiamremark{remark}{Remark}
\newsiamremark{hypothesis}{Hypothesis}
\crefname{hypothesis}{Hypothesis}{Hypotheses}

\newcounter{example}
\newenvironment{example}{\refstepcounter{example}\vspace{1ex}
{\sc Example \theexample.}\hspace{0.3em}\parindent=0pt}{\vspace{1ex}}
\begin{document}

\title{Making the Nystr\"om method highly accurate\\for low-rank approximations\thanks{For review.
\funding{The research of Jianlin Xia was supported in part by an NSF grant DMS-2111007.}}
}
\author{Jianlin Xia\thanks{Department of Mathematics, Purdue University, West Lafayette, IN 47907 (xiaj@purdue.edu).}}
\maketitle

\begin{abstract}
The Nystr\"om method is a convenient heuristic method to obtain low-rank approximations to kernel matrices in
nearly linear complexity. Existing
studies typically use the method to approximate positive semidefinite matrices
with low or modest accuracies. In this work, we propose a series of heuristic
strategies to make the Nystr\"om method reach high accuracies for nonsymmetric
and/or rectangular matrices. The resulting methods (called \textit{high-accuracy Nystr\"om methods}) treat the Nystr\"om method and a
skinny rank-revealing factorization as a fast pivoting strategy in a
progressive alternating direction refinement process. Two refinement
mechanisms are used: alternating the row and column pivoting starting from a
small set of randomly chosen columns, and adaptively increasing the number of samples until a desired rank or accuracy is reached. A fast subset update
strategy based on the progressive sampling of Schur complements is further
proposed to accelerate the refinement process. Efficient randomized accuracy
control is also provided.
Relevant accuracy and singular value analysis is given to support some of the
heuristics. Extensive tests with various kernel functions and data sets show how the methods can quickly reach prespecified high accuracies in practice,
sometimes with quality close to SVDs, using only small numbers of progressive
sampling steps.
\end{abstract}

\headers{Jianlin Xia}{High-accuracy Nystr\"om methods} \begin{keywords}
high-accuracy Nystr\"om method, kernel matrix, low-rank approximation, progressive sampling, alternating direction refinement, error analysis
\end{keywords}

\begin{AMS}
15A23, 65F10, 65F30
\end{AMS}

\section{Introduction\label{sec:intro}}

The Nystr\"{o}m method is a very useful technique for data analysis and
machine learning. It can be used to quickly produce low-rank approximations to data matrices. The original Nystr\"{o}m method
in \cite{wil01} is designed for symmetric positive definite kernel matrices and it essentially uses uniform
sampling to select rows/columns (that correspond to some subsets of data
points) to serve as basis matrices in low-rank approximations. It has been
empirically shown to work reasonably well in practice. The Nystr\"{o}m method is highly efficient in the sense that it can produce a low-rank approximation
in complexity linear in the matrix size $n$ (supposing the target
approximation rank $r$ is small).

For problems with high coherence \cite{git16,tal10}, the accuracy of the usual
Nystr\"{o}m method with uniform sampling may be very low. There have been lots
of efforts to improve the method. See, e.g., \cite{dri05,git16,kum02,zha08}.
In order to gain good accuracy, significant extra costs are needed to estimate
leverage scores or determine sampling probabilities in nonuniform sampling
\cite{des06,dri12,ips13}.

Due to its modest accuracy, the Nystr\"{o}m method is usually used for data analysis and not much for regular
numerical computations. In numerical
analysis and scientific computing where controllable high accuracies are
desired, often truncated SVDs or more practical variations like rank-revealing
factorizations \cite{cha87,srrqr} and randomized SVD/sketching methods
\cite{hal11,tro17} are used. These methods can produce highly reliable
low-rank approximations but usually cost $O(n^{2})$ operations.

The purpose of this work is to propose a set of strategies based on the
Nystr\"{o}m method to produce high-accuracy low-rank approximations for kernel matrices in about linear complexity. The matrices are allowed to be
nonsymmetric and/or rectangular. Examples include \emph{off-diagonal blocks}
of larger kernel matrices that frequently arise from numerical solutions of
differential and integral equations, structured eigenvalue solutions, N-body
simulations, and image processing. There has been a rich history in studying
the low-rank structure of these off-diagonal kernel matrices based on ideas
from the fast multipole method (FMM) \cite{gre87} and hierarchical matrix
methods \cite{hac99}. To obtain a low-rank approximation to such a rectangular
kernel matrix $A$ with the Nystr\"{o}m method, a basic way is to choose
respectively random row and column index sets $\mathcal{I}$ and $\mathcal{J}$ and then get a so-called CUR approximation
\begin{equation}
A\approx A_{:,\mathcal{J}}A_{\mathcal{I},\mathcal{J}}^{+}A_{\mathcal{I},:},
\label{eq:cur}\end{equation}
where $A_{:,\mathcal{J}}$ and $A_{\mathcal{I},:}$ denote submatrices formed by
the columns and rows of $A$ corresponding to the index sets $\mathcal{J}$ and
$\mathcal{I}$, respectively, and $A_{\mathcal{I},\mathcal{J}}$ can be
understood similarly. However, the accuracy of (\ref{eq:cur}) is typically
low, unless the so-called volume of $A_{\mathcal{I},\mathcal{J}}$ happens to
be sufficiently large \cite{gor01}. It is well known that finding a submatrix with the maximum volume is NP-hard.

Here, we would like to design adaptive Nystr\"{o}m schemes that can produce
controllable errors (including near machine precision)\ while still retaining
nearly linear complexity in practice. We start by treating the combination of the Nystr\"{o}m method
and a reliable algebraic rank-revealing factorization as a fast pivoting strategy to select significant rows/columns (called
\emph{representative rows/columns} as in \cite{mhs}). We then provide one way
to analyze the resulting low-rank approximation error, which serves as a
motivation for the design of our new\ schemes. Further key strategies include
the following.

\begin{enumerate}
\item Use selected columns and rows to perform fast \emph{alternating direction row and column pivoting}, respectively, so as to refine selections
of representative rows and columns.

\item Adaptively attach a small number of new samples so as to perform
\emph{progressive alternating direction pivoting}, which produces new expanded
representative rows and columns and \emph{advances the numerical rank} needed
to reach high accuracies.

\item Use a fast subset update strategy that successively samples the Schur
complements so as to improve the efficiency and accelerate the advancement of
the sizes of basis matrix toward target numerical ranks.

\item Adaptively control the accuracy via quick estimation of the
approximation errors.
\end{enumerate}

Specifically, in the first strategy above, randomly selected columns are used
to quickly perform row pivoting for $A$ and obtain representative rows (which
form a \emph{row skeleton} $A_{\mathcal{I},:}$). The row skeleton is further
used to quickly perform column pivoting for $A$\ to obtain some representative columns (which form a
\emph{column skeleton} $A_{:,\mathcal{J}}$). This
refines the original choice of representative columns. Related methods include
various forms of the adaptive cross approximation (ACA) with row/column
pivoting \cite{beb00,mac16}, the volume sampling approximation \cite{des06},
and the iterative cross approximation \cite{lua20}. In particular, the method
in \cite{lua20} iteratively refines selections of significant submatrices
(with volumes as large as possible). However, later we can see that this
strategy alone is not enough to reach high accuracy, even if a large number of
initial samples is used.

Next in the second strategy, new column samples are attached progressively in
small stepsizes so as to repeat the alternating direction pivoting until
convergence is reached. Convenient uniform sampling is used since the sampled
columns are for the purpose of pivoting. This eliminates the need of
estimating sampling probabilities. The third strategy enables to avoid
applying pivoting to row/column skeletons with growing sizes. That is, the row (column) skeleton is expanded by
quickly updating the previous skeleton when
new columns (rows) are attached. We also give an aggressive subset update
method that can quickly reach high accuracies with a small number of
progressive sampling steps in practice. With the forth strategy, we can
conveniently control the number of sampling steps until a desired accuracy is
reached. It avoids the need to perform quadratic cost error estimation.

The combination of these strategies leads to a type of low-rank approximation
schemes which we call \emph{high-accuracy Nystr\"{o}m} (HAN) schemes. They are
heuristic schemes that are both fast and accurate in practice. Although a
fully rigorous justification of the accuracy is lacking, we give different
perspectives to motivate and support the ideas. Relevant analysis is provided
to understand certain singular value and accuracy behaviors in terms of both deterministic rank-revealing factorizations and statistical error evaluation.

We demonstrate the high accuracy of the HAN schemes through comprehensive
numerical tests based on kernel matrices defined from various kernel functions
evaluated at different data sets. In particular, an aggressive HAN scheme can
produce approximation accuracies close to the quality of truncated SVDs. It is
numerically shown to have nearly linear complexity and further usually needs
just a surprisingly small number of sampling steps.

Additionally, the design of the HAN schemes does not require analytical
information from the kernel functions or geometric information from the data points. They can then serve as fully blackbox fast low-rank approximation
methods, as indicated in the tests.

The remaining discussions are organized as follows. We show the pivoting
strategy based on the Nystr\"{o}m method and give a way to study the
approximation error in Section \ref{sec:srr}. The detailed design of the HAN
schemes together with relevant analysis is given in Section \ref{sec:han}.
Section \ref{sec:tests} presents the numerical tests, followed by some
concluding remarks in Section \ref{sec:concl}.

\section{Pivoting based on the Nystr\"{o}m method and an error
study\label{sec:srr}}

We first consider a low-rank approximation method based on a pivoting strategy
consisting of the Nystr\"{o}m method and rank-revealing factorizations of tall and skinny matrices. A way to study the low-rank approximation error will then
be given. These will provide motivations for some of our ideas in the HAN schemes.

Consider two sets of real data points in $d$ dimensions:
\[
\mathbf{x}=\left\{  x_{1},x_{2},\ldots,x_{m}\right\}  ,\quad\mathbf{y}=\left\{  y_{1},y_{2},\ldots,y_{n}\right\}  .
\]
Let $A$ be the $m\times n$ kernel matrix\begin{equation}
A=\left(  \kappa(x_{i},y_{j})\right)  _{x_{i}\in\mathbf{x},y_{j}\in\mathbf{y}}, \label{eq:kermat}\end{equation}
which is sometimes also referred to as the interaction matrix\ between
$\mathbf{x}$ and $\mathbf{y}$. We would like to approximate $A$ by a low-rank
form. The strong rank-revealing QR or LU\ factorizations \cite{srrqr,mir03}
are reliable ways to find low-rank approximations with high accuracy. They may
be used to obtain an approximation (called interpolative decomposition) of the
following form:
\begin{equation}
A\approx A_{:,\mathcal{J}}V^{T}\quad\text{with\quad}V=Q\left(
\begin{array}
[c]{c}I\\
F
\end{array}
\right)  , \label{eq:v}\end{equation}
where $Q$ is a permutation matrix, $r\equiv|\mathcal{J}|$ (size or cardinality
of $\mathcal{J}$) is the approximate (or numerical) rank, and $\Vert
F\Vert_{\max}\leq c$ with $c\geq1$. $c$ is a user-specified parameter and may be set to be a constant or a low-degree polynomial of $m$, $n$, and $r$ \cite{srrqr}.

We suppose $r$ is small. The column skeleton $A_{:,\mathcal{J}}$ corresponds
to a subset $\mathbf{t}\subset\mathbf{y}$ which is a subset of \emph{landmark
points}. Here we also call $\mathbf{t}$ a representative subset, which can be
selected reliably by strong rank-revealing\ factorizations. A strong
rank-revealing\ factorization may be further applied to $A_{:,\mathcal{J}}^{T}$ to select a representative subset $\mathbf{s}\subset\mathbf{x}$
corresponding to a row index set $\mathcal{I}$ in $A_{:,\mathcal{J}}$. That
is, we can find a \emph{pivot block} $A_{\mathcal{I},\mathcal{J}}$. Without
loss of generality, we may assume $|\mathcal{I}|=|\mathcal{J}|=r$. (If the
factorization produces $\mathcal{I}$ with $|\mathcal{I}|<|\mathcal{J}|$, $V$ can be modified so as to replace $\mathcal{J}$ by an appropriate index set
with size $|\mathcal{I}|$.) Thus, the resulting decomposition may be written
as an equality
\begin{equation}
A_{:,\mathcal{J}}=UA_{\mathcal{I},\mathcal{J}}\quad\text{with\quad}U=P\left(
\begin{array}
[c]{c}I\\
E
\end{array}
\right)  , \label{eq:aj}\end{equation}
where $P$ is a permutation matrix and, with $1:m$ standing for $1,2,\ldots,m$,
\begin{equation}
E=A_{\{1:m\}\backslash\mathcal{I},\mathcal{J}}A_{\mathcal{I},\mathcal{J}}^{-1},\quad\Vert E\Vert_{\max}\leq c. \label{eq:e}\end{equation}
Since $A_{:,\mathcal{J}}$ is a tall and skinny matrix, we refer to
(\ref{eq:aj}) as a \emph{skinny rank-revealing (SRR) factorization}.
(\ref{eq:v}) and (\ref{eq:aj}) in turn lead to the approximation\begin{equation}
A\approx UA_{\mathcal{I},\mathcal{J}}V^{T}. \label{eq:ubv}\end{equation}
With (\ref{eq:ubv}), we may further obtain a CUR\ approximation like in
(\ref{eq:cur}) (with the pseudoinverse replaced by $A_{\mathcal{I},\mathcal{J}}^{-1}$).

The direct application of strong rank-revealing\ factorizations to $A$ to
obtain (\ref{eq:v}) is expensive and costs $O(rmn)$. To reduce the cost, we
can instead follow the Nystr\"{o}m method and randomly sample columns from $A$ to form $A_{:,\mathcal{J}}$. However, the accuracy of the resulting
approximation based on the forms (\ref{eq:v}) or (\ref{eq:ubv}) may be low. On
the other hand, we can view the SRR factorization (\ref{eq:aj}) as a way to
quickly choose the representative subset $\mathbf{s}$ (based on the
interaction between $\mathbf{x}$ and $\mathbf{t}$ instead of the interaction
between $\mathbf{x}$ and $\mathbf{y}$). In other words, (\ref{eq:aj}) is a way
to quickly perform \emph{row pivoting} for $A$ so as to select representative rows $A_{\mathcal{I},:}$ from $A$. Then we can use the following low-rank approximation:
\begin{equation}
A\approx UA_{\mathcal{I},:}\quad\text{with\quad}U=P\left(
\begin{array}
[c]{c}I\\
E
\end{array}
\right)  , \label{eq:srrqr}\end{equation}
which may be viewed as a potentially refined form over (\ref{eq:v}) when
$\mathcal{J}$ is randomly selected. (Note that $P$ and $E$ depend on
$\mathcal{J}$.)

We would like to gain some insights into the accuracy of approximations based
on the Nystr\"{o}m method. There are various earlier studies based on
(\ref{eq:cur}). Those in \cite{cai22,zha08} are relevant to our result below.
When $A$ is positive (semi-)definite, the analysis in \cite{zha08} bounds the
errors in terms of the distances between the landmark points and the remaining
data points. A similar strategy is also followed in \cite[Lemma 3.1]{cai22}
for symmetric $A$. The resulting bound may be very conservative since it is
common for some data points in practical data sets to be far away from the
landmark points. In addition, the error bounds in \cite{cai22,zha08}
essentially involve a factor $\Vert A_{\mathcal{I},\mathcal{J}}^{-1}\Vert_{2}$
(or $\Vert A_{\mathcal{I},\mathcal{J}}^{+}\Vert_{2}$), which may be too large
if high accuracy is desired. This is because the smallest singular value of $A_{11}$ may be just slightly larger than a smaller tolerance.

Here, we provide a way to understand the approximation error based on
(\ref{eq:srrqr}). It uses the minimization of a slightly overdetermined
problem and does not involve $\Vert A_{\mathcal{I},\mathcal{J}}^{-1}\Vert_{2}$. The following analysis does not aim to precisely quantify the error
magnitude (which is hard anyway). Instead, it can serve as a motivation for
some strategies in our high-accuracy Nystr\"{o}m methods later.

\begin{lemma}
Suppose $\mathcal{J}$ is a given column index set with $|\mathcal{J}|=r$ and
(\ref{eq:aj})--(\ref{eq:e}) hold. Then the resulting approximation
(\ref{eq:srrqr}) satisfies\begin{equation}
\Vert A-UA_{\mathcal{I},:}\Vert_{\max}\leq2c\sqrt{r}\max_{1\leq i\leq m,1\leq
j\leq n}\min_{v\in\mathbb{R}^{r}}\Vert A_{\widetilde{\mathcal{I}}_{i},\mathcal{J}}v-A_{\widetilde{\mathcal{I}}_{i},j}\Vert_{2}, \label{eq:err}\end{equation}
where $\widetilde{\mathcal{I}}_{i}=\mathcal{I}\cup\{i\}$ for each $1\leq i\leq
m$.
\end{lemma}

\begin{proof}
From (\ref{eq:aj}) and (\ref{eq:srrqr}), we have, for any $1\leq i\leq m$,
$1\leq j\leq n$,\[
(A-UA_{\mathcal{I},:})_{ij}=(A-A_{:,\mathcal{J}}A_{\mathcal{I},\mathcal{J}}^{-1}A_{\mathcal{I},:})_{ij}=A_{ij}-A_{i,\mathcal{J}}A_{\mathcal{I},\mathcal{J}}^{-1}A_{\mathcal{I},j}.
\]
It is obvious that $A_{ij}-A_{i,\mathcal{J}}A_{\mathcal{I},\mathcal{J}}^{-1}A_{\mathcal{I},j}=0$ if $i\in\mathcal{I}$. Thus, suppose $i\in
\{1:m\}\backslash\mathcal{I}$.

For any $v\in\mathbb{R}^{r}$,\begin{align*}
|A_{ij}-A_{i,\mathcal{J}}A_{\mathcal{I},\mathcal{J}}^{-1}A_{\mathcal{I},j}|
&  =|(A_{ij}-A_{i,\mathcal{J}}v)+(A_{i,\mathcal{J}}v-A_{i,\mathcal{J}}A_{\mathcal{I},\mathcal{J}}^{-1}A_{\mathcal{I},j})|\\
&  \leq|A_{ij}-A_{i,\mathcal{J}}v|+\Vert A_{i,\mathcal{J}}A_{\mathcal{I},\mathcal{J}}^{-1}\Vert_{2}\Vert A_{\mathcal{I},\mathcal{J}}v-A_{\mathcal{I},j}\Vert_{2}\\
&  \leq|A_{ij}-A_{i,\mathcal{J}}v|+c\sqrt{r}\Vert A_{\mathcal{I},\mathcal{J}}v-A_{\mathcal{I},j}\Vert_{2},
\end{align*}
where the last step is because $A_{i,\mathcal{J}}A_{\mathcal{I},\mathcal{J}}^{-1}$ is a row of $E$ in (\ref{eq:e}) and its entries have magnitudes
bounded by $c$. With $c\geq1$, we further have\begin{align*}
|A_{ij}-A_{i,\mathcal{J}}A_{\mathcal{I},\mathcal{J}}^{-1}A_{\mathcal{I},j}|
&  \leq c\sqrt{r}\left(  |A_{i,\mathcal{J}}v-A_{ij}|+\Vert A_{\mathcal{I},\mathcal{J}}v-A_{\mathcal{I},j}\Vert_{2}\right) \\
&  \leq2c\sqrt{r}\Vert A_{\widetilde{\mathcal{I}}_{i},\mathcal{J}}v-A_{\widetilde{\mathcal{I}}_{i},j}\Vert_{2}.
\end{align*}
Since this holds for all $v\in
\mathbb{R}
^{r}$, take the minimum for $v$ to get the desired result.
\end{proof}

The bound in this lemma can be roughly understood as follows. If $A_{\widetilde{\mathcal{I}}_{i},j}$ is nearly in the range of
$A_{\widetilde{\mathcal{I}}_{i},\mathcal{J}}$ for all $i,j$, the bound in
(\ref{eq:err}) would then be very small and we would have found $\mathcal{I}$
and $\mathcal{J}$ that produce an accurate low-rank approximation
(\ref{eq:srrqr}). Otherwise, to further improve the accuracy, it would be
necessary to refine $\mathcal{I}$ and $\mathcal{J}$ and possibly include
additional $i$ and $j$ indices respectively into $\mathcal{I}$ and
$\mathcal{J}$. A heuristic strategy is to progressively pick $i$ and $j$ so
that $A_{\widetilde{\mathcal{I}}_{i},j}$ is as linearly independent from the
columns of $A_{\widetilde{\mathcal{I}}_{i},\mathcal{J}}$ as possible.
Motivated by this, we may use a subset refinement process. First, use randomly picked columns
$A_{:,\mathcal{J}}$ to generate a row skeleton and then use the
row skeleton to generate a new column skeleton. The new column skeleton
suggests which new $j$ should be attached to $\mathcal{J}$. Next, if a desired
accuracy is not reached,
then randomly pick more columns to attach to the refined set $\mathcal{J}$ and
start a new round of refinement. Such a process is called \emph{progressive
alternating direction pivoting} (or \emph{subset refinement}) below.

\section{High-accuracy Nystr\"{o}m schemes\label{sec:han}}

In this section, we show how to use the Nystr\"{o}m method to design the
high-accuracy Nystr\"{o}m (HAN) schemes that can produce highly accurate
low-rank approximations in practice.
We begin with the basic idea of the progressive alternating direction pivoting
and then show how to perform fast subset update and how to conveniently
control the accuracy.

\subsection{Progressive alternating direction pivoting}

The direct application of strong rank-revealing factorizations to $A$ has
quadratic complexity. One way to save the cost is as follows. Start from some
column samples of $A$ like in the usual Nystr\"{o}m method. Use the SRR
factorization to select a row\ skeleton, which can then be used to select a
refined column skeleton. The process can be repeated in a recursive way,
leading to a fast alternating direction refinement scheme. A similar empirical scheme has been adopted recently in \cite{lua20,pan19}.
However, when high accuracies are desired, the effectiveness of this scheme
may be limited. That is, just like the usual Nystr\"{o}m method, a brute-force
increase of the initial sample size may not necessarily improve the
approximation accuracy significantly. A high accuracy may require the initial sample size to be overwhelmingly larger than the target numerical rank, which
makes the cost too high.

Here, we instead adaptively or progressively apply the alternating direction refinement based on step-by-step small increases of the sample size.
We use one round of alternating row and column pivoting to refine the subset
selections. After this, if a target accuracy $\tau$ or numerical rank $r$ is
not reached, we include a small number of additional samples to repeat the procedure.

The basic framework to find a low-rank approximation to $A$ in
(\ref{eq:kermat}) is as follows, where the subset $\mathcal{J}$ is initially
an empty set and $b\leq r$ is a small integer as the \emph{stepsize} in the
progressive column sampling.

\begin{enumerate}
\item \label{step:rand}(\emph{Progressive sampling})\ Randomly choose a column
index set $\widehat{\mathcal{J}}\subset\{1:n\}\backslash\mathcal{J}$ with
$|\widehat{\mathcal{J}}|=b$ and set
\[
\widetilde{\mathcal{J}}=\mathcal{J}\cup\widehat{\mathcal{J}}.
\]

\item \label{step:row}(\emph{Row pivoting})\ Apply an SRR\ factorization to
$A_{:,\widetilde{\mathcal{J}}}$ to find a row index set $\mathcal{I}$:
\begin{equation}
A_{:,\widetilde{\mathcal{J}}}\approx UA_{\mathcal{I},\widetilde{\mathcal{J}}},
\label{eq:rowskel}\end{equation}
where $U$ looks like that in (\ref{eq:aj}).

\item \label{stel:col1}(\emph{Column pivoting})\ Apply an SRR\ factorization
to $A_{\mathcal{I},:}$ to find a refined column index set $\mathcal{J}$:
\begin{equation}
A_{\mathcal{I},:}\approx A_{\mathcal{I},\mathcal{J}}V^{T}, \label{eq:colskel}\end{equation}
where $V$ looks like that in (\ref{eq:v}).

\item (\emph{Accuracy check})\ If a desired accuracy, maximum sample size, or
a target numerical rank is reached or if $\mathcal{I}$ stays the same as in
the previous step, return a low-rank approximation to $A$ like the following
and exit:\[
\tilde{A}=UA_{\mathcal{I},:},\quad A_{:,\mathcal{J}}V^{T},\quad\text{or\quad
}UA_{\mathcal{I},\mathcal{J}}V^{T}.
\]

Otherwise, repeat from Step \ref{step:rand}. (More details on the stopping criteria and fast error estimation will be given in Section \ref{sub:acc}.)
\end{enumerate}

This basic HAN scheme (denoted \textsf{HAN-B}) is illustrated in Figure \ref{fig:han}, with more details given in Algorithm \ref{alg:han-b}.
Note that the key outputs of the SRR factorization (\ref{eq:aj}) are the index
set $\mathcal{I}$ and the matrix $E$. (The permutation matrix $P$ is just to
bring the index set $\mathcal{I}$ to the leading part and does not need to be
stored.) For convenience, we denote (\ref{eq:aj}) by the following procedure
in Algorithm \ref{alg:han-b} (with the parameter $c$ in (\ref{eq:e}) assumed
to be fixed):
\[
\lbrack\mathcal{I},E]\leftarrow\mathsf{SRR}(A_{:,\mathcal{J}}).
\]

The scheme may be understood heuristically as follows. Initially, with $\widetilde{\mathcal{J}}$ a random sample from the column indices, it is known
that the expectation of the norm of a row of $A_{:,\widetilde{\mathcal{J}}}$
is a multiple of the norm of the corresponding row in $A$ (see, e.g.,
\cite{avr11,dri06}). Thus, the relative magnitudes of the row norms of $A$ can
be roughly reflected by those of $A_{:,\widetilde{\mathcal{J}}}$. It then
makes sense to use $A_{:,\widetilde{\mathcal{J}}}$ for quick row pivoting (by
finding $A_{\mathcal{I},\widetilde{\mathcal{J}}}$ with determinant as large as
possible). This strategy shares features similar to the randomized pivoting
strategies in \cite{mar17,xia17} which are also heuristic and work well in
practice, except that the methods in \cite{mar17,xia17} need matrix-vector
multiplications with costs $O(mn)$.

\begin{figure}[ptbh]
\centering\includegraphics[height=1.1in]{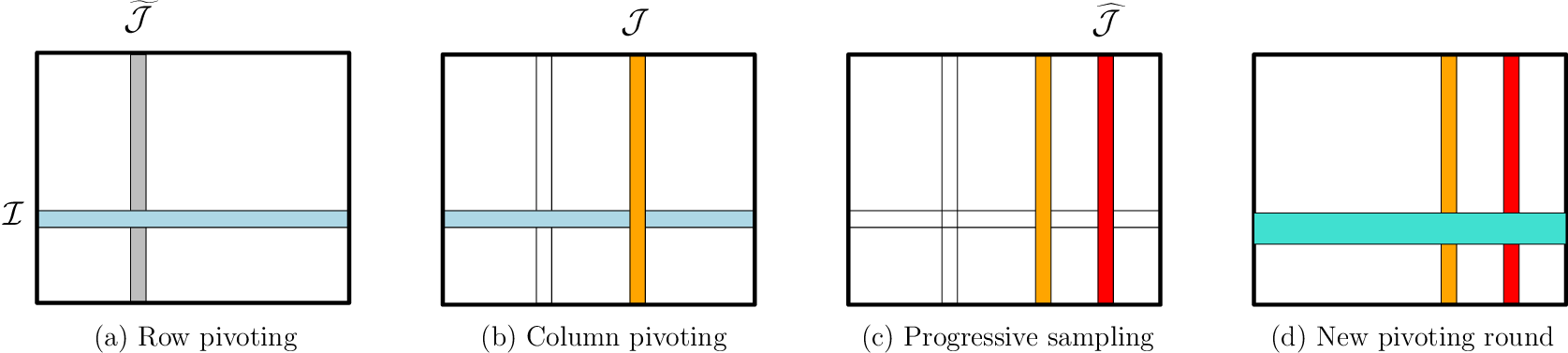}\caption{Illustration of the
basic high-accuracy Nystr\"{o}m (\textsf{HAN-B}) scheme.}\label{fig:han}\end{figure}

\begin{algorithm}
[!h]\caption{Basic HAN scheme}\label{alg:han-b}

\begin{algorithmic}
[1]\Procedure{\sf HAN-B}{$A$, $b$, $\tau$ (or $r$)}

\State\label{line:empty}$\mathcal{J}\leftarrow\varnothing$

\For{$i=1,2,\ldots$}

\State$\widehat{\mathcal{J}}\leftarrow$ randomly chosen index set from
$\{1:n\}\backslash\mathcal{J}$ with $|\widehat{\mathcal{J}}|=b$

\State$\widetilde{\mathcal{J}}\leftarrow\mathcal{J}\cup\widehat{\mathcal{J}}$\Comment{Progressive sampling}

\State\label{line:ibar}$\overline{\mathcal{I}}\leftarrow\mathcal{I}$

\State\label{line:ie}$[\mathcal{I},E]\leftarrow$ \textsf{SRR}$(A_{:,\widetilde{\mathcal{J}}})$\Comment{Row pivoting}

\State$\widehat{\mathcal{I}}\leftarrow\mathcal{I}\backslash\overline
{\mathcal{I}}$\Comment{Removal of indices in the previous $\mathcal{I}$ from the new one}

\State\label{line:jf}$[\mathcal{J},F]\leftarrow$ \textsf{SRR}$(A_{\mathcal{I},:}^{T})$\Comment{Column pivoting}

\If{a target accuracy $\tau$ or numerical rank $r$ is reached or $\widehat{\mathcal{I}}=\varnothing$}

\Comment{See Section \ref{sub:acc} for more discussions}

\State Return $A\approx UA_{\mathcal{I},:}$, $A_{:,\mathcal{J}}V^{T}$, or
$UA_{\mathcal{I},\mathcal{J}}V^{T}$ \EndIf
\EndFor
\EndProcedure

\end{algorithmic}
\end{algorithm}

With the resulting row pivot index set $\mathcal{I}$, the scheme further uses
the SRR\ factorization to find a submatrix $A_{\mathcal{I},\mathcal{J}}$ of
$A_{\mathcal{I},:}$ with determinant as large as possible, which\ enables to refine the column selection. It may be possible to further improve the index
sets through multiple rounds of such refinements like in \cite{lua20,pan19}.
However, the accuracy gain seems limited, even if a large initial sample size
is used (as shown in our test later). Thus, we progressively attach additional
samples (in small stepsizes) to the refined subset $\mathcal{J}$ and then
repeat the previous procedure. In practice, this makes a significant difference
in reducing the approximation error.

In this scheme, the sizes of the index sets $\mathcal{I}$ and $\mathcal{J}$
grow with the progressive sampling. Accordingly, the costs of the SRR
factorizations (\ref{eq:rowskel})--(\ref{eq:colskel}) increase since the
SRR\ factorizations at step $i$ are applied to matrices of sizes $m\times(ib)$ or $(ib)\times n$. With the total number of iterations $N\approx\frac{r}{b}$,
the total cost (excluding the cost to check the accuracy) is\begin{equation}
\xi_{\mathsf{HAN-B}}=\sum_{i=1}^{N}O\left(  (ib)^{2}(m+n)\right)  =O\left(
\frac{r^{3}}{b}(m+n)\right)  . \label{eq:cost1}\end{equation}
With $i$ increases, the iterations advance toward the target numerical rank or accuracy.

\subsection{Fast subset update via Schur complement sampling\label{sub:upd}}

In the basic scheme \textsf{HAN-B}, the complexity count in (\ref{eq:cost1})
for the SRR factorizations at step $i$ gets higher with increasing $i$. To
improve the efficiency, we show how to update the index sets so that at step
$i$, the SRR factorization (for the row pivoting step for example) only needs
to be applied to a matrix of size $(m-(i-1)b)\times b$ instead of
$m\times(ib)$, followed by some quick postprocessing steps.

Suppose we start from a column index set $\widetilde{\mathcal{J}}=\mathcal{J}\cup\widehat{\mathcal{J}}$ as in Step \ref{step:rand} of the basic
HAN\ scheme above. We would like to avoid applying the SRR factorizations to
the full columns $A_{:,\widetilde{\mathcal{J}}}$ in Step \ref{step:row} and
the full rows $A_{\mathcal{I},:}$ in Step \ref{stel:col1}. We seek to directly
produce an expanded column index set over $\mathcal{J}$, as illustrated in
Figure \ref{fig:han1}. It includes two steps. One is to produce an update
$\widehat{\mathcal{I}}$ to the row index set $\mathcal{I}$ (Figure
\ref{fig:han1}(a), which replaces Steps (c)--(d)\ in Figure \ref{fig:han}) and
the other is to produce an update to the column index set (Figure
\ref{fig:han1}(b)). Clearly, we just need to show how to perform the first
step. \begin{figure}[ptbh]
\centering\includegraphics[height=1.1in]{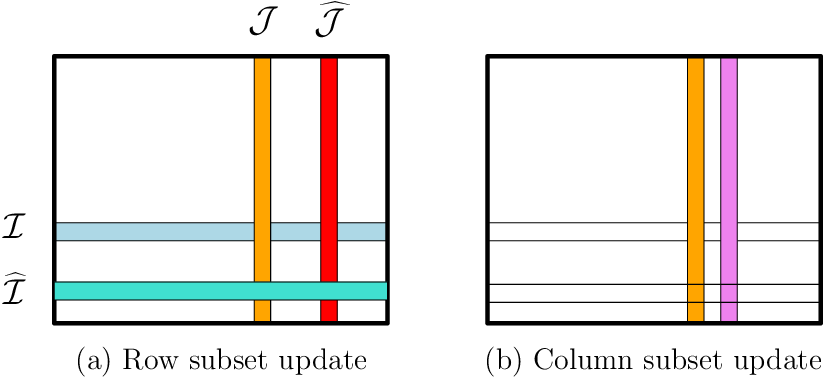}\caption{Illustration of the
subset update process.}\label{fig:han1}\end{figure}

With the row pivoting step like in (\ref{eq:rowskel}), we can obtain a
low-rank approximation of the form (\ref{eq:srrqr}). Using the row permutation
matrix $P$ in (\ref{eq:srrqr}) (computed in (\ref{eq:aj})), we may write $A$
as\begin{align}
A  &  =P\left(
\begin{array}
[c]{cc}A_{11} & A_{12}\\
A_{21} & A_{22}\end{array}
\right)  \approx P\left(
\begin{array}
[c]{c}I\\
E
\end{array}
\right)  \left(
\begin{array}
[c]{cc}A_{11} & A_{12}\end{array}
\right)  \quad\text{with}\label{eq:a11}\\
A_{11}  &  =A_{\mathcal{I},\mathcal{J}},\quad E=A_{21}A_{11}^{-1}.\nonumber
\end{align}
At this point, we have\begin{equation}
A=P\left(
\begin{array}
[c]{cc}I & \\
E & I
\end{array}
\right)  \left(
\begin{array}
[c]{cc}A_{11} & A_{12}\\
& S
\end{array}
\right)  \quad\text{with\quad}S=A_{22}-EA_{12}, \label{eq:afact1}\end{equation}
where $S$ is the \emph{Schur complement}.

\begin{remark}
\label{rem:srrlu}In the usual strong rank-revealing factorizations like the
one in \cite{mir03}, the low-rank approximation is obtained also from a
decomposition of the form (\ref{eq:afact1}) with $S$ dropped. Here, our fast
pivoting scheme is more efficient. Of course, the strong rank-revealing
factorization in \cite{mir03} guarantees the quality of the low-rank
approximation in the sense that, there exist low-degree polynomials $c\geq1$
and $f\geq1$ in $m$, $n$, and $k$ (size of $A_{11}$) such that (\ref{eq:e})
holds and, for $1\leq i\leq k$, $1\leq j\leq\min\{m,n\}-k$,
\begin{equation}
\sigma_{i}(A_{11})\geq\frac{\sigma_{i}(A)}{f},\quad\sigma_{j}(S)\leq
\sigma_{k+j}(A)f,\quad\Vert A_{11}^{-1}A_{12}\Vert_{\max}\leq c,
\label{eq:svsrrlu}\end{equation}
where $\sigma_{i}(\cdot)$ denotes the $i$-th largest singular value of a matrix.
\end{remark}

Our subset update strategy is via the sampling of the Schur complement $S$. In fact, when $A_{\mathcal{I},:}=\left(
\begin{array}
[c]{cc}A_{11} & A_{12}\end{array}
\right)  $ is accepted as a reasonable row skeleton, we then continue to find
a low-rank approximation to $S$ in (\ref{eq:afact1}) so it makes sense to
sample $S$. It is worth noting that the full matrix $S$ is not needed.
Instead, only its columns corresponding to $A_{:,\widehat{\mathcal{J}}}$ are formed. That is, we form
\[
S_{:,\mathcal{L}}=(A_{22})_{:,\mathcal{L}}-E(A_{12})_{:,\mathcal{L}},
\]
where $\mathcal{L}$ corresponds to $\widehat{\mathcal{J}}$ and selects entries
from $\{1:n\}\backslash\mathcal{J}$ in a two-level composition of the index
sets as follows:
\begin{equation}
(\{1:n\}\backslash\mathcal{J})\circ\mathcal{L}=\widehat{\mathcal{J}}.
\label{eq:jtilde}\end{equation}
That is, sampling the columns of $A$ with the index set $\widehat{\mathcal{J}}$ is essentially to sample the columns of $S$ with $\mathcal{L}$. For
notational convenience, suppose the columns of $A$ have been permuted so that
\[
P^{T}A=\left(
\begin{array}
[c]{c|c}A_{11} & A_{12}\\
A_{21} & A_{22}\end{array}
\right)  \quad\text{with\quad}A_{:,\mathcal{J}}=\begin{pmatrix}
A_{11}\\
A_{21}\end{pmatrix}
.
\]

Now, apply an SRR\ factorization to $S_{:,\mathcal{L}}$ to get\begin{equation}
S_{:,\mathcal{L}}\approx\hat{P}\left(
\begin{array}
[c]{c}I\\
\hat{E}\end{array}
\right)  S_{\mathcal{K},\mathcal{L}}. \label{eq:sskel}\end{equation}
Then $S\approx\hat{P}\left(
\begin{array}
[c]{c}I\\
\hat{E}\end{array}
\right)  S_{\mathcal{K},:}$. Accordingly, we may write $S$ as\begin{equation}
S=\hat{P}\left(
\begin{array}
[c]{cc}S_{22} & S_{23}\\
S_{32} & S_{33}\end{array}
\right)  , \label{eq:spart}\end{equation}
where $S_{\mathcal{K},:}=\left(
\begin{array}
[c]{cc}S_{22} & S_{23}\end{array}
\right)  $. From (\ref{eq:sskel}) and (\ref{eq:spart}), $S$ can be further
written as\begin{equation}
S=\hat{P}\left(
\begin{array}
[c]{cc}I & \\
\hat{E} & \hat{S}\end{array}
\right)  \left(
\begin{array}
[c]{cc}S_{22} & S_{23}\\
& I
\end{array}
\right)  \quad\text{with\quad}\hat{S}=S_{33}-\hat{E}S_{23}, \label{eq:s}\end{equation}
where $\hat{S}$ is a new Schur complement (and is not formed).

At this point, we have the following proposition which shows how to expand the
row index set $\mathcal{I}$ by an update $\widehat{\mathcal{I}}$.

\begin{proposition}
$A$ may be factorized as\begin{equation}
A=\tilde{P}\left(
\begin{array}
[c]{cc}I & \\
\tilde{E} & I
\end{array}
\right)  \left(
\begin{array}
[c]{cc}A_{\widetilde{\mathcal{I}},\widetilde{\mathcal{J}}} &
A_{\widetilde{\mathcal{I}},\{1:n\}\backslash\widetilde{\mathcal{J}}}\\
& \hat{S}\end{array}
\right)  , \label{eq:a1}\end{equation}
where $\tilde{P}$ is a permutation matrix, $\widetilde{\mathcal{I}}=\mathcal{I}\cup\widehat{\mathcal{I}}$ with
\begin{equation}
\widehat{\mathcal{I}}=(\{1:m\}\backslash\mathcal{I)}\circ\mathcal{K},
\label{eq:hati}\end{equation}
and $\Vert\tilde{E}\Vert_{\max}\leq bc^{2}+c$ when $\Vert E\Vert_{\max}\leq c$, $\Vert\hat{E}\Vert_{\max}\leq c$, and $|\widehat{\mathcal{J}}|=b$.
\end{proposition}

\begin{proof}
(\ref{eq:afact1}) and (\ref{eq:s}) lead to\begin{align}
A  &  =P\left(
\begin{array}
[c]{cc}I & \\
E & I
\end{array}
\right)  \left(
\begin{array}
[c]{cc}A_{11} & A_{12}\\
& \hat{P}\begin{pmatrix}
I & \\
\hat{E} & \hat{S}\end{pmatrix}\begin{pmatrix}
S_{22} & S_{23}\\
& I
\end{pmatrix}
\end{array}
\right) \label{eq:afact2}\\
&  =P\left(
\begin{array}
[c]{cc}I & \\
E & \hat{P}\begin{pmatrix}
I & \\
\hat{E} & \hat{S}\end{pmatrix}
\end{array}
\right)  \left(
\begin{array}
[c]{cc}A_{11} & A_{12}\\
&
\begin{pmatrix}
S_{22} & S_{23}\\
& I
\end{pmatrix}
\end{array}
\right) \nonumber\\
&  =P\left(
\begin{array}
[c]{cc}I & \\
& \hat{P}\end{array}
\right)  \left(
\begin{array}
[c]{cc}I & \\
\hat{P}^{T}E &
\begin{pmatrix}
I & \\
\hat{E} & \hat{S}\end{pmatrix}
\end{array}
\right)  \left(
\begin{array}
[c]{cc}A_{11} & A_{12}\\
&
\begin{pmatrix}
S_{22} & S_{23}\\
& I
\end{pmatrix}
\end{array}
\right) \nonumber\\
&  =\tilde{P}\left(
\begin{array}
[c]{ccc}I &  & \\
E_{1} & I & \\
E_{2} & \hat{E} & \hat{S}\end{array}
\right)  \left(
\begin{array}
[c]{ccc}A_{11} & \hat{A}_{12} & \hat{A}_{13}\\
& S_{22} & S_{23}\\
&  & I
\end{array}
\right)  ,\nonumber
\end{align}
where $\tilde{P}=P\begin{pmatrix}
I & \\
& \hat{P}\end{pmatrix}
$, $\hat{P}^{T}E$ is partitioned as $\begin{pmatrix}
E_{1}\\
E_{2}\end{pmatrix}
$ conformably with $\begin{pmatrix}
I\\
\hat{E}\end{pmatrix}
$, and $A_{12}$ is partitioned as $\left(
\begin{array}
[c]{cc}\hat{A}_{12} & \hat{A}_{13}\end{array}
\right)  $ with $\hat{A}_{12}=A_{\mathcal{I},\widehat{\mathcal{J}}}$.

We can now factorize the second factor on the far right-hand side of
(\ref{eq:afact2}) as\[
\left(
\begin{array}
[c]{ccc}I &  & \\
E_{1} & I & \\
E_{2} & \hat{E} & \hat{S}\end{array}
\right)  =\left(
\begin{array}
[c]{ccc}I &  & \\
& I & \\
\bar{E} & \hat{E} & \hat{S}\end{array}
\right)  \left(
\begin{array}
[c]{ccc}I &  & \\
E_{1} & I & \\
&  & I
\end{array}
\right)  ,
\]
where\begin{equation}
\bar{E}=E_{2}-\hat{E}E_{1}. \label{eq:ehat}\end{equation}
Then, $A$ may be written as\begin{align}
A  &  =\tilde{P}\left(
\begin{array}
[c]{ccc}I &  & \\
& I & \\
\bar{E} & \hat{E} & \hat{S}\end{array}
\right)  \left(
\begin{array}
[c]{ccc}I &  & \\
E_{1} & I & \\
&  & I
\end{array}
\right)  \left(
\begin{array}
[c]{ccc}A_{11} & \hat{A}_{12} & \hat{A}_{13}\\
& S_{22} & S_{23}\\
&  & I
\end{array}
\right) \label{eq:apart3}\\
&  =\tilde{P}\left(
\begin{array}
[c]{ccc}I &  & \\
& I & \\
\bar{E} & \hat{E} & I
\end{array}
\right)  \left(
\begin{array}
[c]{ccc}A_{11} & \hat{A}_{12} & \hat{A}_{13}\\
\hat{A}_{21} & \hat{A}_{22} & \hat{A}_{23}\\
&  & \hat{S}\end{array}
\right) \nonumber
\end{align}
where\[\begin{pmatrix}
\hat{A}_{21} & \hat{A}_{22} & \hat{A}_{23}\end{pmatrix}
=\begin{pmatrix}
0 & S_{22} & S_{23}\end{pmatrix}
+E_{1}\begin{pmatrix}
A_{11} & \hat{A}_{12} & \hat{A}_{13}\end{pmatrix}
.
\]
The block $\begin{pmatrix}
\hat{A}_{21} & \hat{A}_{22} & \hat{A}_{23}\end{pmatrix}
$ essentially corresponds to the rows of $A$ with index set
$\widehat{\mathcal{I}}$ in (\ref{eq:hati}). This is because of the special
form of the second factor on the right-hand side of (\ref{eq:apart3}). Then
get (\ref{eq:a1}) by letting\[
\tilde{E}=\begin{pmatrix}
\bar{E} & \hat{E}\end{pmatrix}
.
\]

Now since $\Vert E\Vert_{\max}\leq c$, $\Vert\hat{E}\Vert_{\max}\leq c$, and
$\hat{E}$ has column size $b$, we have $\Vert\bar{E}\Vert_{\max}\leq bc^{2}+c$. Accordingly, $\Vert\tilde{E}\Vert_{\max}\leq bc^{2}+c$.
\end{proof}

This proposition shows that we can get a factorization (\ref{eq:a1}) similar
to (\ref{eq:afact1}), but with the expanded row skeleton
$A_{\widetilde{\mathcal{I}},:}$. Accordingly, we may then obtain a new
approximation to $A$ similar to (\ref{eq:srrqr}):\begin{equation}
A\approx\tilde{U}A_{\widetilde{\mathcal{I}},:}\quad\text{with\quad}\tilde
{U}=\tilde{P}\left(
\begin{array}
[c]{c}I\\
\tilde{E}\end{array}
\right)  , \label{eq:approx1}\end{equation}

To support the reliability of such an approximation, we can use the following
way. As mentioned in Remark \ref{rem:srrlu}, if (\ref{eq:afact1}) is assumed
to be obtained by a strong rank-revealing factorization, then we would have
nice singular value bounds in (\ref{eq:svsrrlu}). Now, if we assume that is
the case and (\ref{eq:s}) is also obtained by a strong rank-revealing
factorization, then we would like to show (\ref{eq:a1}) from the subset update
would also satisfy some nice singular value bounds. For this purpose, we need the following lemma.

\begin{lemma}
\label{lem:mu}If (\ref{eq:afact1}) is assumed to satisfy (\ref{eq:svsrrlu})
with $k=|\mathcal{I}|$ the size of $A_{11}$, and (\ref{eq:s}) is assumed to
satisfy, for $1\leq i\leq b$, $1\leq j\leq\min\{m,n\}-k-b$,\begin{equation}
\sigma_{i}(S_{22})\geq\frac{\sigma_{i}(S)}{\hat{f}},\quad\sigma_{j}(\hat
{S})\leq\sigma_{b+j}(S)\hat{f}, \label{eq:ssrrlu}\end{equation}
where $\hat{f}\geq1$, then $\mu=\frac{\sigma_{k}(A_{11})}{\sigma_{1}(S_{22})}$
satisfies\[
\frac{1}{f^{2}}\leq\mu\leq s\hat{f},
\]
where $s=\frac{\sigma_{k}(A)}{\sigma_{k+1}(A)}$.
\end{lemma}

\begin{proof}
By (\ref{eq:svsrrlu}) and the interlacing property of singular values,
\[
\sigma_{k}(A_{11})\geq\frac{\sigma_{k}(A)}{f},\quad\sigma_{1}(S_{22})\leq\sigma_{1}(S)\leq\sigma_{k+1}(A)f,
\]
which yield\[
\mu\geq\frac{\sigma_{k}(A)/f}{\sigma_{k+1}(A)f}\geq\frac{1}{f^{2}}.
\]
Similarly, by the interlacing property of singular values and (\ref{eq:ssrrlu}),
\[
\sigma_{k}(A_{11})\leq\sigma_{k}(A),\quad\sigma_{1}(S_{22})\geq\frac
{\sigma_{1}(S)}{\hat{f}}\geq\frac{\sigma_{k+1}(A)}{\hat{f}},
\]
where the result $\sigma_{1}(S)\geq\sigma_{k+1}(A)$ directly follows from Weyl's inequality or \cite[Theorem 3.3.16]{hor91}:
\begin{align*}
\sigma_{k+1}(A)  &  =\sigma_{k+1}\left(  P\left(
\begin{array}
[c]{c}I\\
E
\end{array}
\right)  \left(
\begin{array}
[c]{cc}A_{11} & A_{12}\end{array}
\right)  +P\left(
\begin{array}
[c]{cc}0 & \\
& S
\end{array}
\right)  \right) \\
&  \leq\sigma_{k+1}\left(  P\left(
\begin{array}
[c]{c}I\\
E
\end{array}
\right)  \left(
\begin{array}
[c]{cc}A_{11} & A_{12}\end{array}
\right)  \right)  +\sigma_{1}\left(  P\left(
\begin{array}
[c]{cc}0 & \\
& S
\end{array}
\right)  \right)  =0+\sigma_{1}(S).
\end{align*}
Then\[
\mu\leq\frac{\sigma_{k}(A)}{\sigma_{k+1}(A)/\hat{f}}=s\hat{f}.
\]

\end{proof}

As a quick note, here $s=\frac{\sigma_{k}(A)}{\sigma_{k+1}(A)}$ reflects the
gap between $\sigma_{k}(A)$ and $\sigma_{k+1}(A)$. Since we seek to expand the
index sets $\mathcal{I}$ and $\mathcal{J}$ (and $k$ hasn't yet reached the
target numerical rank $r$), it is reasonable to regard $s$ as a modest magnitude. Now we are ready to show the singular value bounds.

\begin{proposition}
With the assumptions and notation in Lemma \ref{lem:mu}, (\ref{eq:a1})
satisfies, for $1\leq i\leq k+b$, $1\leq j\leq\min\{m,n\}-k-b$,
\begin{equation}
\sigma_{i}(A_{\widetilde{\mathcal{I}},\widetilde{\mathcal{J}}})\geq
\frac{\sigma_{i}(A)}{\tilde{f}},\quad\sigma_{j}(\hat{S})\leq\sigma
_{k+b+j}(A)\tilde{f}, \label{eq:srrlu1}\end{equation}
where $\tilde{f}=(1+s\hat{f}+s\hat{f}b^{2}c^{2})f^{2}\hat{f}$.
\end{proposition}

\begin{proof}
According to (\ref{eq:apart3}), $A_{\widetilde{\mathcal{I}},\widetilde{\mathcal{J}}}=\begin{pmatrix}
I & \\
E_{1} & I
\end{pmatrix}\begin{pmatrix}
A_{11} & \hat{A}_{12}\\
& S_{22}\end{pmatrix}
$. With a strategy like in \cite{srrqr,mir03}, rewrite
$A_{\widetilde{\mathcal{I}},\widetilde{\mathcal{J}}}$ as\[
A_{\widetilde{\mathcal{I}},\widetilde{\mathcal{J}}}=\begin{pmatrix}
I & \\
E_{1} & \frac{1}{\sqrt{\mu}}I
\end{pmatrix}\begin{pmatrix}
A_{11} & \\
& \mu S_{22}\end{pmatrix}\begin{pmatrix}
I & A_{11}^{-1}\hat{A}_{12}\\
& \frac{1}{\sqrt{\mu}}I
\end{pmatrix}
,
\]
where $\mu=\frac{\sigma_{k}(A_{11})}{\sigma_{1}(S_{22})}$. Then
\[\begin{pmatrix}
A_{11} & \\
& \mu S_{22}\end{pmatrix}
=\begin{pmatrix}
I & \\
-\sqrt{\mu}E_{1} & \sqrt{\mu}I
\end{pmatrix}
A_{\widetilde{\mathcal{I}},\widetilde{\mathcal{J}}}\begin{pmatrix}
I & -\sqrt{\mu}A_{11}^{-1}\hat{A}_{12}\\
& \sqrt{\mu}I
\end{pmatrix}
\]
By \cite[Theorem 3.3.16]{hor91},\begin{align*}
\sigma_{i}\left(
\begin{pmatrix}
A_{11} & \\
& \mu S_{22}\end{pmatrix}
\right)   &  \leq\sigma_{i}(A_{\widetilde{\mathcal{I}},\widetilde{\mathcal{J}}})\left\Vert
\begin{pmatrix}
I & \\
-\sqrt{\mu}E_{1} & \sqrt{\mu}I
\end{pmatrix}
\right\Vert _{2}\left\Vert
\begin{pmatrix}
I & -\sqrt{\mu}A_{11}^{-1}\hat{A}_{12}\\
& \sqrt{\mu}I
\end{pmatrix}
\right\Vert _{2}\\
&  \leq\sigma_{i}(A_{\widetilde{\mathcal{I}},\widetilde{\mathcal{J}}})\sqrt{1+\mu+\mu\Vert E_{1}\Vert_{2}^{2}}\sqrt{1+\mu+\mu\Vert A_{11}^{-1}\hat{A}_{12}\Vert_{2}^{2}}\\
&  \leq\sigma_{i}(A_{\widetilde{\mathcal{I}},\widetilde{\mathcal{J}}})(1+s\hat{f}+s\hat{f}b^{2}c^{2}),
\end{align*}
where the last inequality is from Lemma \ref{lem:mu} and the fact that $E_{1}$
and $A_{11}^{-1}\hat{A}_{12}$ are $b\times b$ matrices with entrywise
magnitudes bounded by $c$. Thus,\begin{equation}
\sigma_{i}(A_{\widetilde{\mathcal{I}},\widetilde{\mathcal{J}}})\geq\frac
{1}{1+s\hat{f}+s\hat{f}b^{2}c^{2}}\sigma_{i}\left(  \left(
\begin{array}
[c]{cc}A_{11} & \\
& \mu S_{22}\end{array}
\right)  \right)  . \label{eq:sviaii}\end{equation}

Since $\sigma_{k}(A_{11})=\sigma_{1}(\mu S_{22})$, we get
\begin{align*}
\sigma_{i}\left(
\begin{pmatrix}
A_{11} & \\
& \mu S_{22}\end{pmatrix}
\right)   &  =\sigma_{i}(A_{11}),\quad1\leq i\leq k,\\
\sigma_{k+i}\left(
\begin{pmatrix}
A_{11} & \\
& \mu S_{22}\end{pmatrix}
\right)   &  =\sigma_{i}(\mu S_{22}),\quad1\leq i\leq b.
\end{align*}
By (\ref{eq:ssrrlu}) and Lemma \ref{lem:mu},\begin{equation}
\sigma_{i}(\mu S_{22})\geq\mu\frac{\sigma_{i}(S)}{\hat{f}}\geq\frac{1}{f^{2}\hat{f}}\sigma_{i}(S)\geq\frac{1}{f^{2}\hat{f}}\sigma_{k+i}(A),\quad1\leq i\leq b, \label{eq:svs22}\end{equation}
where the result $\sigma_{i}(S)\geq\sigma_{k+i}(A)$ again follows from Weyl's
inequality or \cite[Theorem 3.3.16]{hor91}. Putting (\ref{eq:svs22}) and the
first inequality in (\ref{eq:svsrrlu}) into (\ref{eq:sviaii}) to get\[
\sigma_{i}(A_{\widetilde{\mathcal{I}},\widetilde{\mathcal{J}}})\geq\frac
{1}{1+s\hat{f}+s\hat{f}b^{2}c^{2}}\frac{1}{f^{2}\hat{f}}\sigma_{i}(A),\quad1\leq i\leq k+b.
\]

Finally, (\ref{eq:ssrrlu}) and (\ref{eq:svsrrlu}) yield
\[
\sigma_{j}(\hat{S})\leq\sigma_{b+j}(S)\hat{f}\leq\sigma_{k+b+j}(A)f\hat
{f},\quad1\leq j\leq\min\{m,n\}-k-b.
\]
Then taking $\tilde{f}=(1+s\hat{f}+s\hat{f}b^{2}c^{2})f^{2}\hat{f}$ to get
(\ref{eq:srrlu1}).
\end{proof}

This proposition indicates that, if (\ref{eq:svsrrlu}) and (\ref{eq:ssrrlu})
are assumed to result from strong rank-revealing factorizations, then
(\ref{eq:a1}) as produced by the subset update method would also enjoy nice
singular value properties like in a strong rank-revealing factorization. This
supports the effectiveness of performing the subset update. Although here we
obtain (\ref{eq:svsrrlu}) and (\ref{eq:ssrrlu}) through the much more economic SRR\ factorizations coupled
with the Nystr\"{o}m method, it would be natural
to use subset updates to quickly get the expanded index set
$\widetilde{\mathcal{I}}$ (from the original index set $\mathcal{I}$). The
SRR\ factorizations are only applied to blocks with column sizes $b$ instead
of $ib$ in step $i$.

In a nutshell, the subset update process starts from a row skeleton
$A_{\mathcal{I},:}$, samples the Schur complement $S$, and produces an
expanded row skeleton $A_{\widetilde{\mathcal{I}},:}$ and the basis matrix
$\tilde{U}$ in (\ref{eq:approx1}). The process is outlined in Algorithm
\ref{alg:set-upd}.

\begin{algorithm}
[!h]\caption{Subset update}\label{alg:set-upd}

\begin{algorithmic}
[1]\Procedure{[$\widetilde{\mathcal{I}},\tilde{E}$]$=$ \sf Set-Upd}{$A$, $\mathcal{I}$, $E$, $\widehat{\mathcal{J}}$}

\State Get $\mathcal{L}$ as in (\ref{eq:jtilde})

\State$S_{:,\mathcal{L}}\leftarrow A_{\{1:m\}\backslash\mathcal{I},\widehat{\mathcal{J}}}-EA_{\mathcal{I},\widehat{\mathcal{J}}}$\Comment{Sampling the Schur complement}

\State$[\mathcal{K},\hat{E}]\leftarrow$ \textsf{SRR}$(S_{:,\mathcal{L}})$

\State$\left(
\begin{array}
[c]{c}E_{1}\\
E_{2}\end{array}
\right)  \leftarrow\left(
\begin{array}
[c]{c}E_{\mathcal{K},:}\\
E_{\{1:m-k\}\backslash\mathcal{K},:}\end{array}
\right)  $\Comment{$m-k$: row size of $S$ and $E$}

\State$\bar{E}\leftarrow E_{2}-\hat{E}E_{1}$

\State$\mathcal{\widehat{I}}\leftarrow(\{1:m\}\backslash\mathcal{I)}\circ\mathcal{K}$

\State$\widetilde{\mathcal{I}}\mathcal{\leftarrow I}\cup\mathcal{\widehat{I}}$,\quad$\tilde{E}\leftarrow\left(
\begin{array}
[c]{cc}\bar{E} & \hat{E}\end{array}
\right)  $

\EndProcedure

\end{algorithmic}
\end{algorithm}

Such a subset update strategy can also be applied to expand the column index
set $\mathcal{J}$. That is, when $\mathcal{I}$ is expanded into $\mathcal{I}\cup\widehat{\mathcal{I}}$, we can apply the strategy above with $\mathcal{J}$
replaced by $\mathcal{I}$, $\widehat{\mathcal{J}}$ replaced by
$\widehat{\mathcal{I}},$ and relevant columns replaced by rows.

We then incorporate the subset update strategy into the basic HAN\ scheme.
There are two ways to do so with different performance (see Algorithm
\ref{alg:han-update}).

\begin{itemize}
\item \textsf{HAN-U}: This is an HAN scheme with fast\emph{ }updates for both
the row subsets and the column subsets. Thus, both the index sets
$\mathcal{I}$ and $\mathcal{J}$ are expanded through updates. In this scheme,
$|\mathcal{I}|$ and $|\mathcal{J}|$ are each advanced by stepsize $b$ in every
iteration step.

\item \textsf{HAN-A}: This is an HAN scheme with aggressive updates where
only, say, the column index set $\mathcal{J}$ is updated. The row index set
$\mathcal{I}$ is still updated via the usual SRR pivoting applied to
$A_{:,\widetilde{\mathcal{J}}}$ (line \ref{line:hana} of Algorithm
\ref{alg:han-update}). This scheme potentially expands the index sets much more aggressively. The reason is as follows.
The SRR factorization of
$A_{:,\widetilde{\mathcal{J}}}$ may update $\mathcal{I}$ to a very different
set and the set difference $\widehat{\mathcal{I}}$ (line \ref{line:idiff} of
Algorithm \ref{alg:han-update}) may have size comparable to $|\mathcal{J}|$.
Then, the column subset update is applied based on $A_{\widehat{\mathcal{I}},:}$ as in line \ref{line:jupd} of Algorithm \ref{alg:han-update} and can
potentially increase the size of $\mathcal{J}$ by $|\widehat{\mathcal{I}}|$.
\end{itemize}

\begin{algorithm}
[!h]\caption{HAN scheme with fast or aggressive subset update}\label{alg:han-update}

\begin{algorithmic}
[1]\Procedure{$\tilde{L}=$ {\sf HAN-U} or {\sf HAN-A}}{$A$, $b$, $\tau$ (or $r$)}

\State$\vdots$\Comment{Keeping lines \ref{line:empty}--\ref{line:ibar} \& replacing lines \ref{line:ie}--\ref{line:jf} of Algorithm \ref{alg:han-b} by the following}

\If{$i=1$ or {\sf HAN-A}}\Comment{Row pivoting in the initial step or in {\sf HAN-A}}

\State\label{line:hana}$[\mathcal{I},E]\leftarrow$ \textsf{SRR}$(A_{:,\widetilde{\mathcal{J}}})$

\Else\Comment{Row subset update in {\sf HAN-U}}

\State\label{line:updi}$[\mathcal{I},E]\leftarrow$ \textsf{Set-Upd}$(A,\mathcal{I},E,\widehat{\mathcal{J}})$

\EndIf

\State\label{line:idiff}$\widehat{\mathcal{I}}\leftarrow\mathcal{I}\backslash\overline{\mathcal{I}}$\Comment{Removal of indices in the previous $\mathcal{I}$ from the new one}

\If{$i=1$}\Comment{Column pivoting in the initial step}

\State$[\mathcal{J},F]\leftarrow$ \textsf{SRR}$(A_{\mathcal{I},:}^{T})$

\Else\Comment{Column subset update}

\State\label{line:jupd}$[\mathcal{J},F]\leftarrow$ \textsf{Set-Upd}$(A^{T},\mathcal{J},F,\widehat{\mathcal{I}})$

\EndIf

\State$\vdots$\Comment{Keeping the remaining lines of Algorithm \ref{alg:han-b}}

\EndProcedure

\end{algorithmic}
\end{algorithm}

If Algorithm \ref{alg:set-upd} is applied at the $i$th iteration of Algorithm \ref{alg:han-update} as in line \ref{line:updi}, the main costs are as follows.

\begin{itemize}
\item The formation of $S_{:,\mathcal{L}}$ costs\[
O\left(  (m-(i-1)b)(i-1)b^{2}\right)  +O\left(  (m-(i-1)b)b\right)  .
\]

\item The SRR\ factorization of $S_{:,\mathcal{L}}$ in (\ref{eq:sskel}) costs
\[
O\left(  b^{2}(m-(i-1)b)\right)  .
\]

\item The computation of (\ref{eq:ehat}) costs\[
O\left(  (m-ib)(i-1)b^{2}\right)  +O\left(  (m-ib)(i-1)b\right)  .
\]

\end{itemize}

These costs add up to $O(ib(2bm+m+2b^{2}))$, where some low-order terms are dropped and $b$ is assumed to be
a small fixed stepsize. The \textsf{HAN-U} scheme applies Algorithm \ref{alg:set-upd} to both the row and the column
subset updates. Accordingly, with $N\approx\frac{r}{b}$ iterations, the total
cost of the \textsf{HAN-U} scheme is
\[
\xi_{\mathsf{HAN-U}}=\sum_{i=1}^{N}O(ib(2b(m+n)+m+n+4b^{2}))=O\left(
r^{2}(m+n)\right)  ,
\]
which is a significant reduction over the cost in (\ref{eq:cost1}).

The cost of the \textsf{HAN-A} scheme depends on how many iteration steps are involved and on how aggressive the index sets advance. In the most aggressive
case, suppose at each step the updated index set $\mathcal{I}$ (or
$\mathcal{J}$) doubles the size from the previous step, then it only needs
$\tilde{N}\approx\log_{2}\frac{r}{b}$ steps. Accordingly, the cost is
\[
\xi_{\mathsf{HAN-A}}=\sum_{i=1}^{\tilde{N}}O\left(  (2^{i-1}b)^{2}(m+n)\right)  =O(r^{2}(m+n)),
\]
which is comparable to $\xi_{\mathsf{HAN-U}}$. Moreover, in such a case,
\textsf{HAN-A} would only need about $b\log_{2}\frac{r}{b}$ column samples
instead of about $r$ samples, which makes it possible to find a low-rank
approximation with a total sample size much smaller than $r$. This has been
observed frequently in numerical tests (see Section \ref{sec:tests}).

\subsection{Stopping criteria and adaptive accuracy control\label{sub:acc}}

The HAN\ schemes output both $\mathcal{I}$ and $\mathcal{J}$ so we may use
$UA_{\mathcal{I},:}$, $A_{:,\mathcal{J}}V^{T}$, or $UA_{\mathcal{I},\mathcal{J}}V^{T}$ as the output low-rank approximation, where $V$ and $U$
look like those in (\ref{eq:v}) and (\ref{eq:srrqr}), respectively. Based on
the differences of the schemes, we use the following choice which works well
in practice:\begin{equation}
\tilde{A}=\left\{
\begin{array}
[c]{cc}A_{:,\mathcal{J}}V^{T}, & \text{\textsf{HAN-B} or \textsf{HAN-U},}\\
UA_{\mathcal{I},:}, & \text{\textsf{HAN-A}.}\end{array}
\right.  \label{eq:atilde}\end{equation}
The reason is as follows. $A_{:,\mathcal{J}}V^{T}$ is the output from the end
of the iteration and is generally a good choice. On the other hand, since
\textsf{HAN-A} obtains $U$ from a full strong rank-revealing factorization
step which potentially gives better accuracy, so $UA_{\mathcal{I},:}$ is used
for \textsf{HAN-A}.

The following stopping criteria may be used in the iterations.

\begin{itemize}
\item The iterations stop when a maximum sample size or a target numerical
rank is reached. The numerical rank is reflected by $|\mathcal{I}|$ or
$|\mathcal{J}|$, depending on the output low-rank form in (\ref{eq:atilde}).

\item In \textsf{HAN-B} and \textsf{HAN-A}, the iteration stops when $\mathcal{I}$ stays the same as in the previous step.

\item Another criterion is when the approximation error is smaller than $\tau
$. It is generally expensive to directly evaluate the error. There are various
ways to estimate it. For example, in \textsf{HAN-U} and \textsf{HAN-A}, we may
use the following bound based on (\ref{eq:afact1}) and (\ref{eq:s}):
\[
\Vert A-\tilde{A}\Vert_{2}=\Vert S\Vert_{2}\approx\left\Vert
\begin{pmatrix}
I\\
\hat{E}\end{pmatrix}\begin{pmatrix}
S_{22} & S_{23}\end{pmatrix}
\right\Vert _{2}\leq\left\Vert
\begin{pmatrix}
I\\
\hat{E}\end{pmatrix}
\right\Vert _{2}\left\Vert
\begin{pmatrix}
S_{22} & S_{23}\end{pmatrix}
\right\Vert _{2}.
\]
(Note the approximations to $A$ and $S$ are obtained by randomization.) We may
also directly estimate the absolute or relative approximation errors without
the need to evaluate $\Vert A-\tilde{A}\Vert_{2}$ or $\Vert A\Vert_{2}$. In the following, we give more details.
\end{itemize}

The following lemmas suggest how to estimate the absolute and relative errors.

\begin{lemma}
\label{lem:errexp}Suppose $k=|\mathcal{J}|$ for the column index set
$\mathcal{J}$ of $A_{11}$ in (\ref{eq:a11}) and $\mathcal{L}$ is given in
(\ref{eq:jtilde}) with $\widehat{\mathcal{J}}$ from a uniform random sampling
of $\{1:n\}\backslash\mathcal{J}$ and $|\widehat{\mathcal{J}}|=b$. Let
$\theta=\frac{n-k}{b}\Vert S_{:,\mathcal{L}}\Vert_{F}^{2}$ and $\mathcal{E}=A-UA_{\mathcal{I},:}$. Then\begin{equation}
\operatorname{E}(\theta)=\Vert\mathcal{E}\Vert_{F}^{2}. \label{eq:errest1}\end{equation}
If (\ref{eq:afact1}) is further assumed to satisfy (\ref{eq:svsrrlu}), then\[
\Pr\left(  \left\vert \theta-\Vert\mathcal{E}\Vert_{F}^{2}\right\vert
\geq((n-k)f^{2})(\sigma_{k+1}(A))^{2}\right)  \leq2\exp\left(  -2b\right)  .
\]

\end{lemma}

\begin{proof}
From (\ref{eq:afact1}), $\Vert\mathcal{E}\Vert_{F}=\Vert S\Vert_{F}$. $S$ has
size $(m-k)\times(n-k)$. $S_{:,\mathcal{L}}$ essentially results from the
uniform sampling of the columns of $S$ with $\mathcal{L}$ in (\ref{eq:jtilde}). Let $C$ be the submatrix formed by the $k$ columns of the order-$(n-k)$
identity matrix corresponding to the column index set $\mathcal{L}$. Then\begin{align}
S_{:,\mathcal{L}}  &  =SC,\label{eq:c}\\
\operatorname{E}(\Vert S_{:,\mathcal{L}}\Vert_{F}^{2})  &  =\operatorname{E}(\operatorname{trace}(S_{:,\mathcal{L}}^{T}S_{:,\mathcal{L}}))=\operatorname{E}(\operatorname{trace}(C^{T}S^{T}SC))\nonumber\\
&  =\frac{b}{n-k}\operatorname{trace}(S^{T}S)=\frac{b}{n-k}\Vert S\Vert
_{F}^{2},\nonumber
\end{align}
where the equality from the first line to the second directly comes by the
definition of expectations and is a trace estimation result in \cite{avr11}.
This gives (\ref{eq:errest1}).

If (\ref{eq:afact1}) is further assumed to satisfy (\ref{eq:svsrrlu}), then
$\Vert S_{:,j}\Vert_{2}\leq\Vert S\Vert_{2}\leq f\sigma_{k+1}(A)$. The
probability result can be obtained like in \cite[Theorem 8.2]{avr11} by
writing $\Vert S_{:,\mathcal{L}}\Vert_{F}^{2}$ as the sum of $b$ squares $\Vert S_{:,j}\Vert_{2}^{2}$ and applying Hoeffding's inequality:
\begin{align*}
\Pr\left(  \left\vert \theta-\Vert\mathcal{E}\Vert_{F}^{2}\right\vert
\geq\varepsilon\right)   &  \leq2\exp\left(  -\frac{2b\varepsilon^{2}}{(n-k)^{2}\max_{j}\Vert S_{:,j}\Vert_{2}^{4}}\right) \\
&  \leq2\exp\left(  -\frac{2b\varepsilon^{2}}{(n-k)^{2}\left(  f\sigma
_{k+1}(A)\right)  ^{4}}\right)  .
\end{align*}
Setting $\varepsilon=(n-k)\left(  f\sigma_{k+1}(A)\right)  ^{2}$ to get the result.
\end{proof}

The probability result indicates that, even with small $b$, $\theta$ is a very
accurate estimator for $\Vert\mathcal{E}\Vert_{F}^{2}$ (provided that
(\ref{eq:svsrrlu}) holds). We can further consider the estimation of the
relative error.

\begin{lemma}
With the assumptions and notation in Lemma \ref{lem:errexp}, $H=\frac{n-k}{b}S_{:,\mathcal{L}}S_{:,\mathcal{L}}^{T}$ satisfies\begin{equation}
\frac{\Vert\mathcal{E}\Vert_{2}}{\Vert A\Vert_{2}}\leq\frac{\sqrt
{\Vert\operatorname{E}(H)\Vert_{2}}}{\Vert A_{11}\Vert_{2}}\leq f^{2}\frac{\sigma_{k+1}(A)}{\Vert A\Vert_{2}}. \label{eq:errest2}\end{equation}

\end{lemma}

\begin{proof}
With (\ref{eq:c}),\[
\operatorname{E}(S_{:,\mathcal{L}}S_{:,\mathcal{L}}^{T})=\operatorname{E}(SCC^{T}S^{T})=S[\operatorname{E}(CC^{T})]S^{T}=\frac{b}{n-k}SS^{T},
\]
where $\operatorname{E}(CC^{T})=\frac{b}{n-r}I$ is simply by the definition of
expectations and has been explored in, say, \cite{dri06}. This leads to\[
\sqrt{\Vert\operatorname{E}(H)\Vert_{2}}=\Vert S\Vert_{2}=\Vert\mathcal{E}\Vert_{2},
\]
which, together with $\Vert A_{11}\Vert_{2}\leq\Vert A\Vert_{2}$, yields the
first inequality in (\ref{eq:errest2}). The second inequality in
(\ref{eq:errest2}) is based on (\ref{eq:svsrrlu}):\[
\frac{\sqrt{\Vert\operatorname{E}(H)\Vert_{2}}}{\Vert A_{11}\Vert_{2}}=\frac{\Vert S\Vert_{2}}{\Vert A_{11}\Vert_{2}}\leq\frac{f\sigma_{k+1}(A)}{\Vert A\Vert_{2}/f}=f^{2}\frac{\sigma_{k+1}(A)}{\Vert A\Vert_{2}}.
\]

\end{proof}

From these lemmas, we can see that the absolute or relative errors in the
low-rank approximation may be estimated by using $S_{:,\mathcal{L}}$ and
$A_{11}$. For example, a reasonable estimator for the relative error of the
low-rank approximation $\tilde{A}$ is given by\begin{equation}
\phi=\sqrt{\frac{n-k}{b}}\frac{\Vert S_{:,\mathcal{L}}\Vert_{2}}{\Vert
A_{11}\Vert_{2}}(\approx\frac{\Vert A-\tilde{A}\Vert_{2}}{\Vert A\Vert_{2}}).
\label{eq:est}\end{equation}
This estimator can be quickly evaluated and only costs $O(b(m-k)+b^{2}+k^{2})$. The cost may be further reduced to $O(b^{2}+k^{2})$ by using $\sqrt
{\frac{n-k}{b}}\frac{\Vert S_{\mathcal{K},\mathcal{L}}\Vert_{2}}{\Vert
A_{11}\Vert_{2}}$ since $S_{\mathcal{K},\mathcal{L}}$ results from a strong
rank-revealing factorization applied to $S_{:,\mathcal{L}}$ and there is a
low-degree polynomial $g$ in $m-k$ and $b$ such that\[
\frac{\Vert S_{:,\mathcal{L}}\Vert_{2}}{g}\leq\Vert S_{\mathcal{K},\mathcal{L}}\Vert_{2}\leq\Vert S_{:,\mathcal{L}}\Vert_{2}.
\]
To enhance the reliability, we may stop the iteration if the estimators return
errors smaller than a threshold consecutively for multiple steps.

\section{Numerical tests\label{sec:tests}}

We now illustrate the performance of the HAN schemes and compare with some
other Nystr\"{o}m-based schemes. The following methods will be tested:

\begin{itemize}
\item \textsf{HAN-B}, \textsf{HAN-U}, \textsf{HAN-A}: the HAN schemes as in
Algorithms \ref{alg:han-b} and \ref{alg:han-update};

\item \textsf{Nys-B}: the traditional Nystr\"{o}m method to produce an
approximation like in (\ref{eq:cur}), where both the row index set
$\mathcal{I}$ and the column index set $\mathcal{J}$ are uniformly and randomly selected;

\item \textsf{Nys-P}: the scheme to find an approximation like in
(\ref{eq:srrqr}) but with $\mathcal{I}$ obtained by one pivoting step
(\ref{eq:aj}) applied to uniformly and randomly selected $A_{:,\mathcal{J}}$;

\item \textsf{Nys-R}: the scheme that extends \textsf{Nys-P} by applying
several steps of alternating direction refinements to improve $\mathcal{I}$
and $\mathcal{J}$ like in lines \ref{line:ie}--\ref{line:jf} of Algorithm
\ref{alg:han-b}, which corresponds to the iterative cross-approximation scheme in \cite{lua20}. (In \textsf{Nys-R}, the accuracy typically stops improving
after few steps of refinement, so we fix the number of refinement steps to be
$10$ in the tests.)

\end{itemize}

In the HAN schemes \textsf{HAN-B}, \textsf{HAN-U}, and \textsf{HAN-A}, the
stepsize $b$ in the progressive column sampling is set to be $b=5$. The
stopping criteria follow the discussions at the beginning of Section
\ref{sub:acc}. Specifically, the iteration stops if the randomized relative
error estimate in (\ref{eq:est}) is smaller than the threshold $\tau=10^{-14}$, or if the total sample size $S$ (in all progressive sampling steps) reaches
a certain maximum, or if the index refinement no longer updates the row index
set $\mathcal{I}$. Since the HAN schemes involve randomized error estimation,
it is possible for some iterations to stop earlier or later than necessary.
Also, \textsf{HAN-B} does not use the fast subset update strategy in Section \ref{sub:upd}, so an extra step is added to estimate the accuracy with (\ref{eq:est}).

The Nystr\"{o}m-based schemes \textsf{Nys-B}, \textsf{Nys-P}, and
\textsf{Nys-R} are directly applied with different given sample sizes $S$ and
do not really have a fast accuracy estimation mechanism. In the plots below
for the relative approximation errors $\frac{\Vert A-\tilde{A}\Vert_{2}}{\Vert A\Vert_{2}}$, the Nystr\"{o}m and HAN schemes are put together for comparison.
However, it is important to distinguish the meanings of the sample sizes $S$
for the two cases along the horizontal axes. For the Nystr\"{o}m schemes, each
$S$ is set directly. For the HAN schemes, each $S$ is the total sample size of
all sampling steps and is reached progressively through a sequence of steps each of stepsize $b$.

In the three Nystr\"{o}m schemes, the cardinality $|\mathcal{I}|$ will be
reported as the numerical rank. In the HAN schemes, the numerical rank will be
either $|\mathcal{I}|$ or $|\mathcal{J}|$, depending on the low-rank form in
(\ref{eq:atilde}).

Since the main applications of the HAN schemes are numerical computations, our
tests below focus on two and three dimensional problems, including some
discretized meshes and some structured matrix problems. We also include an
example related to high-dimensional data sets. The tests are done in Matlab R2019a on a cluster using two 2.60GHz cores and $32$GB of memory.

\begin{example}
\label{ex:mesh}First consider some kernel matrices generated by the evaluation
of various commonly encountered kernel functions evaluated at two
well-separated data points $\mathbf{x}$ and $\mathbf{y}$ in two and three
dimensions. $\mathbf{x}$ and $\mathbf{y}$ are taken from the following four data sets (see Figure \ref{fig:mesh}).

\begin{enumerate}
\renewcommand\labelenumi{(\alph{enumi})}

\item \texttt{Flower}: a flower shape curve, where the $\mathbf{x}$ set is
located at a corner and $|\mathbf{x}|=1018$, $|\mathbf{y}|=13965$.

\item \texttt{FEM}: a 2D\ finite element mesh extracted from the package
MESHPART \cite{meshpart}, where the $\mathbf{x}$ set is surrounded by the
points in $\mathbf{y}$ with $|\mathbf{x}|=821$, $|\mathbf{y}|=4125$. The mesh
is from an example in \cite{kercompr} that shows the usual Nystr\"{o}m method
fails to reach high accuracies for some kernel matrices even with the number
of samples near the numerical rank.

\item \texttt{Airfoil}: an unstructured 2D mesh (airfoil) from the SuiteSparse
matrix collection (http://sparse.tamu.edu), where the $\mathbf{x}$ and
$\mathbf{y}$ sets are extracted so that $\mathbf{x}$ has a roughly rectangular
shape and $|\mathbf{x}|=617$, $|\mathbf{y}|=11078$.

\item \texttt{Set3D}:\ A set of 3D data points extract from the package
DistMesh \cite{distmesh} but with the $\mathbf{y}$ points randomly perturbed
with $|\mathbf{x}|=717$, $|\mathbf{y}|=6650$.
\end{enumerate}

\begin{figure}[ptbh]
\centering\tabcolsep-2pt
\begin{tabular}
[c]{cccc}\begin{tabular}
[c]{c}\includegraphics[height=0.85in]{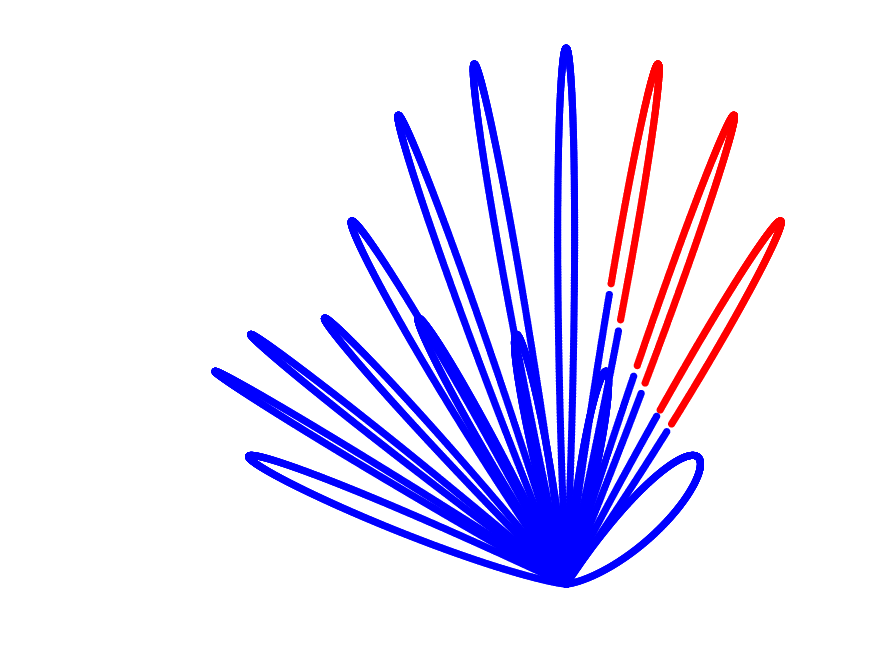}
\end{tabular}
&
\begin{tabular}
[c]{c}\includegraphics[height=0.6in]{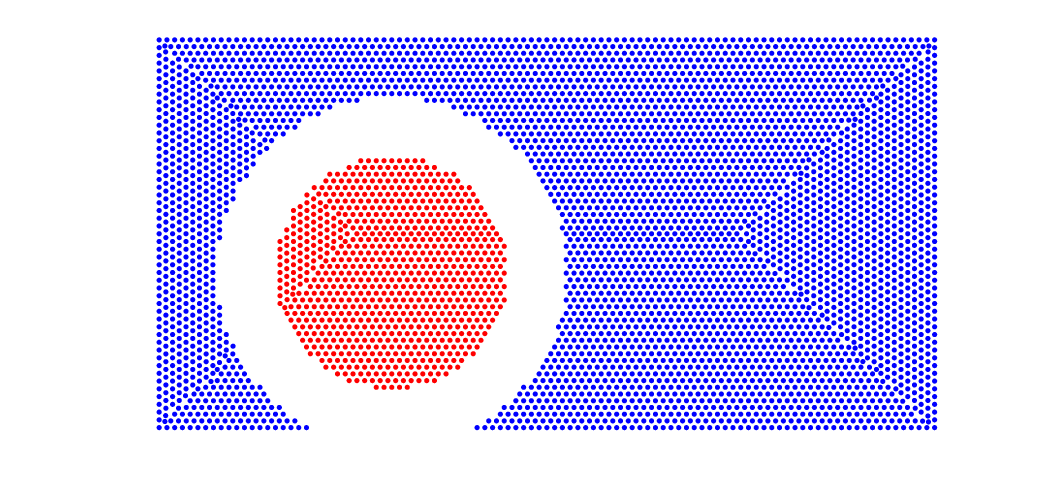}
\end{tabular}
&
\begin{tabular}
[c]{c}\includegraphics[height=1.1in]{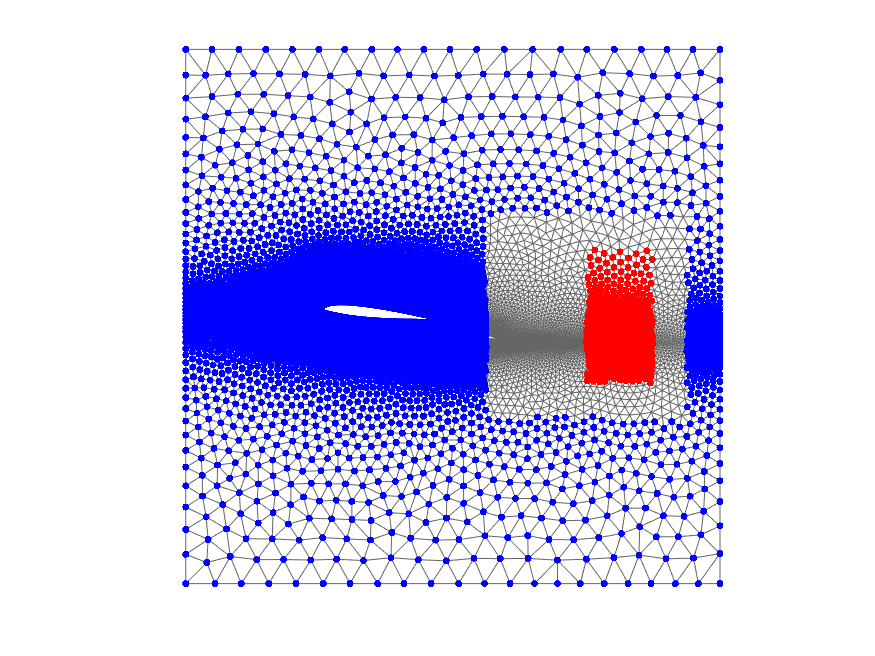}
\end{tabular}
&
\begin{tabular}
[c]{c}\includegraphics[height=1.in]{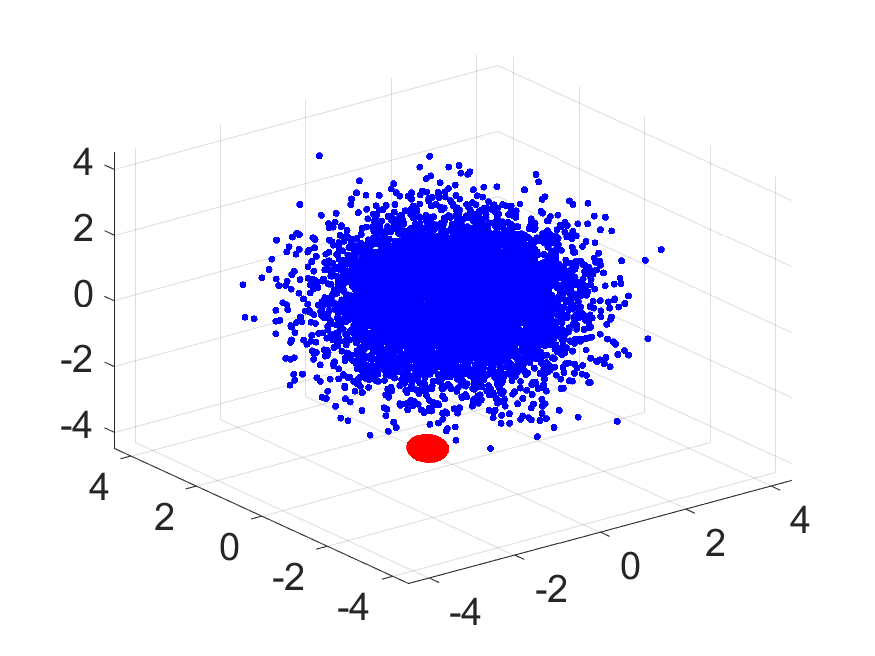}
\end{tabular}
\\
{\small (a) \texttt{Flower}} & {\small (b) \texttt{FEM}} & {\small (c)
\texttt{Airfoil}} & {\small (d) \texttt{Set3D}}\end{tabular}
\caption{Example \ref{ex:mesh}. Data sets under consideration with the
$\mathbf{x}$ and $\mathbf{y}$ sets marked in red and blue, respectively.}\label{fig:mesh}\end{figure}

The points in the data sets are nonuniformly distributed in general, except in
the case \texttt{FEM} where the points are more uniform. The data points in
two dimensions are treated as complex numbers. The setup of the $\mathbf{x}$
and $\mathbf{y}$ sets has the size of $\mathbf{x}$ just several times larger
than the target numerical rank. This is often the case in the FMM and
structured solvers where the corresponding matrix blocks are short and wide off-diagonal blocks that need to be compressed in the hierarchical
approximation of a global kernel matrix (see, e.g.,
\cite{gre87,spectral1d,mfhssrs,fasthss}). We consider several types of kernels
as follows:\begin{align*}
\kappa(x,y)=\  &  \frac{1}{x-y},\quad\frac{1}{(x-y)^{2}},\quad\frac{1}{|x-y|},\quad\sqrt{|x-y|+1},\quad\frac{1}{\sqrt{|x-y|^{2}+1}},\\
&  e^{-|x-y|},\quad e^{-\alpha|x-y|^{2}},\quad\log|x-y|,\quad\tan(x\cdot y+1),
\end{align*}
where $\alpha$ is a parameter. Such kernels are frequently used in the
FMM\ and in structured matrix computations like Toeplitz solutions \cite{toep}
and some structured eigenvalue solvers \cite{gu95,superdc,hsseig}. For data points in three dimensions, $|x-y|$ represents the distance between $x$ and $y$.
\end{example}

For each data set, we apply the methods above to the kernel matrices $A$ as in
(\ref{eq:kermat}) formed by evaluating some $\kappa(x,y)$ at $\mathbf{x}$ and
$\mathbf{y}$. Most of the kernel matrices have modest numerical ranks. The
schemes \textsf{Nys-B}, \textsf{Nys-P}, and \textsf{Nys-R} use sample sizes
$S$ up to $400$ in almost all the tests.
The HAN\ schemes use much smaller sample sizes. \textsf{HAN-B} and
\textsf{HAN-U} use sample sizes $S\leq200$ for most tests, and \textsf{HAN-A}
uses sample sizes $S\leq50$ for all the cases.

For some kernels evaluated at the set \texttt{Flower},
the relative errors $\frac{\Vert A-\tilde{A}\Vert_{2}}{\Vert A\Vert_{2}}$ in
one test run are reported in Figure \ref{fig:flowererr}. With larger $S$, the
error typically gets smaller. However, \textsf{Nys-B} is only able to reach
modest accuracies even if $S$ is quite large. (The error curve nearly
stagnate in the first row of Figure \ref{fig:flowererr} with increasing $S$.)
The accuracy gets better with \textsf{Nys-P} for some cases. \textsf{Nys-R}
can further improve the accuracy. However, they still cannot get accuracy
close to $\tau=10^{-14}$ and their error curves in the second row of Figure
\ref{fig:flowererr} get stuck around some small rank sizes insufficient to reach high accuracies.

In comparison, the HAN schemes usually yield much better accuracies,
especially with \textsf{HAN-B} and \textsf{HAN-A}. \textsf{HAN-U} is often
less accurate than \textsf{HAN-B} but is more efficient because of the fast
subset update. The most remarkable result is from \textsf{HAN-A}, which
quickly reaches accuracies around $10^{-15}$ after few sampling steps (with
small overall sample sizes). The second row of Figure \ref{fig:flowererr} also
includes the scaled singular values $\frac{\sigma_{i}(A)}{\sigma_{1}(A)}$. We
can observe that \textsf{HAN-B} and particularly \textsf{HAN-A} produce
approximation errors with decay patterns very close to that of SVD.

To further confirm the accuracies, we run each scheme 100 times and report the
results in Figure \ref{fig:flower}. In general, we observe that the HAN
schemes are more accurate, especially \textsf{HAN-B} and \textsf{HAN-A}. The
direct outcome from \textsf{HAN-U} is not accurate, but this is likely due to
the quality of the $V$ factor in (\ref{eq:atilde}). In fact, most other
schemes end the iteration with a low-rank approximation in (\ref{eq:atilde})
after one row or column pivoting step by an SRR factorization. Thus, if we
apply an additional row pivoting step to $A_{:,\mathcal{J}}$ at the end of
\textsf{HAN-U }so as to generate a new approximation $UA_{\mathcal{I},:}$ like
in (\ref{eq:atilde}), then the resulting errors of \textsf{HAN-U} (called
\emph{effective} errors in Figure \ref{fig:flower}) are close to those of
\textsf{HAN-B}.

\begin{figure}[ptbh]
\centering\tabcolsep-0.5mm
\begin{tabular}
[c]{cccc}\includegraphics[height=0.97in]{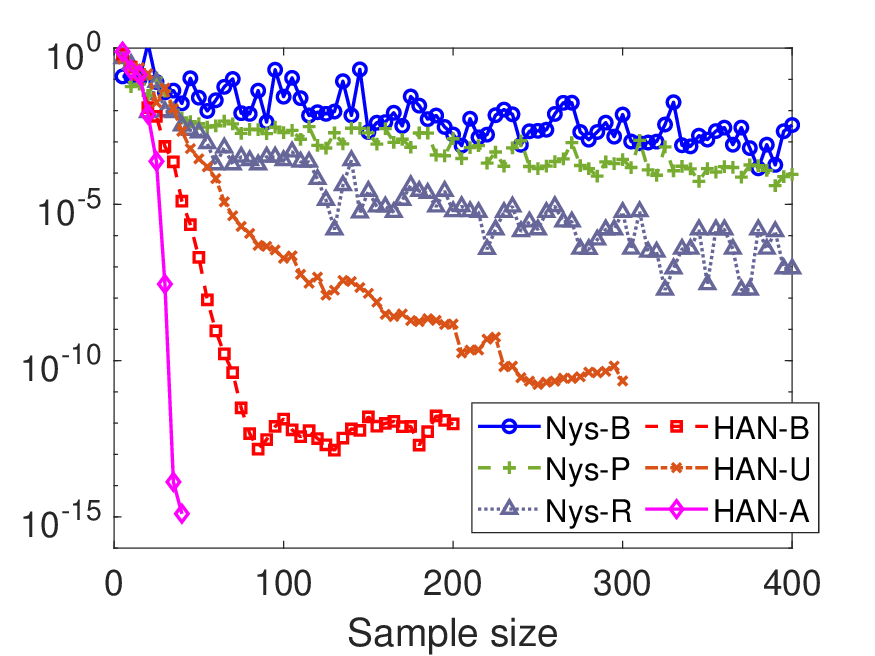} &
\includegraphics[height=0.97in]{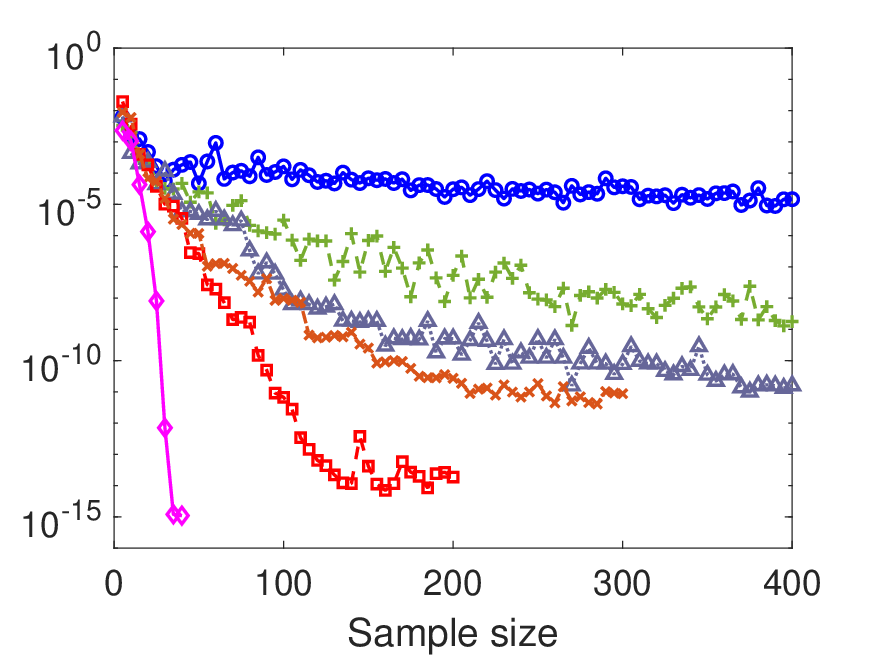} &
\includegraphics[height=0.97in]{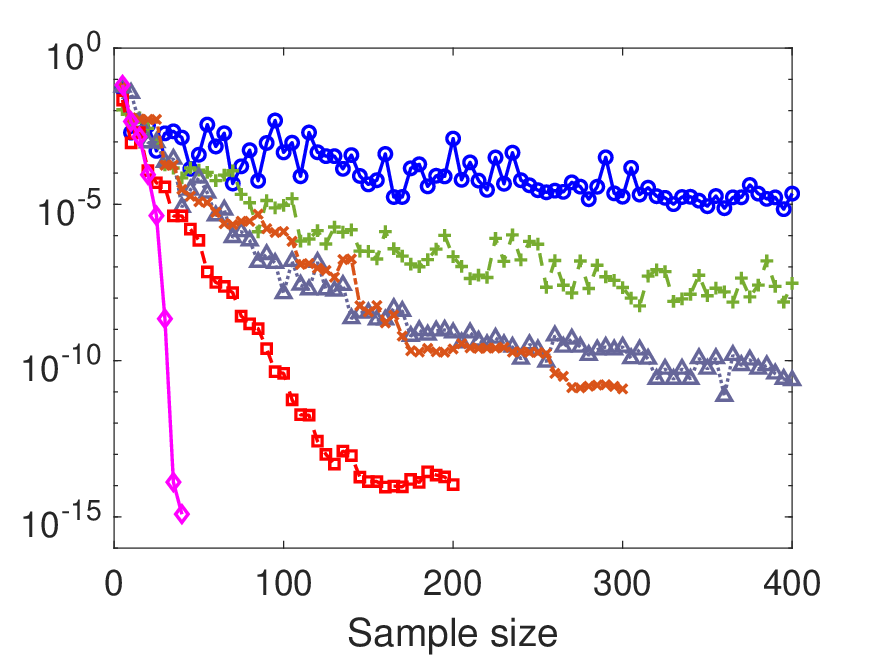} &
\includegraphics[height=0.97in]{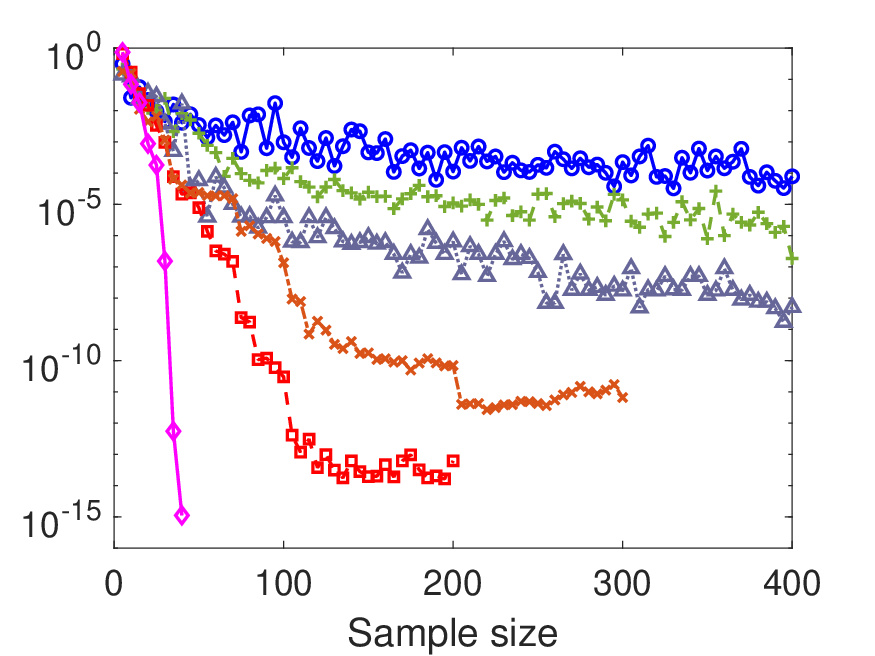}\\
\includegraphics[height=0.97in]{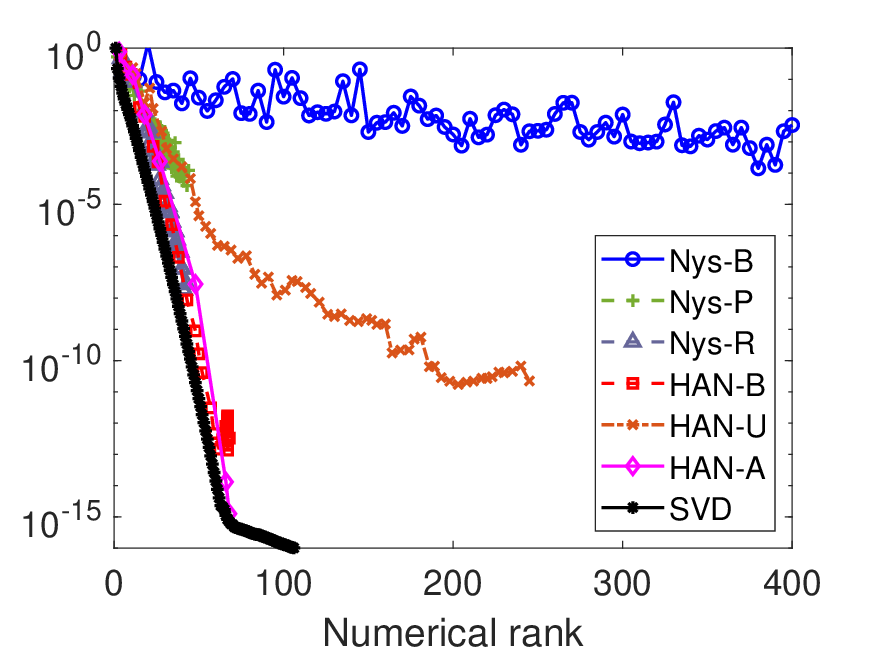} &
\includegraphics[height=0.97in]{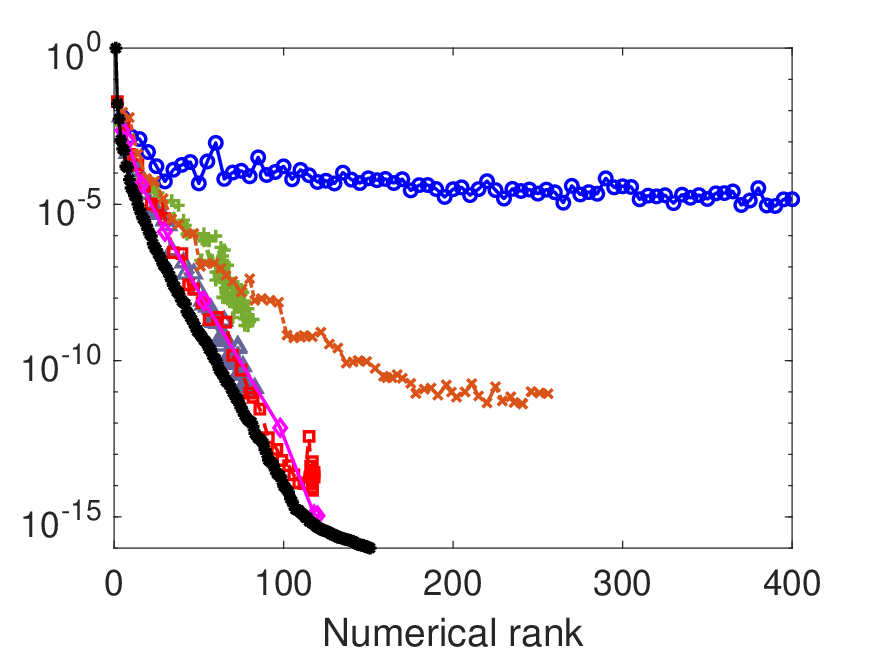} &
\includegraphics[height=0.97in]{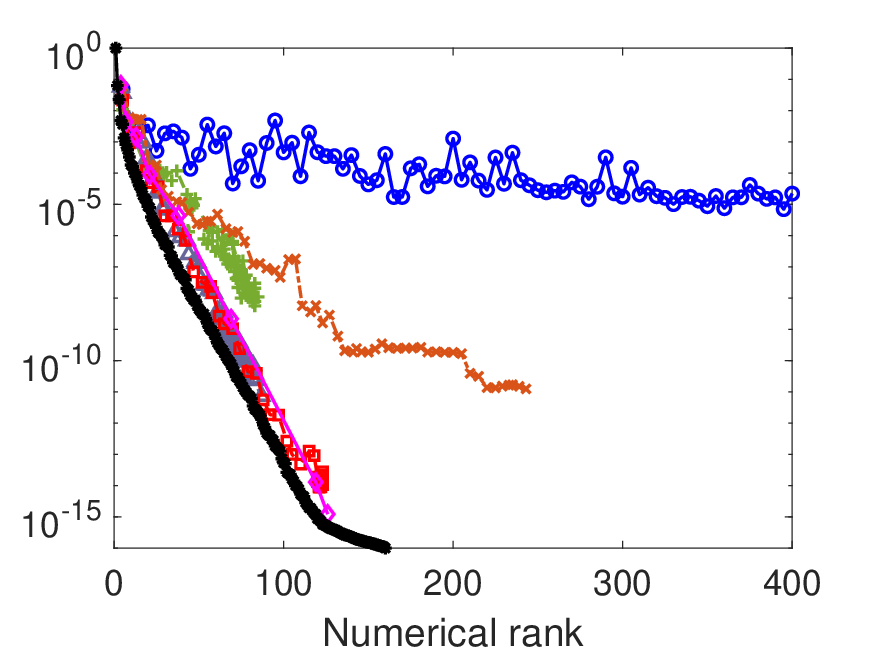} &
\includegraphics[height=0.97in]{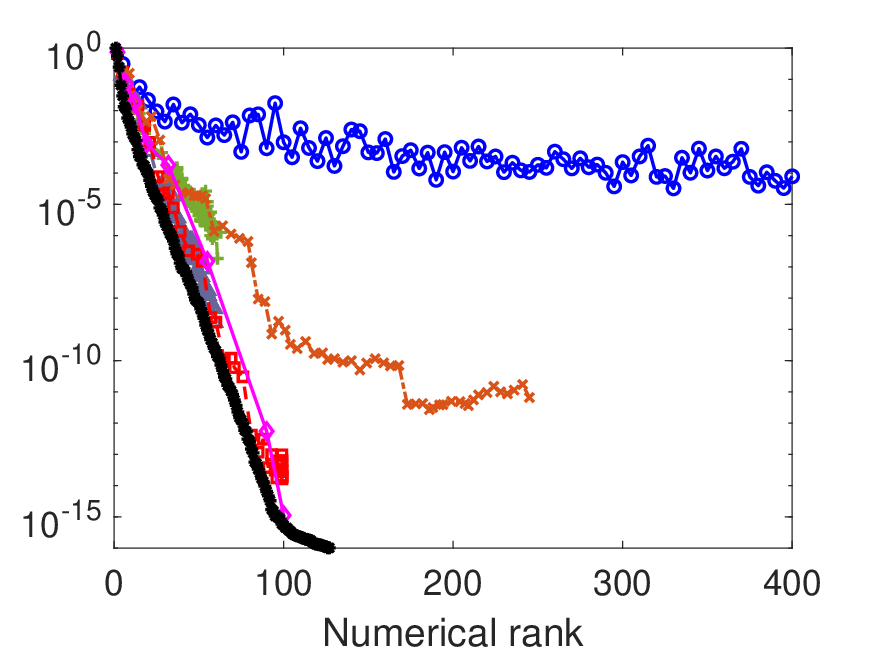}\\
{\small $\kappa(x,y)=\frac{1}{x-y}$} & {\small $\sqrt{|x-y|+1}$} &
{\small $e^{-|x-y|}$} & {\small $\log|x-y|$}\end{tabular}
\caption{Example \ref{ex:mesh} (data set \texttt{Flower}): Low-rank
approximation errors $\frac{\Vert A-\tilde{A}\Vert_{2}}{\Vert A\Vert_{2}}$ as
the sample size $S$ increases in one test, where the second row shows the
errors with respect to the resulting numerical ranks corresponding to the
first row and the SVD line shows the scaled singular values.}\label{fig:flowererr}\end{figure}

\begin{figure}[ptbh]
\centering\tabcolsep1mm
\begin{tabular}
[c]{cccc}\includegraphics[height=1in]{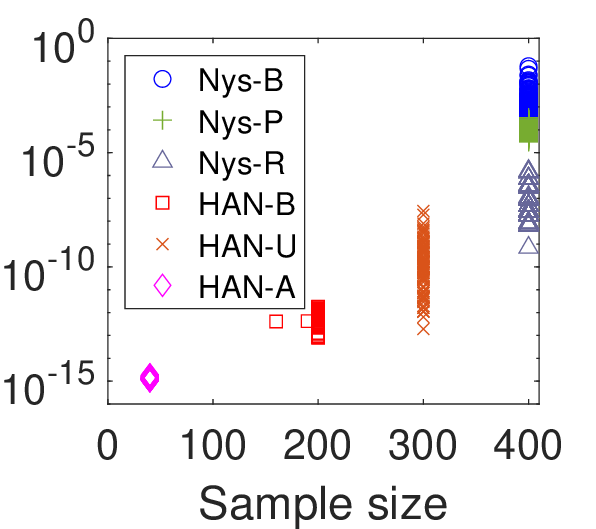} &
\includegraphics[height=1in]{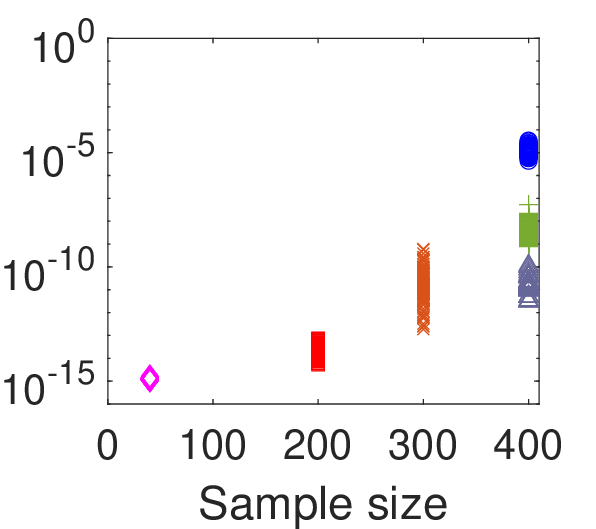} &
\includegraphics[height=1in]{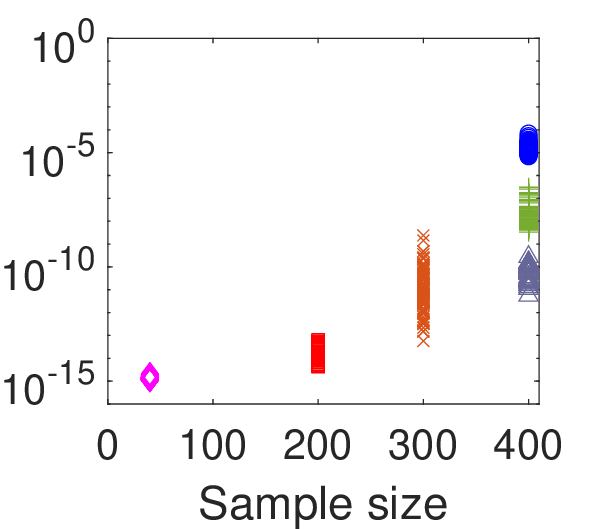} &
\includegraphics[height=1in]{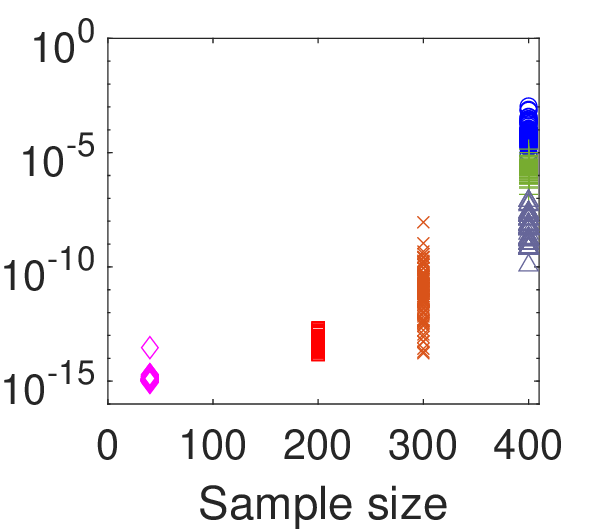}\\
\includegraphics[height=1in]{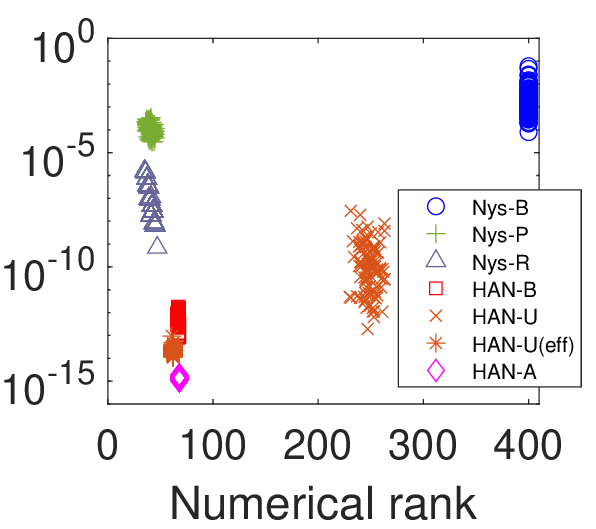} &
\includegraphics[height=1in]{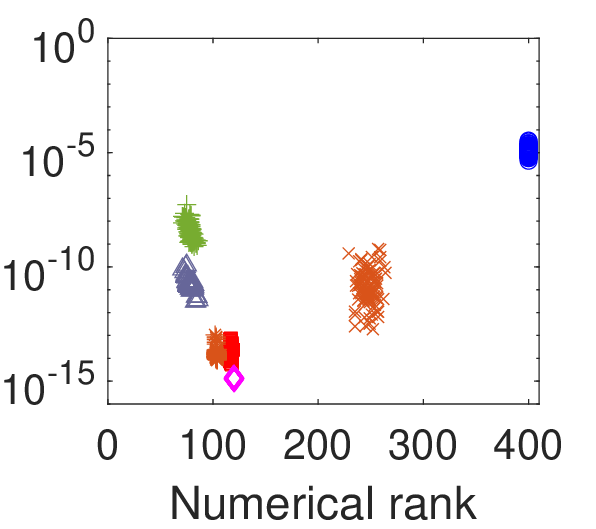} &
\includegraphics[height=1in]{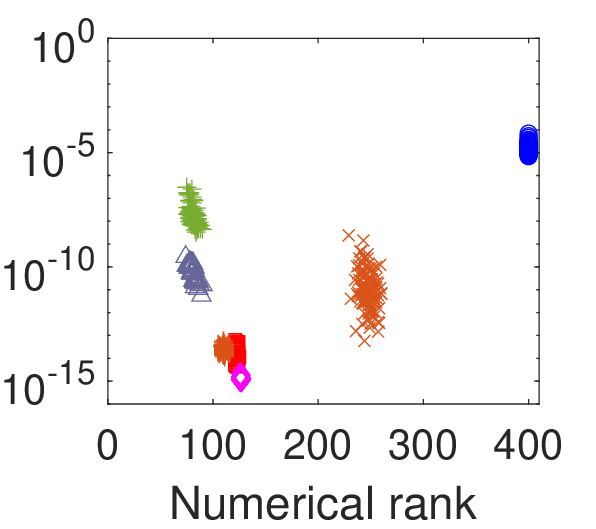} &
\includegraphics[height=1in]{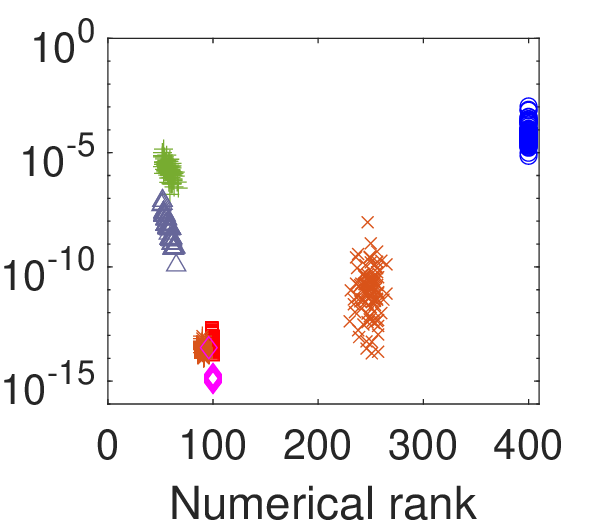}\\
{\small $\kappa(x,y)=\frac{1}{x-y}$} & {\small $\sqrt{|x-y|+1}$} &
{\small $e^{-|x-y|}$} & {\small $\log|x-y|$}\end{tabular}
\caption{Example \ref{ex:mesh} (data set \texttt{Flower}): Relative errors
from running the methods for $100$ times, where \textsf{HAN-U(eff)} is for the
effective errors of \textsf{HAN-U}.}\label{fig:flower}\end{figure}

Similarly, for the other data sets and various different kernel functions, we
have test results as given in Figures \ref{fig:meshparterr}--\ref{ex:mesh3d}.
The results can be interpreted similarly. For some cases, \textsf{Nys-B},
\textsf{Nys-P}, and even \textsf{Nys-R} may be quite inaccurate. One example
is for $\kappa(x,y)=e^{-16|x-y|^{2}}$ in Figures \ref{ex:mesh3derr} and
\ref{ex:mesh3d}, where even \textsf{Nys-R} becomes quite unreliable and
demonstrates oscillatory errors for different $S$ and different tests.

\begin{figure}[ptbh]
\centering\tabcolsep-0.5mm
\begin{tabular}
[c]{cccc}\includegraphics[height=0.97in]{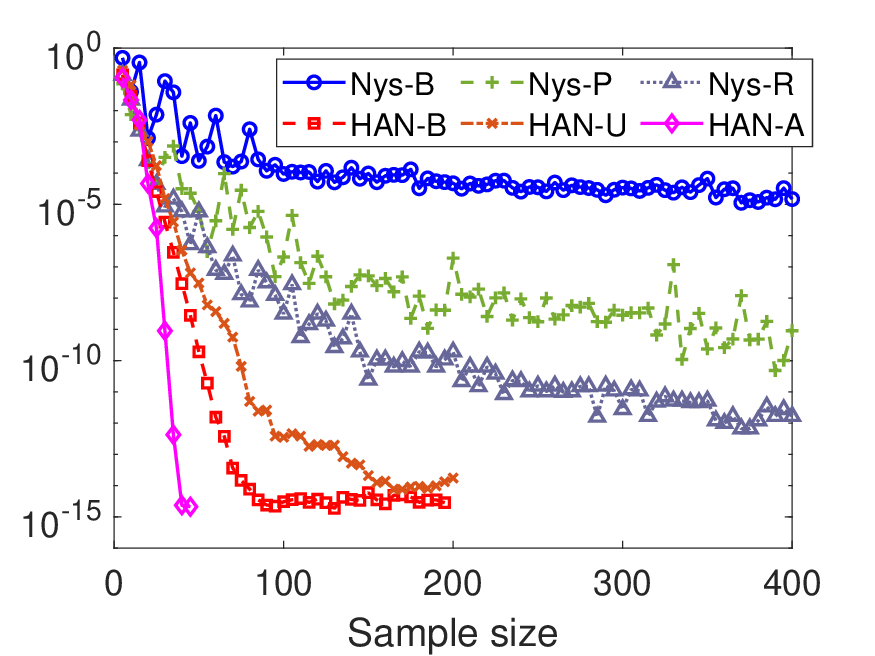} &
\includegraphics[height=0.97in]{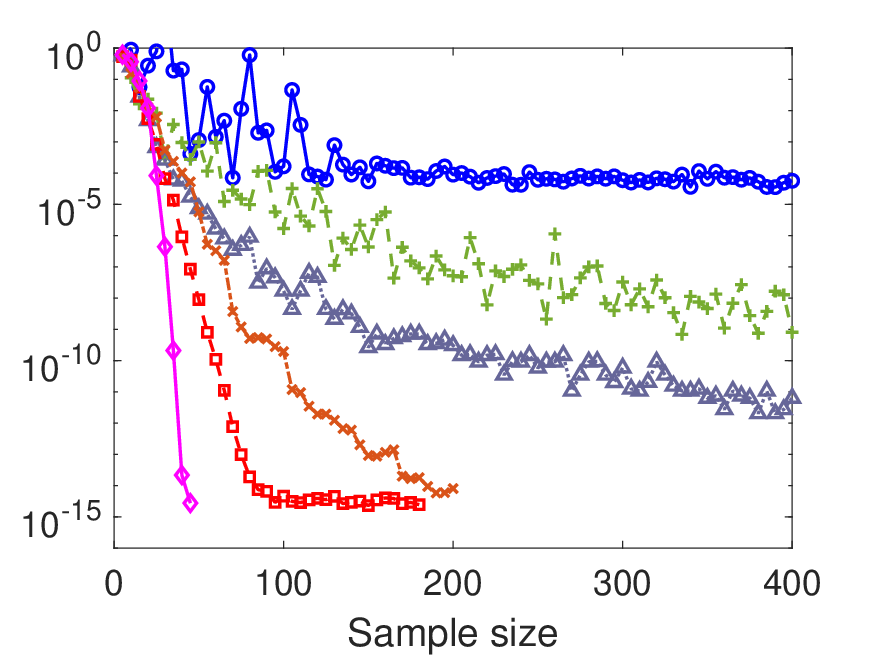} &
\includegraphics[height=0.97in]{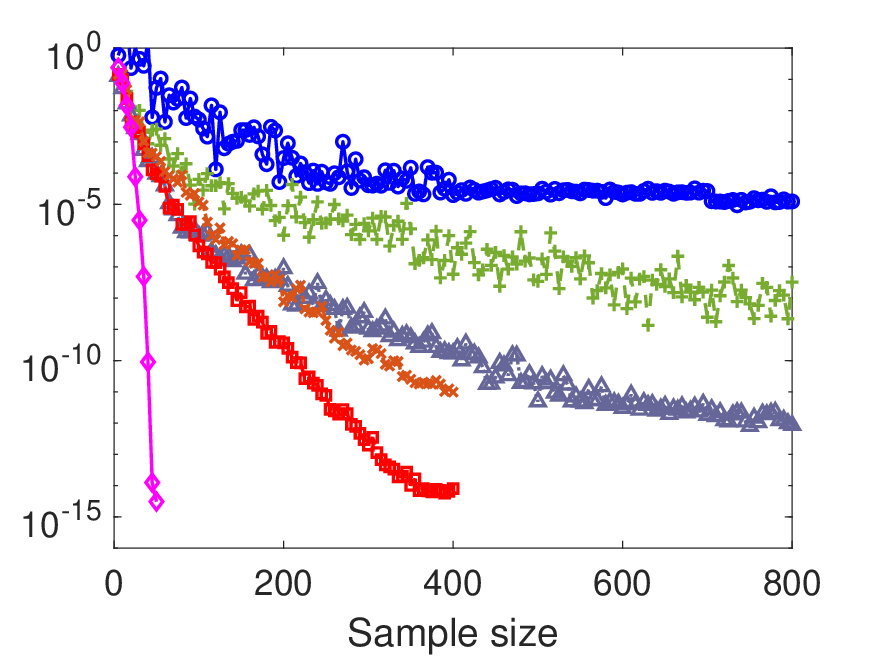} &
\includegraphics[height=0.97in]{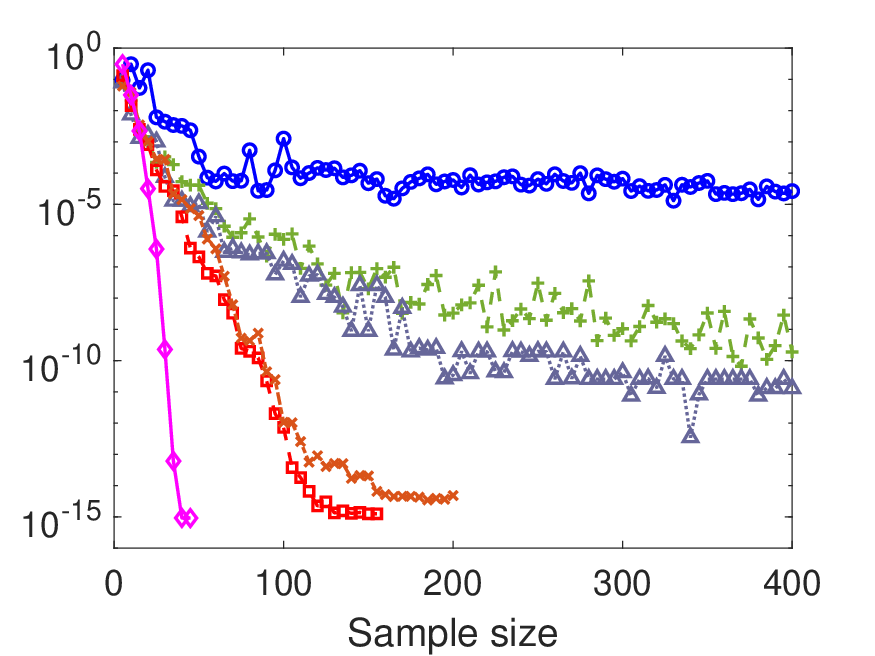}\\
\includegraphics[height=0.97in]{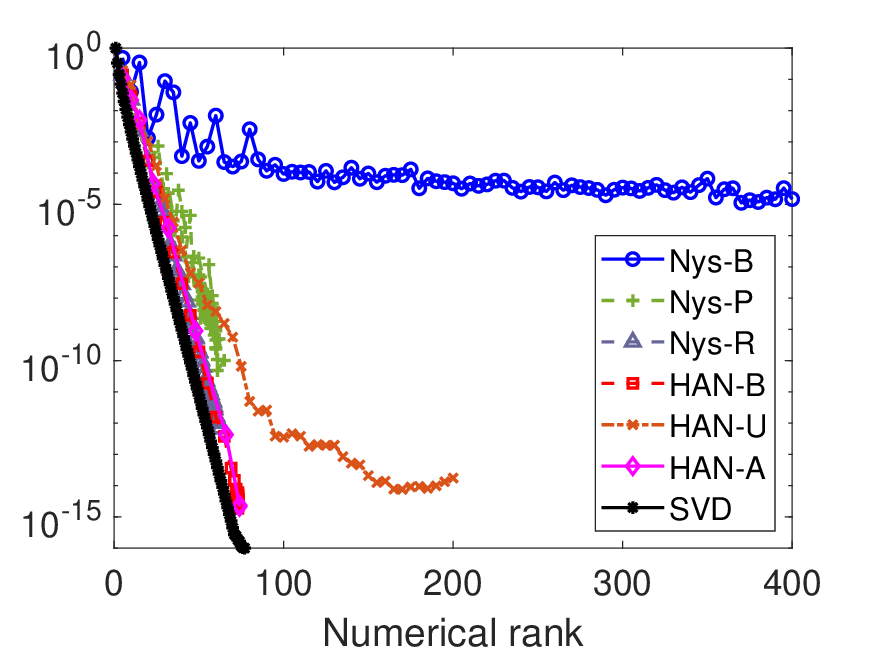} &
\includegraphics[height=0.97in]{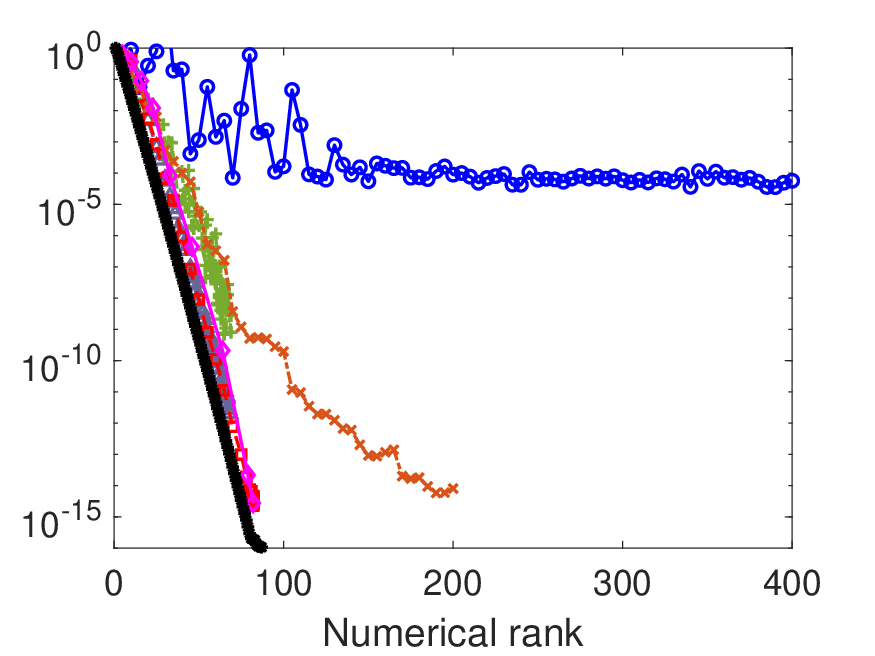} &
\includegraphics[height=0.97in]{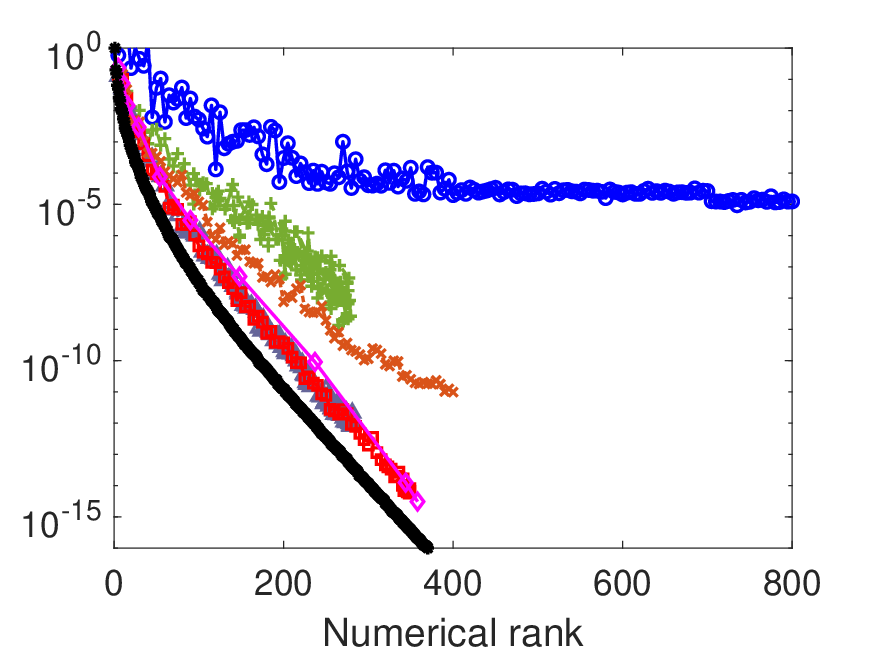} &
\includegraphics[height=0.97in]{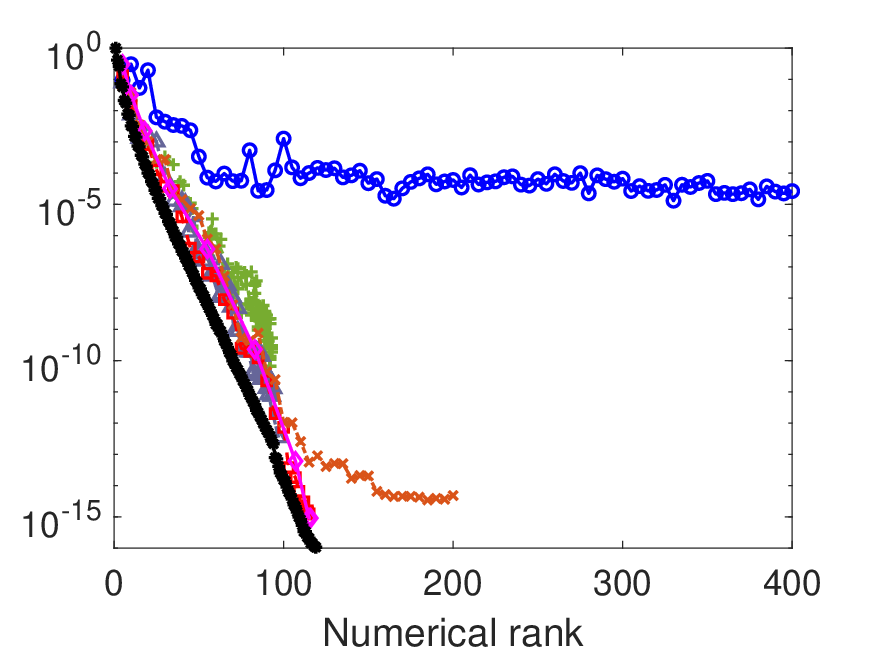}\\
$\kappa(x,y)=\frac{1}{x-y}$ & $\frac{1}{(x-y)^{2}}$ & $\frac{1}{|x-y|}$ &
$\log|x-y|$\end{tabular}
\caption{Example \ref{ex:mesh} (data set \texttt{FEM}): Low-rank approximation
errors $\frac{\Vert A-\tilde{A}\Vert_{2}}{\Vert A\Vert_{2}}$ as the sample
size $S$ increases in one test, where the second row shows the errors with
respect to the resulting numerical ranks corresponding to the first row and
the SVD line shows the scaled singular values.}\label{fig:meshparterr}\end{figure}\begin{figure}[ptbh]
\centering\tabcolsep1mm
\begin{tabular}
[c]{cccc}\includegraphics[height=1in]{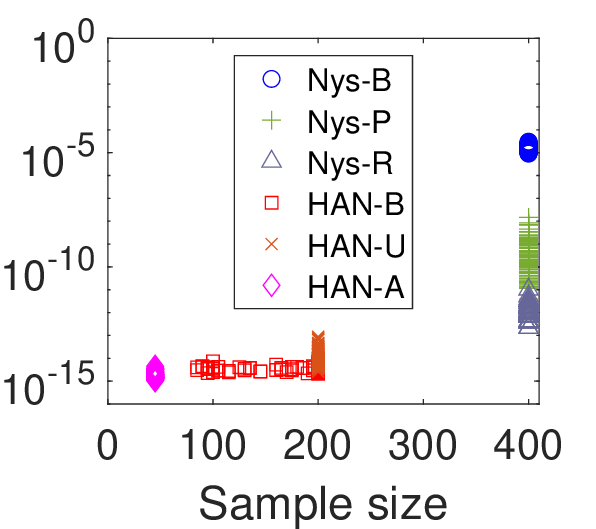} &
\includegraphics[height=1in]{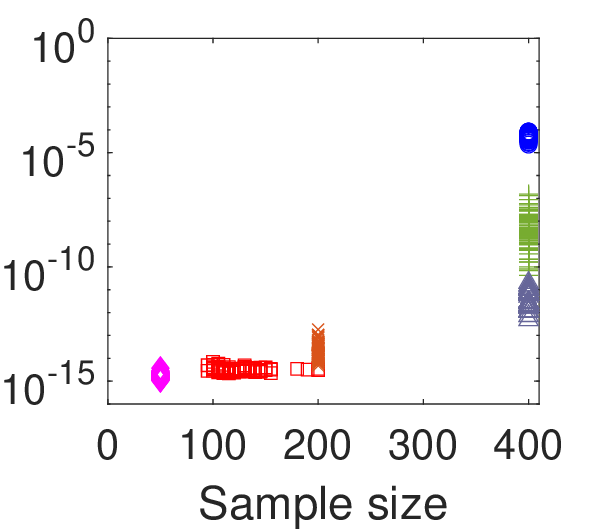} &
\includegraphics[height=1in]{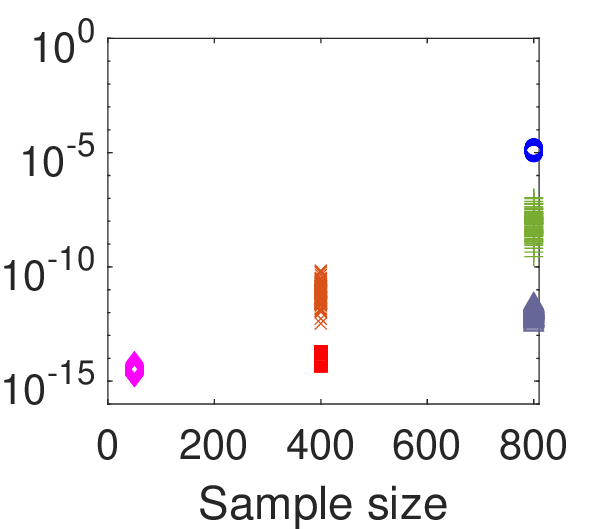} &
\includegraphics[height=1in]{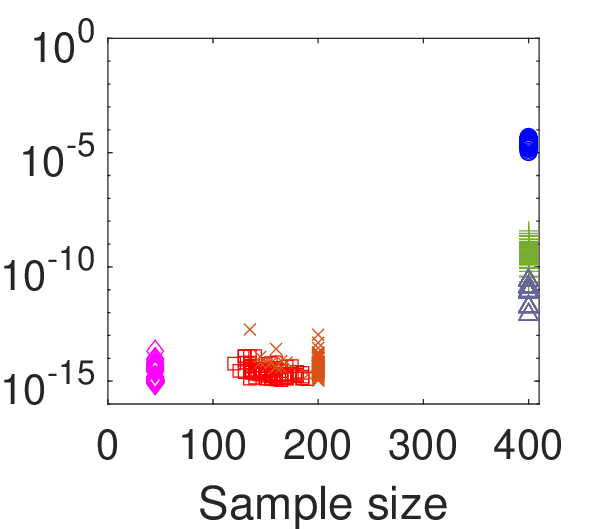}\\
\includegraphics[height=1in]{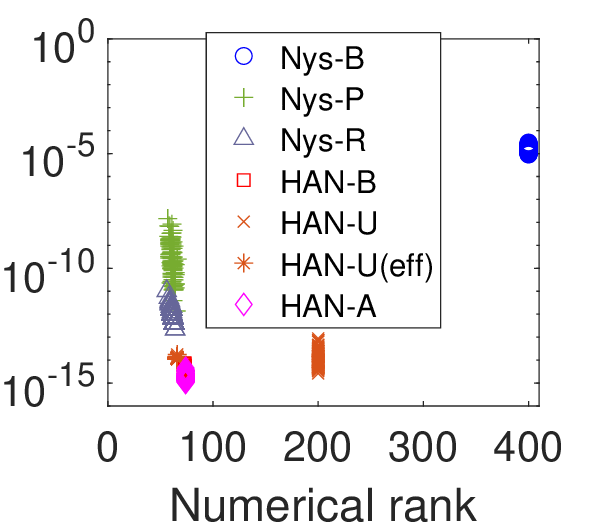} &
\includegraphics[height=1in]{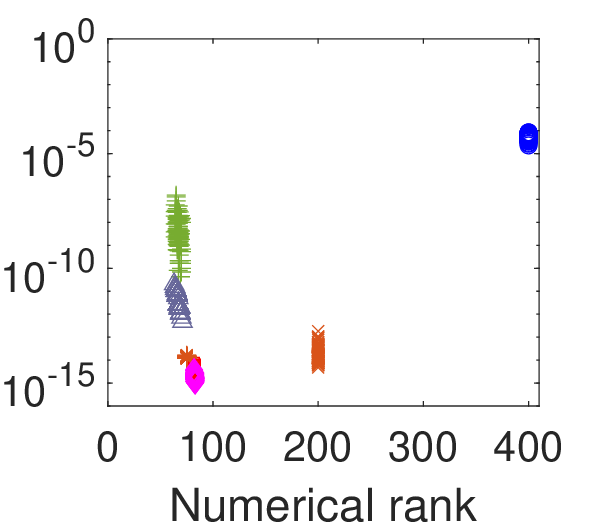} &
\includegraphics[height=1in]{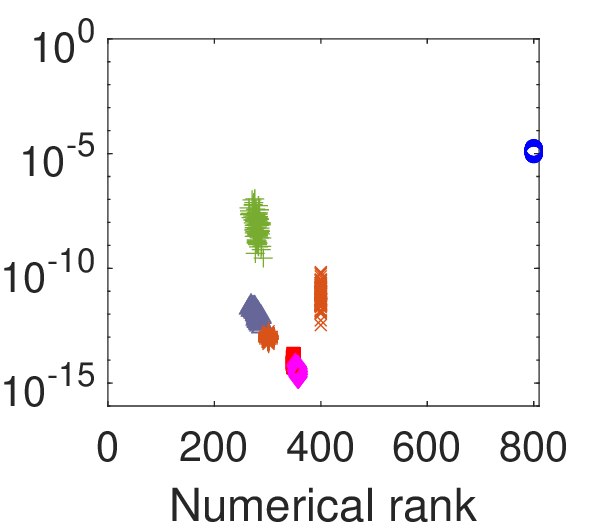} &
\includegraphics[height=1in]{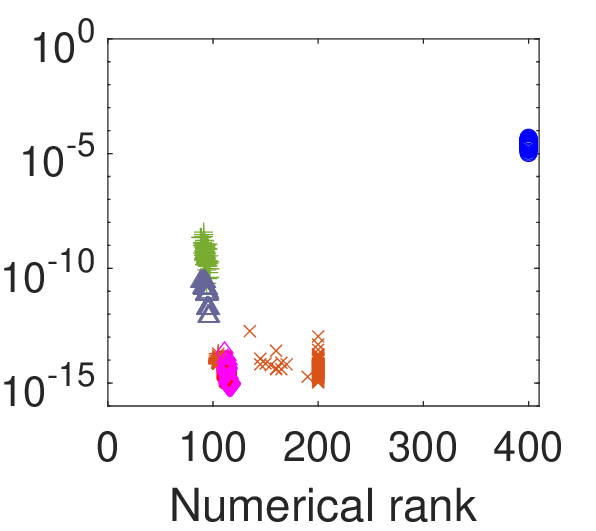}\\
$\kappa(x,y)=\frac{1}{x-y}$ & $\frac{1}{(x-y)^{2}}$ & $\frac{1}{|x-y|}$ &
$\log|x-y|$\end{tabular}
\caption{Example \ref{ex:mesh} (data set \texttt{FEM}): Relative errors from
running the methods for $100$ times, where \textsf{HAN-U(eff)} is for the
effective errors of \textsf{HAN-U}.}\label{fig:meshpart}\end{figure}

\begin{figure}[ptbh]
\centering\tabcolsep-0.5mm
\begin{tabular}
[c]{cccc}\includegraphics[height=0.97in]{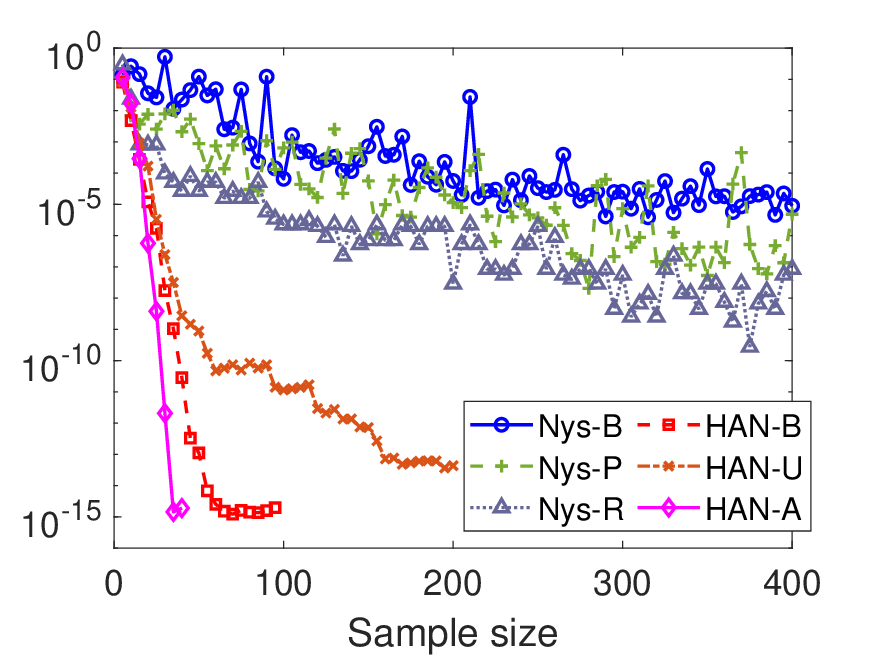} &
\includegraphics[height=0.97in]{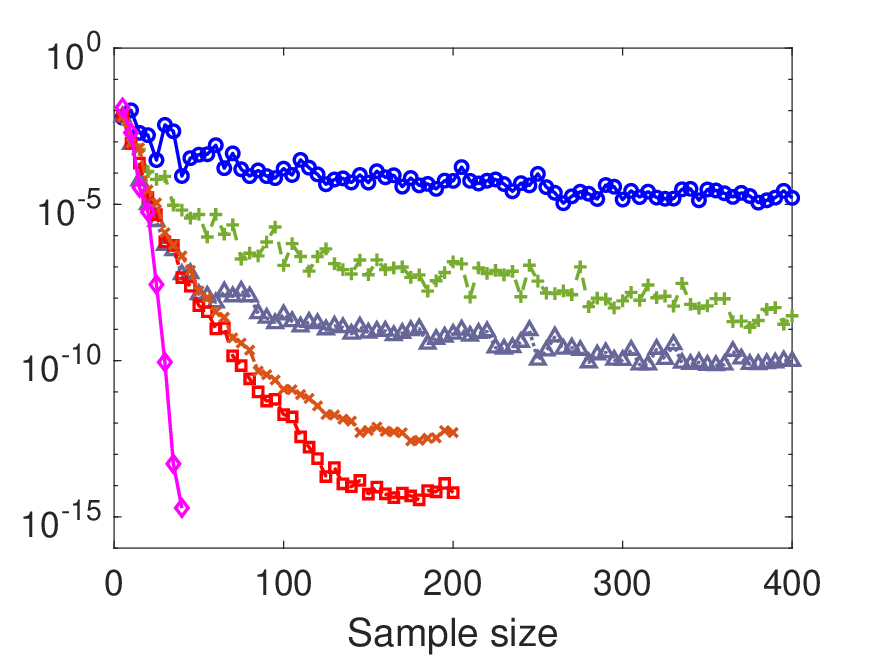} &
\includegraphics[height=0.97in]{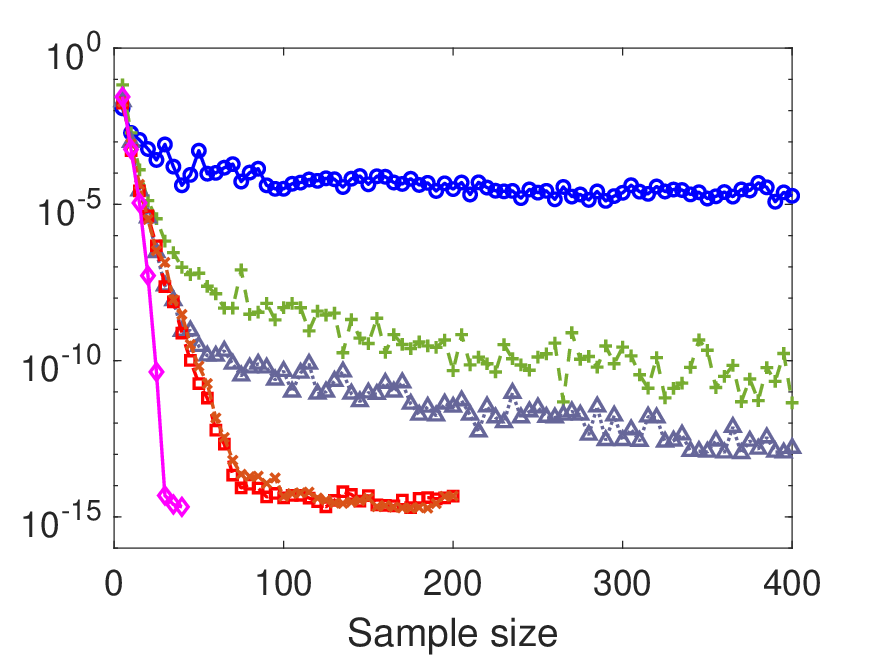} &
\includegraphics[height=0.97in]{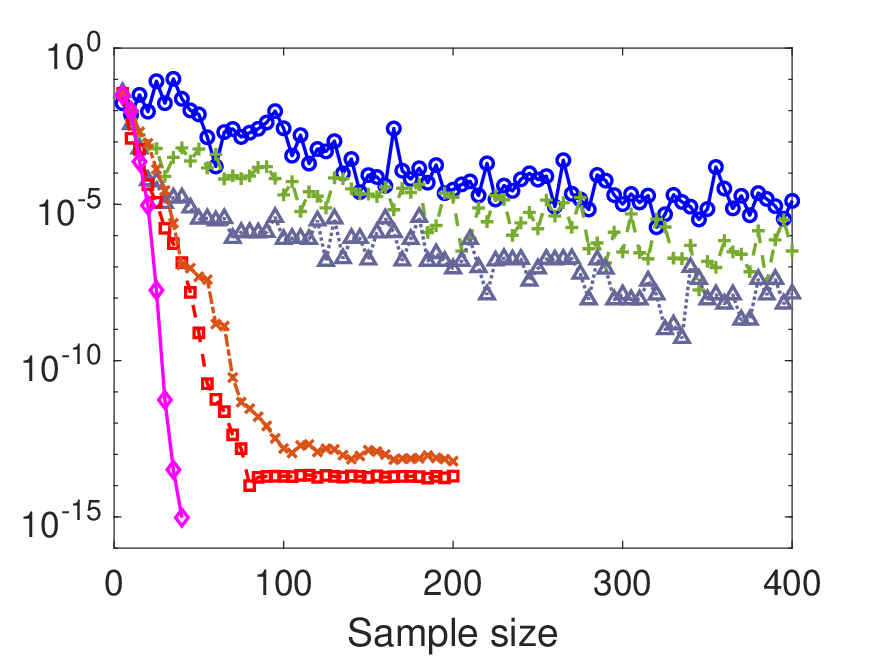}\\
\includegraphics[height=0.97in]{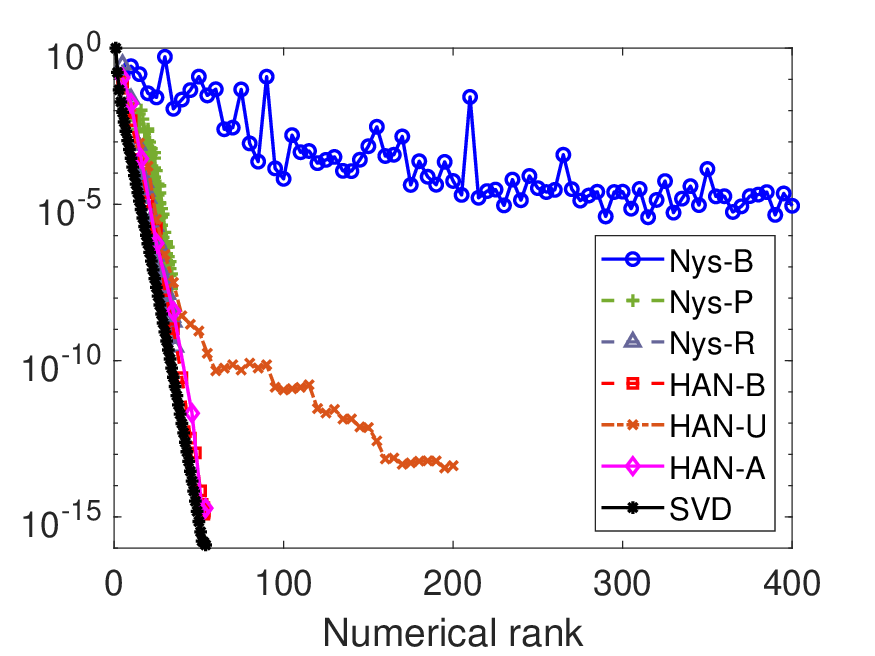} &
\includegraphics[height=0.97in]{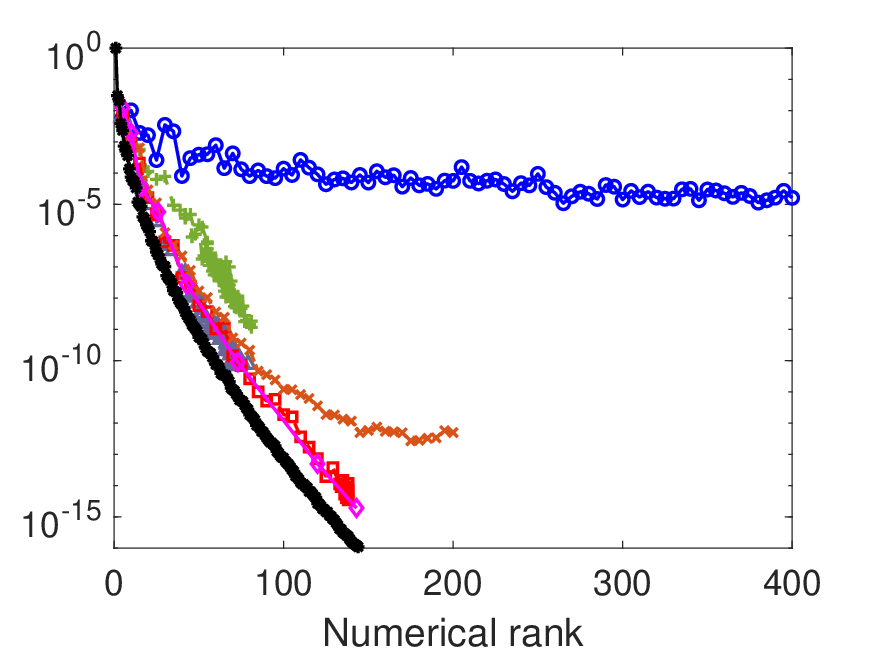} &
\includegraphics[height=0.97in]{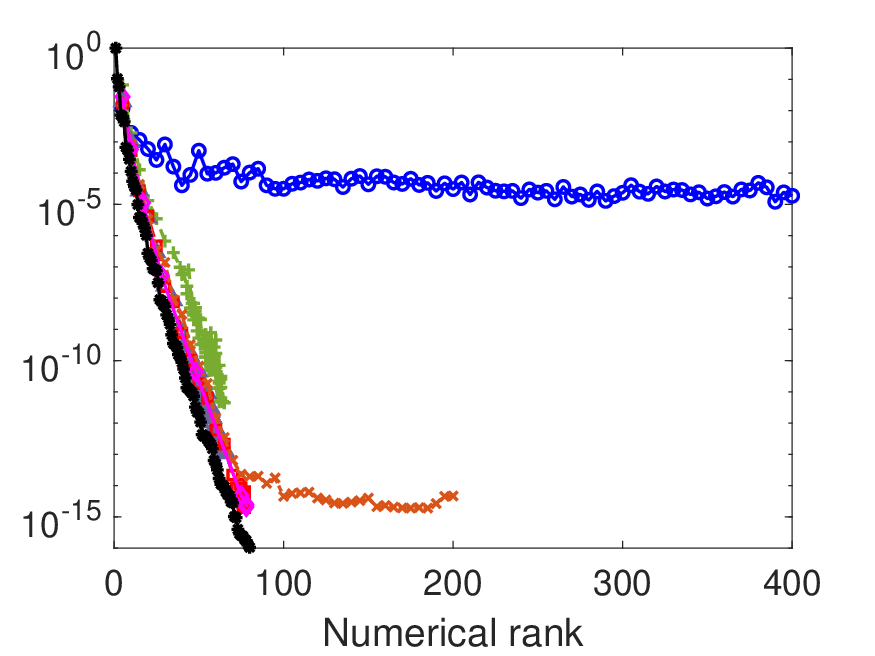} &
\includegraphics[height=0.97in]{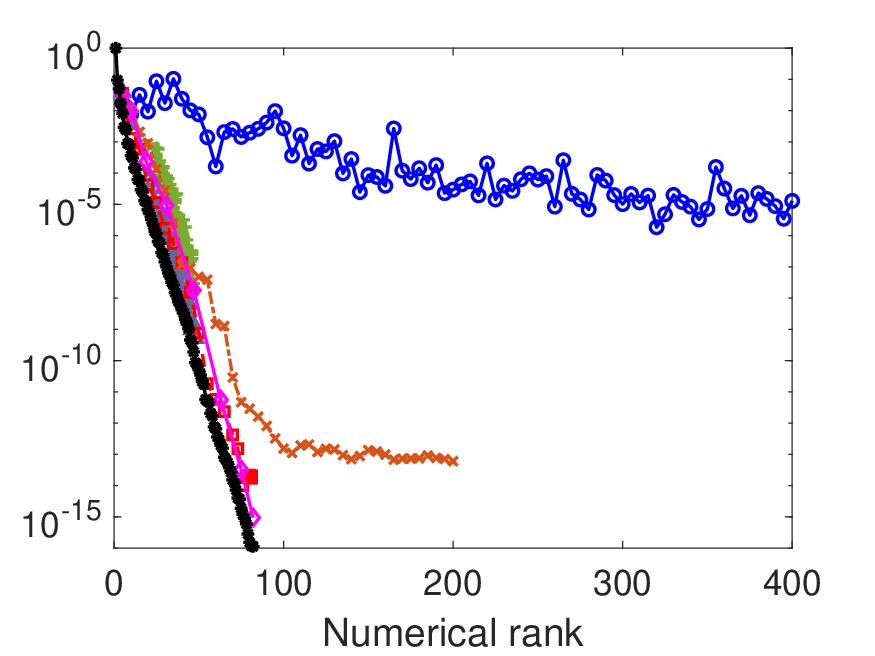}\\
$\kappa(x,y)=\frac{1}{x-y}$ & $\frac{1}{\sqrt{|x-y|^{2}+1}}$ & $e^{-\frac
{|x-y|^{2}}{2}}$ & $\log|x-y|$\end{tabular}
\caption{Example \ref{ex:mesh} (data set \texttt{Airfoil}): Low-rank
approximation errors $\frac{\Vert A-\tilde{A}\Vert_{2}}{\Vert A\Vert_{2}}$ as
the sample size $S$ increases in one test, where the second row shows the
errors with respect to the resulting numerical ranks corresponding to the
first row and the SVD line shows the scaled singular values.}\label{ex:airfoilerr}
\end{figure}
\begin{figure}[ptbh]
\centering\tabcolsep1mm
\begin{tabular}
[c]{cccc}\includegraphics[height=1in]{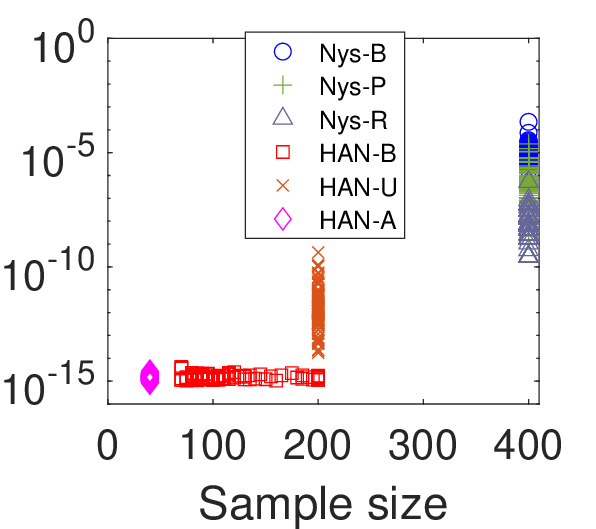} &
\includegraphics[height=1in]{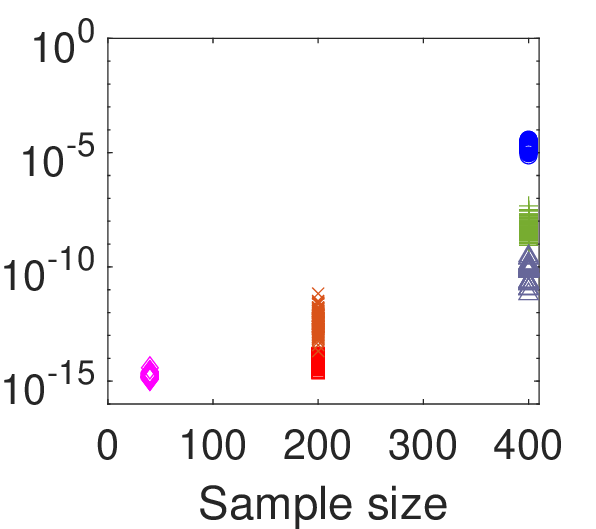} &
\includegraphics[height=1in]{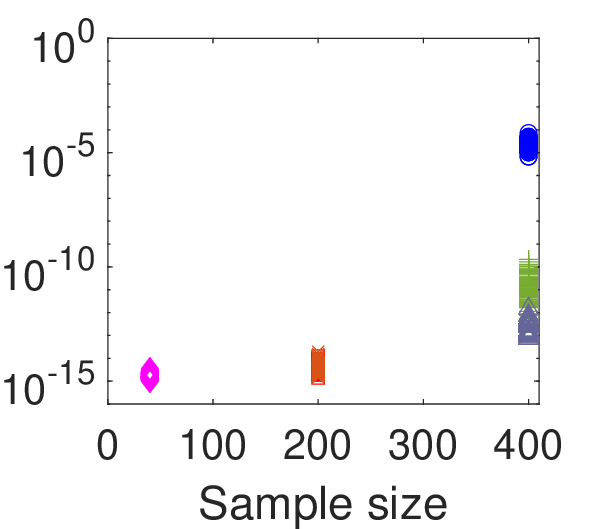} &
\includegraphics[height=1in]{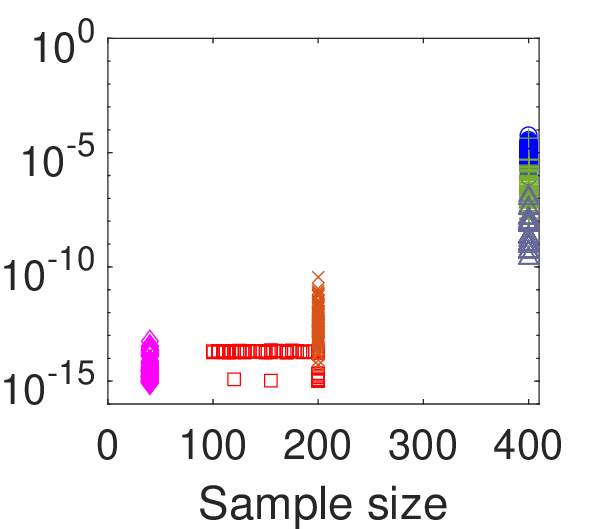}\\
\includegraphics[height=1in]{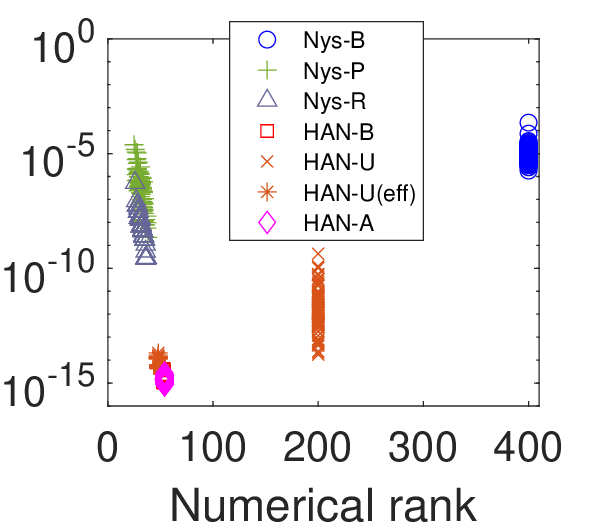} &
\includegraphics[height=1in]{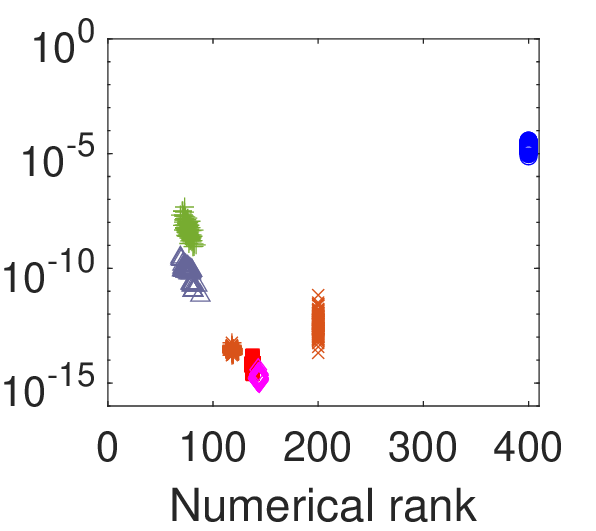} &
\includegraphics[height=1in]{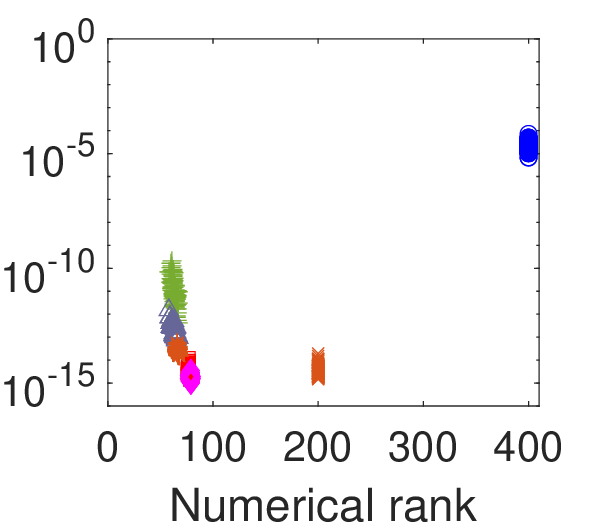} &
\includegraphics[height=1in]{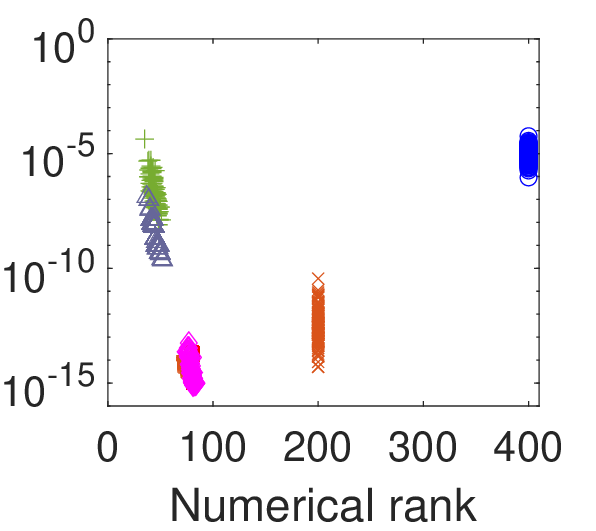}\\
$\kappa(x,y)=\frac{1}{x-y}$ & $\frac{1}{\sqrt{|x-y|^{2}+1}}$ & $e^{-\frac{|x-y|^{2}}{2}}$ & $\log|x-y|$\end{tabular}
\caption{Example \ref{ex:mesh} (data set \texttt{Airfoil}): Relative errors
from running the methods for $100$ times, where \textsf{HAN-U(eff)} is for the
effective errors of \textsf{HAN-U}.}\label{ex:airfoil}\end{figure}

\begin{figure}[ptbh]
\centering\tabcolsep-0.5mm
\begin{tabular}
[c]{cccc}\includegraphics[height=0.97in]{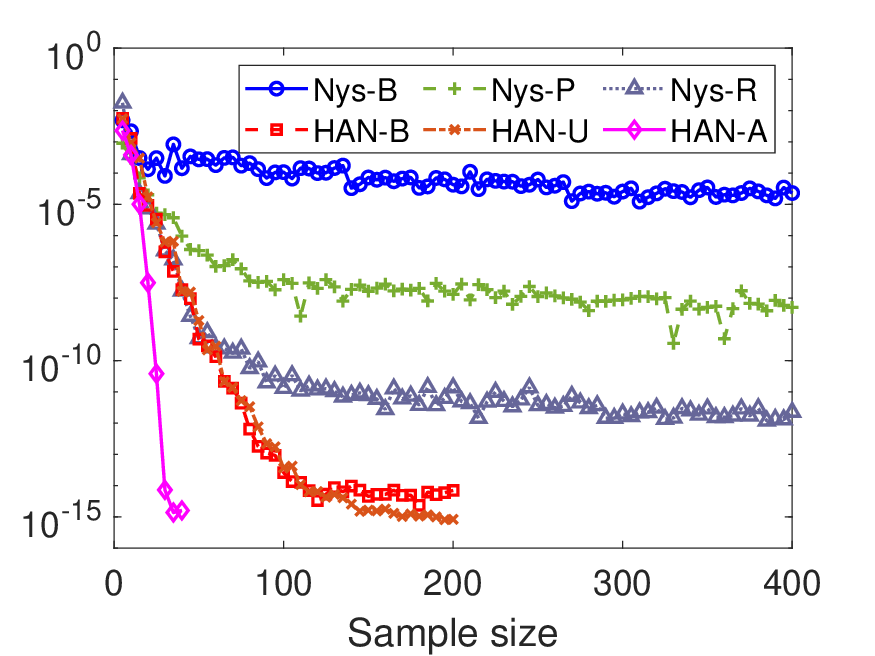} &
\includegraphics[height=0.97in]{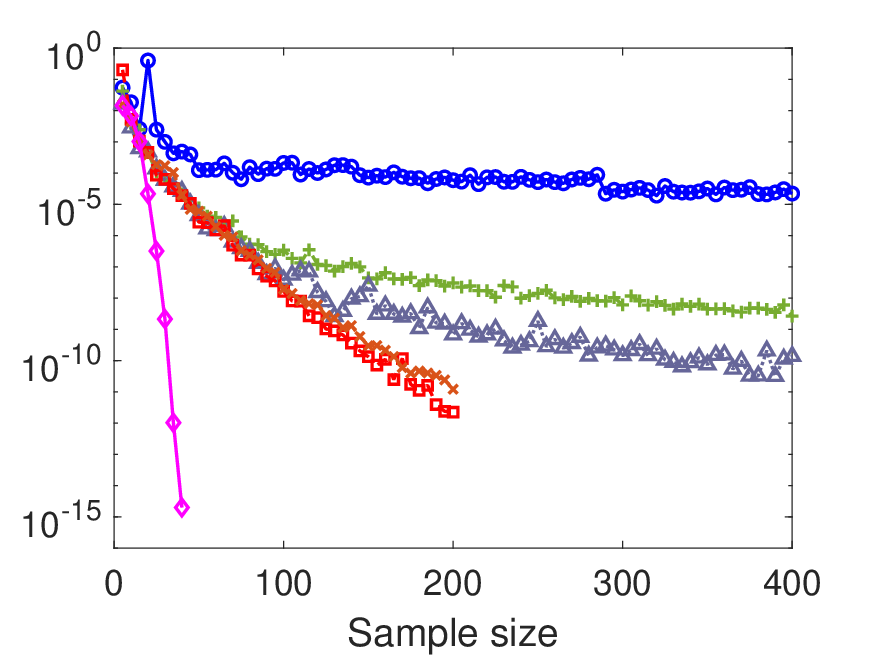} &
\includegraphics[height=0.97in]{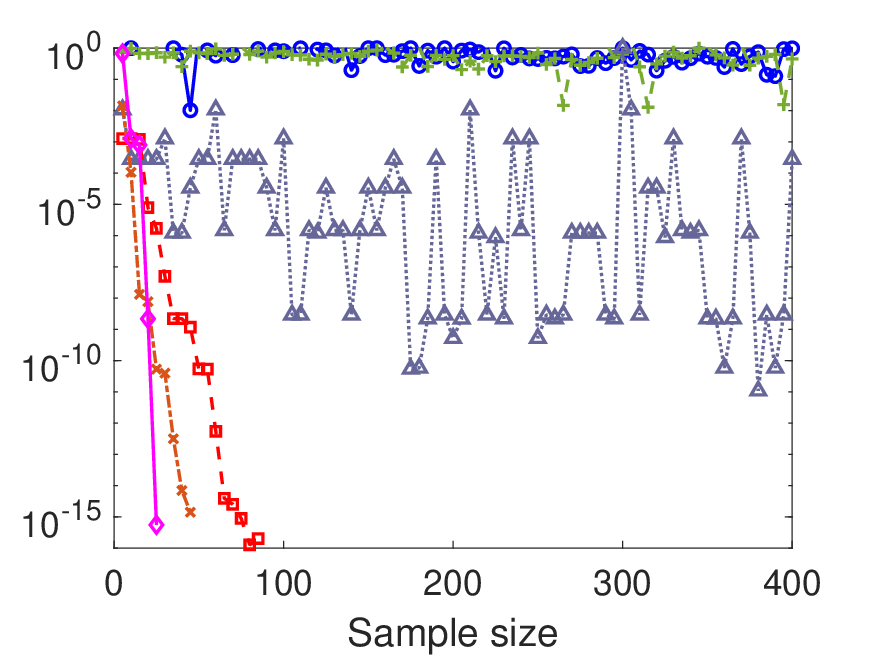} &
\includegraphics[height=0.97in]{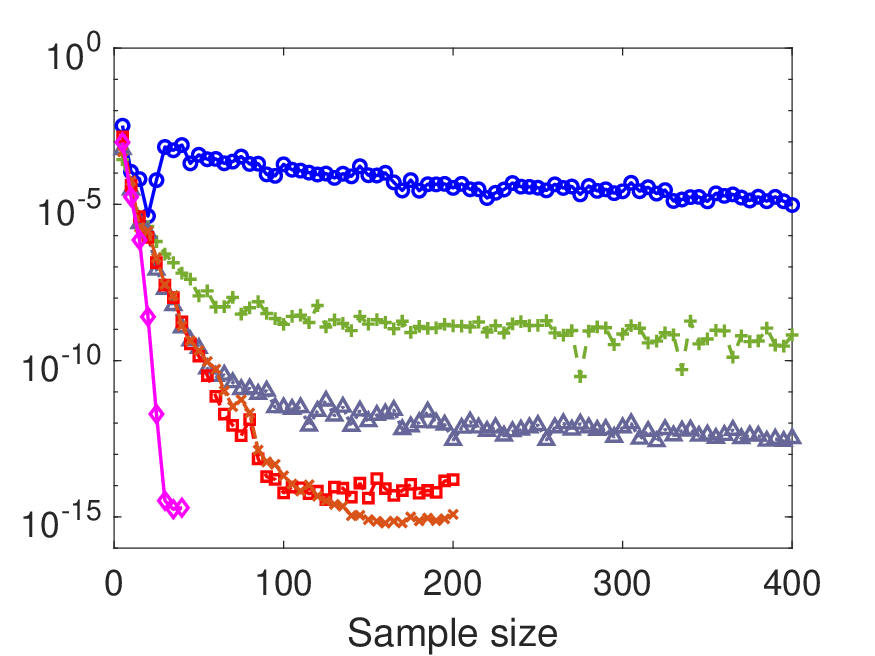}\\
\includegraphics[height=0.97in]{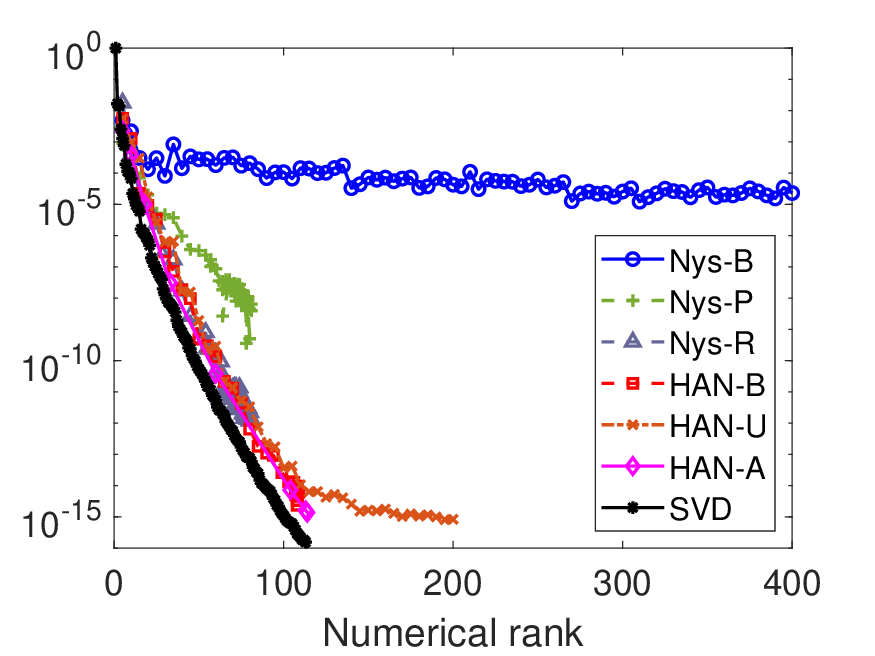} &
\includegraphics[height=0.97in]{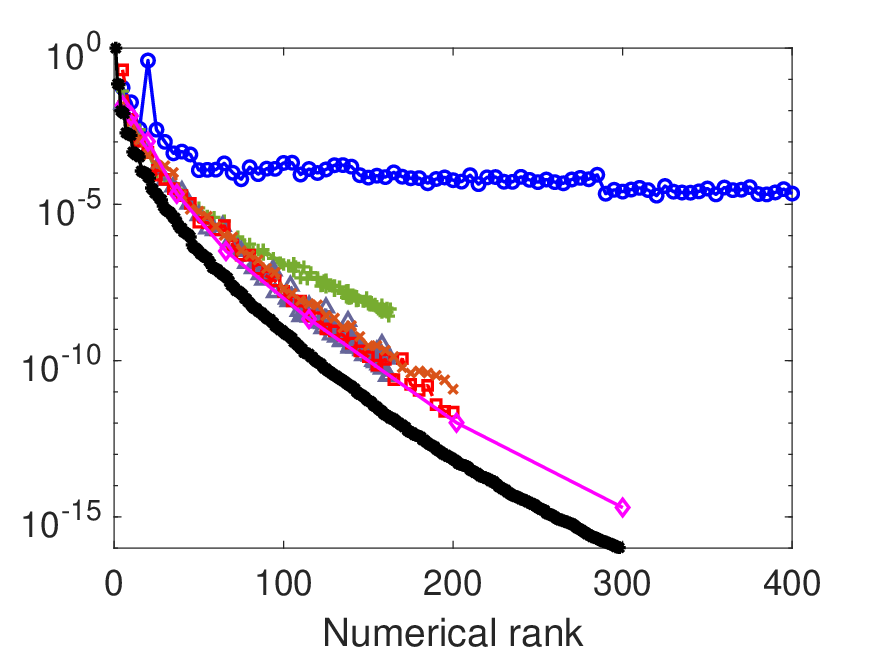} &
\includegraphics[height=0.97in]{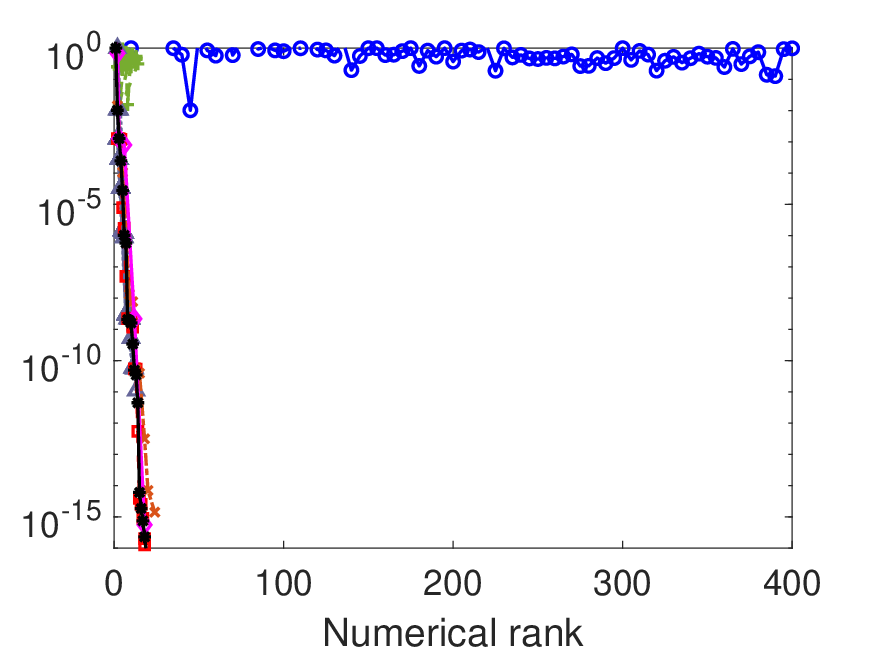} &
\includegraphics[height=0.97in]{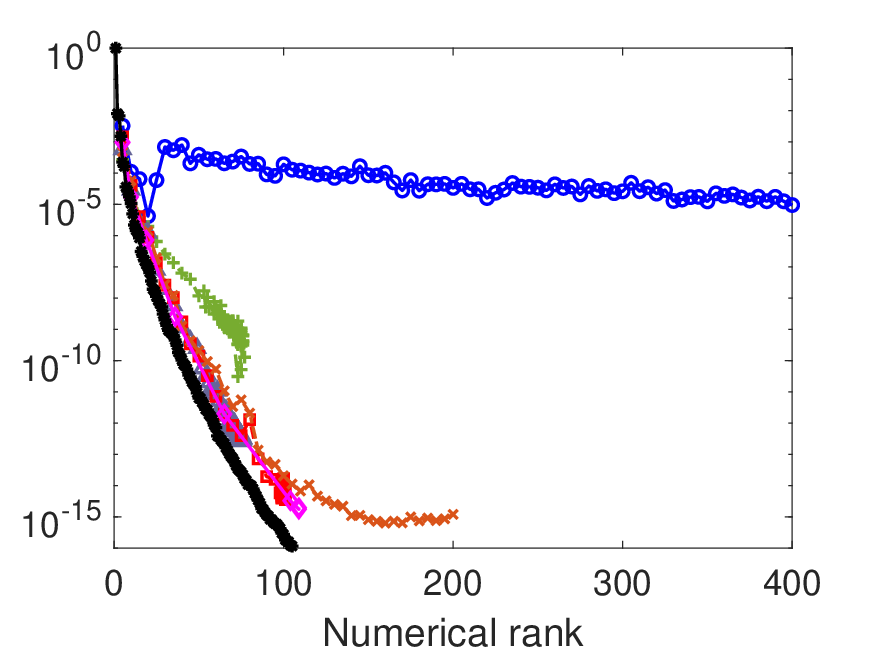}\\
$\kappa(x,y)=\frac{1}{|x-y|}$ & $\tan(x\cdot y+1)$ & $e^{-16|x-y|^{2}}$ &
$\log|x-y|$\end{tabular}
\caption{Example \ref{ex:mesh} (data set \texttt{Set3D}): Low-rank
approximation errors $\frac{\Vert A-\tilde{A}\Vert_{2}}{\Vert A\Vert_{2}}$ as
the sample size $S$ increases in one test, where the second row shows the
errors with respect to the resulting numerical ranks corresponding to the
first row and the SVD line shows the scaled singular values.}\label{ex:mesh3derr}\end{figure}
\begin{figure}[ptbh]
\centering\tabcolsep1mm
\begin{tabular}
[c]{cccc}\includegraphics[height=1in]{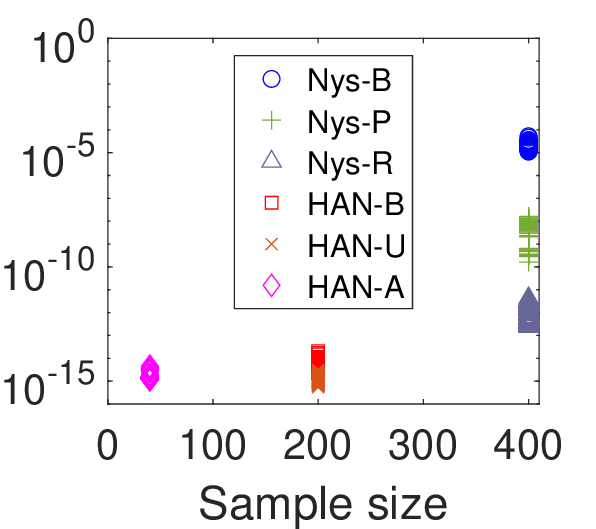} &
\includegraphics[height=1in]{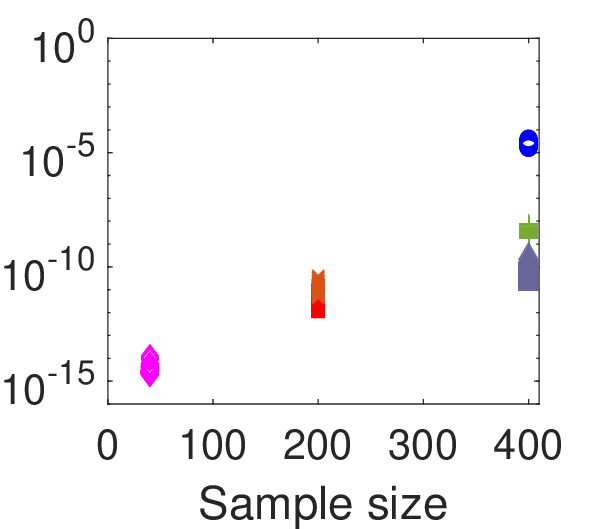} &
\includegraphics[height=1in]{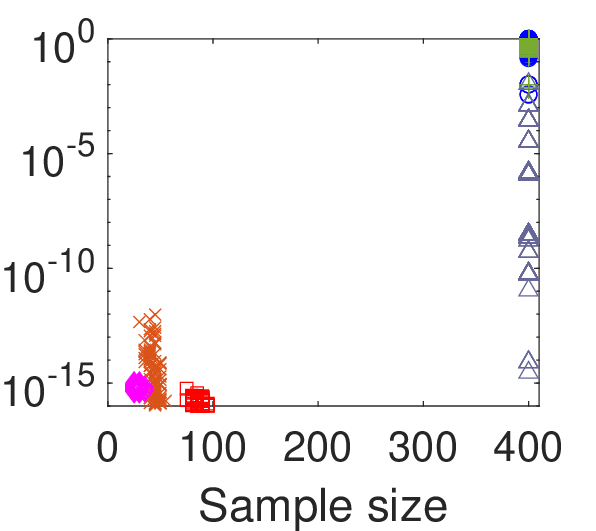} &
\includegraphics[height=1in]{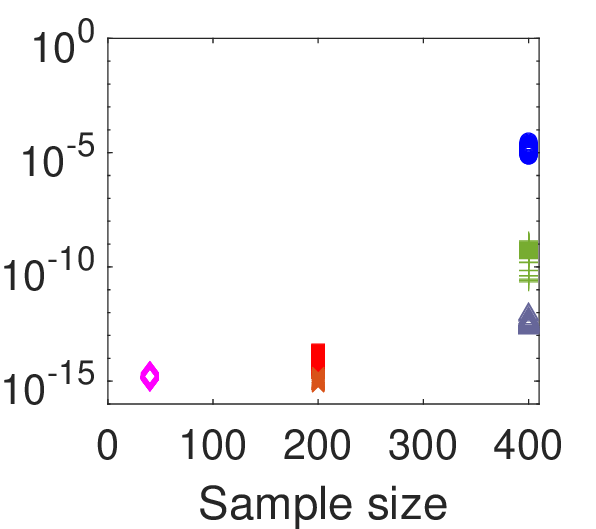}\\
\includegraphics[height=1in]{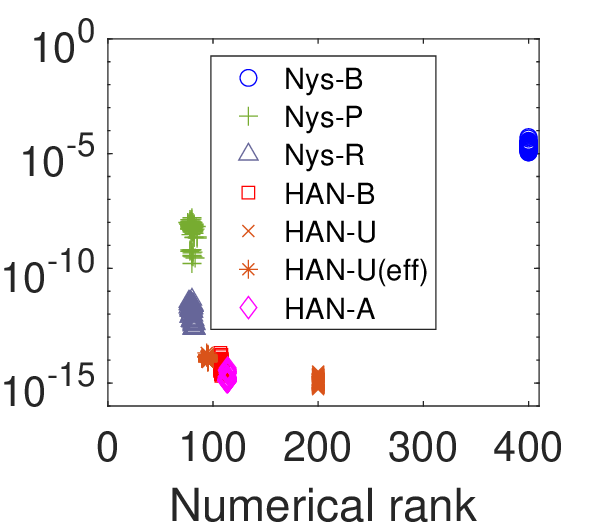} &
\includegraphics[height=1in]{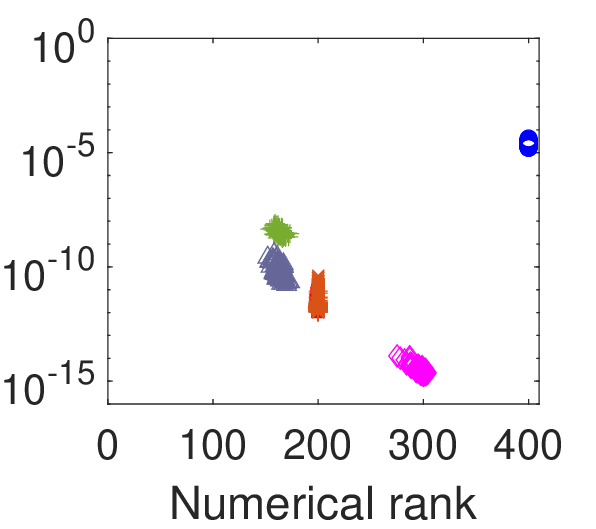} &
\includegraphics[height=1in]{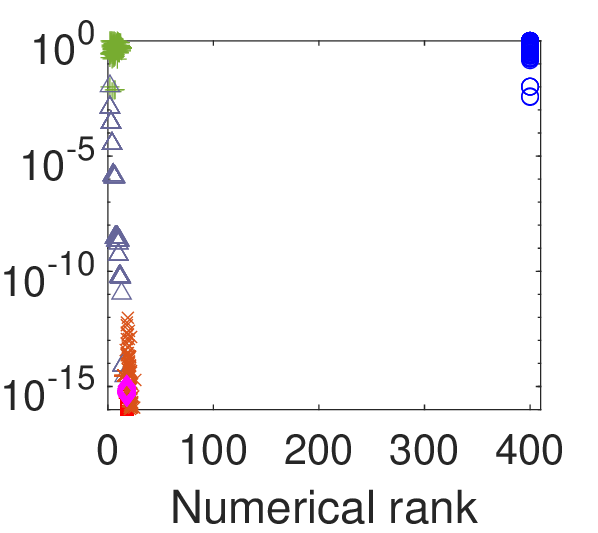} &
\includegraphics[height=1in]{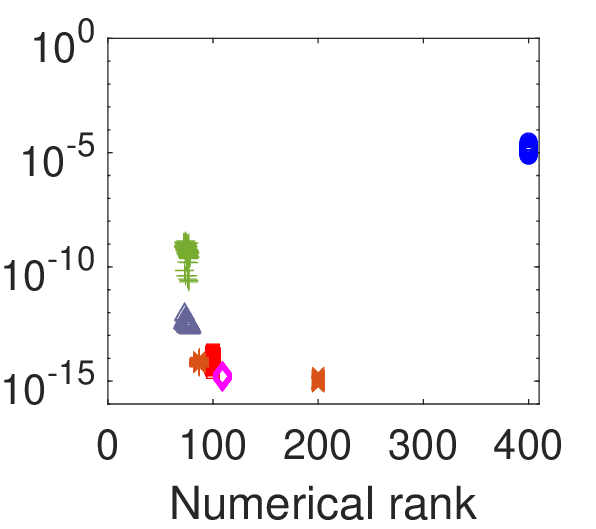}\\
$\kappa(x,y)=\frac{1}{|x-y|}$ & $\tan(x\cdot y+1)$ & $e^{-16|x-y|^{2}}$ &
$\log|x-y|$\end{tabular}
\caption{Example \ref{ex:mesh} (data set \texttt{Set3D}): Relative errors from
running the methods for $100$ times, where \textsf{HAN-U(eff)} is for the
effective errors of \textsf{HAN-U}.}\label{ex:mesh3d}\end{figure}

The aggressive rank advancement also makes \textsf{HAN-A} very efficient. For
each data set, the average timing of \textsf{HAN-A} and \textsf{Nys-R} from
$100$ runs is shown in Table \ref{tab:ex1time}. \textsf{HAN-A} is generally
faster than \textsf{Nys-R} by multiple times.
\begin{table}[ptbh]
\caption{Example \ref{ex:mesh}. Average timing (in seconds) of \textsf{HAN-A}
and \textsf{Nys-R} from $100$ runs.}\label{tab:ex1time}
\begin{center}
\tabcolsep5pt\renewcommand{\arraystretch}{1.2}
\begin{tabular}
[c]{|c|c|c|c|c|c|}\hline
\multirow{3}{*}{\texttt{Flower}} & $\kappa(x,y)$ & $\frac{1}{x-y}$ &
$\sqrt{|x-y|+1}$ & $e^{-|x-y|}$ & $\log|x-y|$\\\cline{2-6}
& \textsf{Nys-R} & 1.751 & 1.285 & 1.313 & 0.851\\
& \textsf{HAN-A} & 0.802 & 0.727 & 0.690 & 0.512\\\hline
\multirow{3}{*}{\texttt{FEM}} & $\kappa(x,y)$ & $\frac{1}{x-y}$ & $\frac
{1}{(x-y)^{2}}$ & $\frac{1}{|x-y|}$ & $\log|x-y|$\\\cline{2-6}
& \textsf{Nys-R} & 0.868 & 0.962 & 2.495 & 0.469\\
& \textsf{HAN-A} & 0.191 & 0.193 & 0.607 & 0.108\\\hline
\multirow{3}{*}{\texttt{Airfoil}} & $\kappa(x,y)$ & $\frac{1}{x-y}$ &
$\frac{1}{\sqrt{|x-y|^{2}+1}}$ & $e^{-\frac{|x-y|^{2}}{2}}$ & $\log
|x-y|$\\\cline{2-6}
& \textsf{Nys-R} & 1.155 & 0.937 & 0.734 & 0.457\\
& \textsf{HAN-A} & 0.322 & 0.451 & 0.289 & 0.158\\\hline
\multirow{3}{*}{\texttt{Set3D}} & $\kappa(x,y)$ & $\frac{1}{|x-y|}$ &
$\tan(x\cdot y+1)$ & $e^{-16|x-y|^{2}}$ & $\log|x-y|$\\\cline{2-6}
& \textsf{Nys-R} & 0.723 & 1.910 & 0.121 & 0.680\\
& \textsf{HAN-A} & 0.355 & 0.788 & 0.035 & 0.333\\\hline
\end{tabular}
\end{center}
\end{table}

\begin{example}
\label{ex:circ}Next, consider a class of implicitly defined kernel matrices
with varying sizes. Suppose $C$ is a circulant matrix with eigenvalues being discretized values of a function
$f(t)$ at some points in an interval. Such
matrices appear in some image processing problems \cite{nag97}, solutions of
ODEs and PDEs \cite{bai03,bai20}, and spectral methods \cite{spectral1d}. They
are usually multiplied or added to some other matrices so that the circulant
structure is destroyed. However, it is shown in \cite{spectral1d,circ} that
they have small off-diagonal numerical ranks for some $f(t)$. Such rank
structures are preserved under various matrix operations. The matrix $A$ we
consider here is the $n\times n$ upper right corner block of $C$ (with half of
the size of $C$). It is also shown in \cite{circ} that $A$ is the evaluation
of an implicit kernel function over certain data points.
\end{example}

We consider $A$ with its size $n=512,1024,\ldots,16384$ so as to demonstrate that \textsf{HAN-A} can reach high accuracies with nearly linear complexity.
For each $n$, we run \textsf{HAN-A} for $10$ times and report the outcome. As
$n$ doubles, Figure \ref{fig:circ}(a) shows the numerical ranks $r$ from
\textsf{HAN-A}, which slowly increase with $n$. This is consistent with the
result in \cite{circ} where it is shown that the numerical ranks grow as a
low-degree power of $\log n$. The low-rank approximation errors are given in
Figure \ref{fig:circ}(b) and the average time from the $10$ runs for each $n$ is given in Figure \ref{fig:circ}(c).
The runtimes roughly follow the
$O(r^{2}n)$ pattern, as explained in Section \ref{sub:upd}.

\begin{figure}[ptbh]
\centering\tabcolsep0mm
\begin{tabular}
[c]{ccc}\includegraphics[height=1.2in]{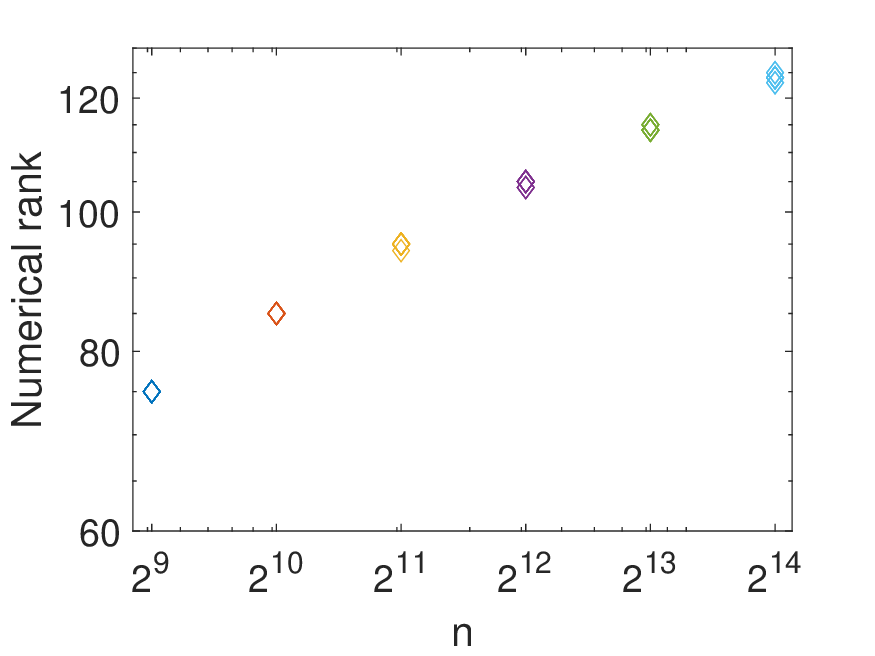} &
\includegraphics[height=1.2in]{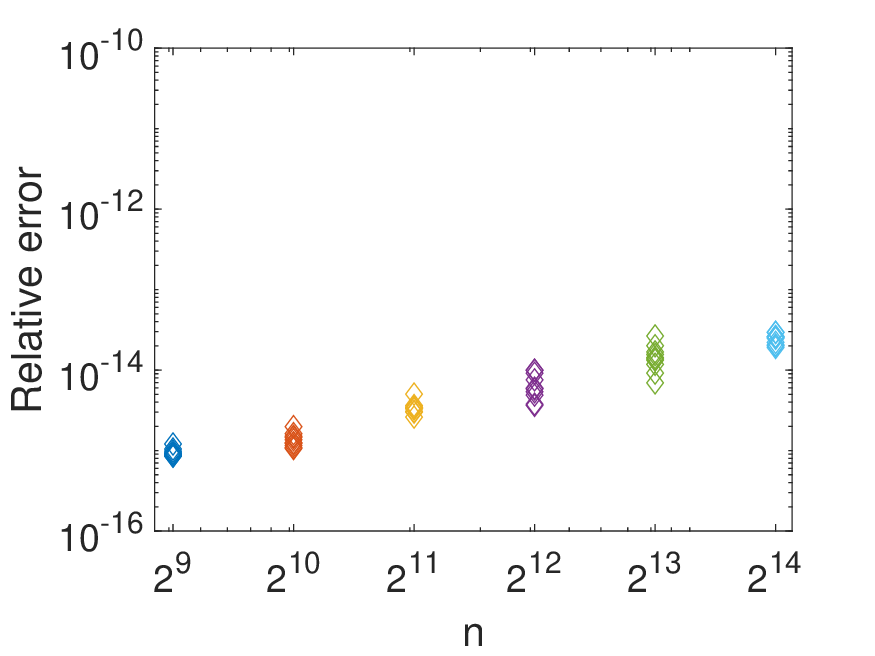} &
\includegraphics[height=1.2in]{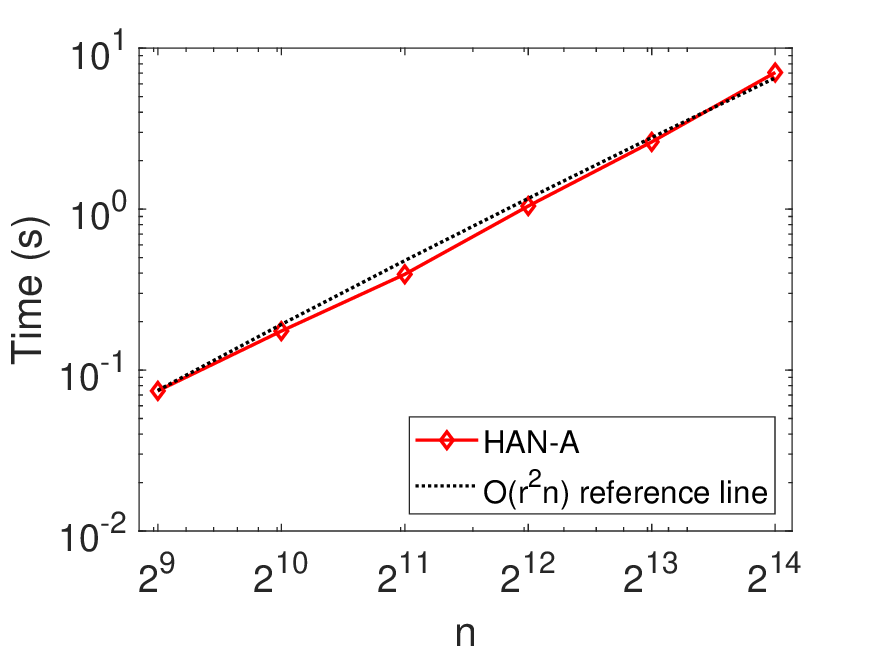}\\
{\small (a) Numerical rank $r$} & {\small (b) $\frac{\Vert A-\tilde{A}\Vert_{2}}{\Vert A\Vert_{2}}$} & {\small (c) Average timing}\end{tabular}
\caption{Example \ref{ex:circ}: Results from \textsf{HAN-A} for a type of
implicitly defined kernel matrices.}\label{fig:circ}\end{figure}

\begin{example}
\label{ex:data}Finally for completeness, we would like to show that the HAN
schemes also work for high-dimensional data sets. (We remark that practical
data analysis may not necessarily need very high accuracies. However, the
HAN\ schemes can serve as a fast way to convert such data matrices into some rank structured forms that
allow quick matrix operations.) We consider kernel
matrices resulting from the evaluation of some kernel functions at two data
sets \texttt{Abalone} and \texttt{DryBean} from the UCI Machine Learning
Repository (https://archive.ics.uci.edu). The two data sets have $4177$ and
$13611$ points in $8$ and $16$ dimensions, respectively. Here, each data set
is standardized to have mean $0$ and variance $1$. We take the submatrix of
each resulting kernel matrix formed by the first $1000$ rows so as to make it
rectangular and nonsymmetric.
\end{example}

A\ set of test results is given in Figure \ref{fig:data}. \textsf{Nys-B} can
only reach modest accuracies around $10^{-5}$. \textsf{Nys-R} can indeed gets
quite good accuracies. Nevertheless, \textsf{HAN-A} still reaches high
accuracies with a small number of sampling steps. Similar results are observed
with multiple runs.

\begin{figure}[ptbh]
\centering\tabcolsep-0.5mm
\begin{tabular}
[c]{cccc}\multicolumn{2}{c}{{\small (Set \texttt{Abalone})}} &
\multicolumn{2}{c}{{\small (Set \texttt{DryBean})}}\\
\includegraphics[height=0.97in]{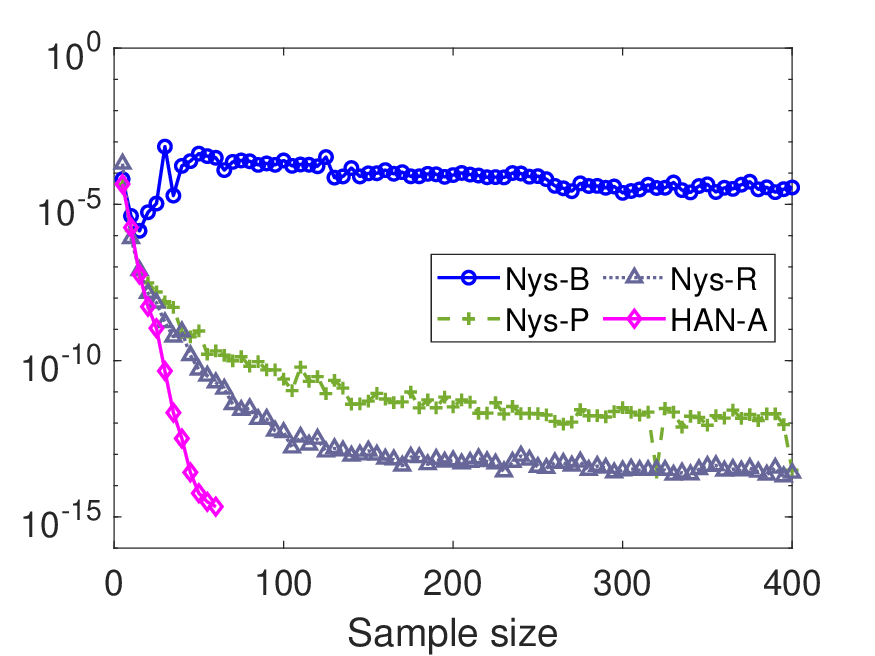} &
\includegraphics[height=0.97in]{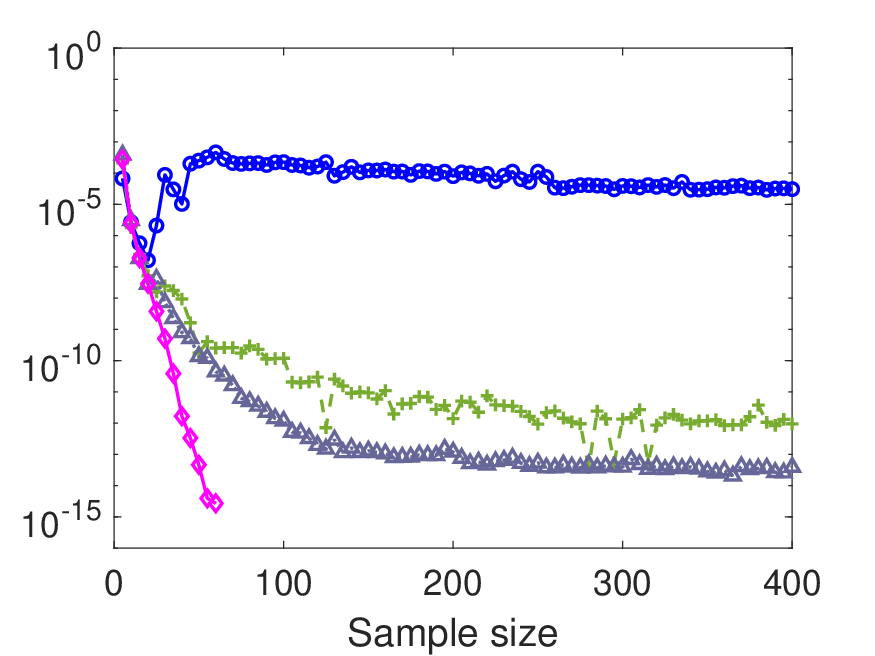} &
\includegraphics[height=0.97in]{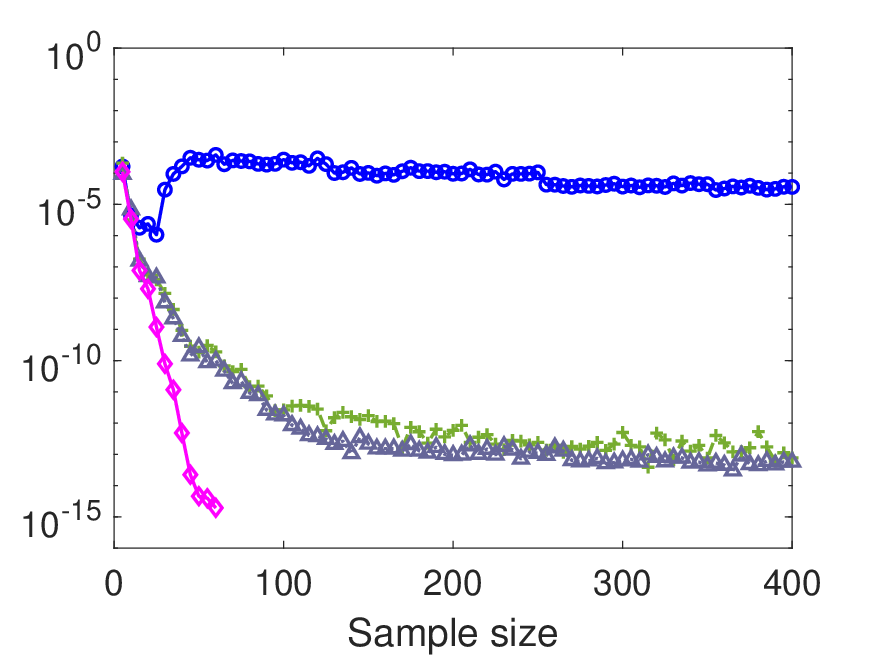} &
\includegraphics[height=0.97in]{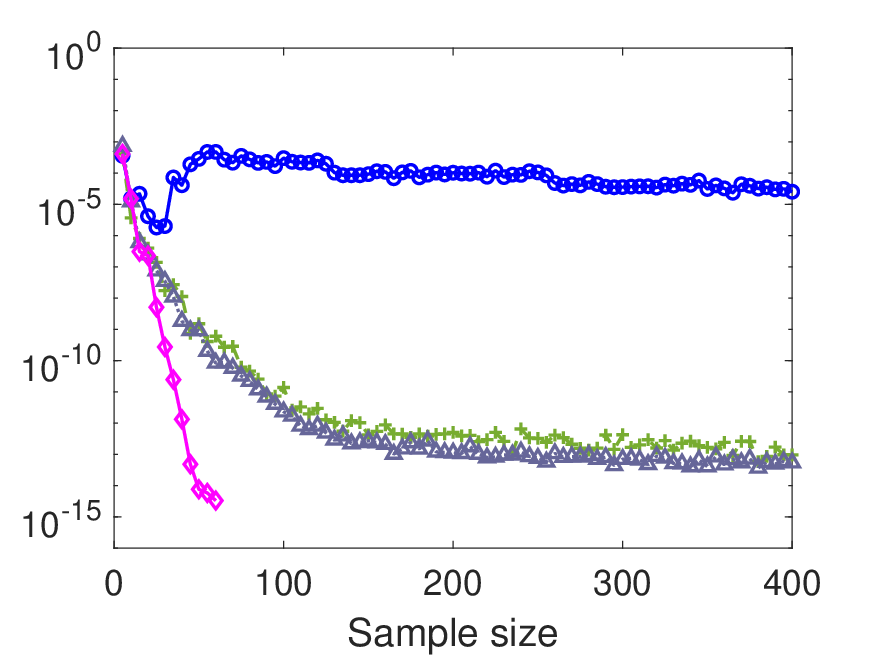}\\
\includegraphics[height=0.97in]{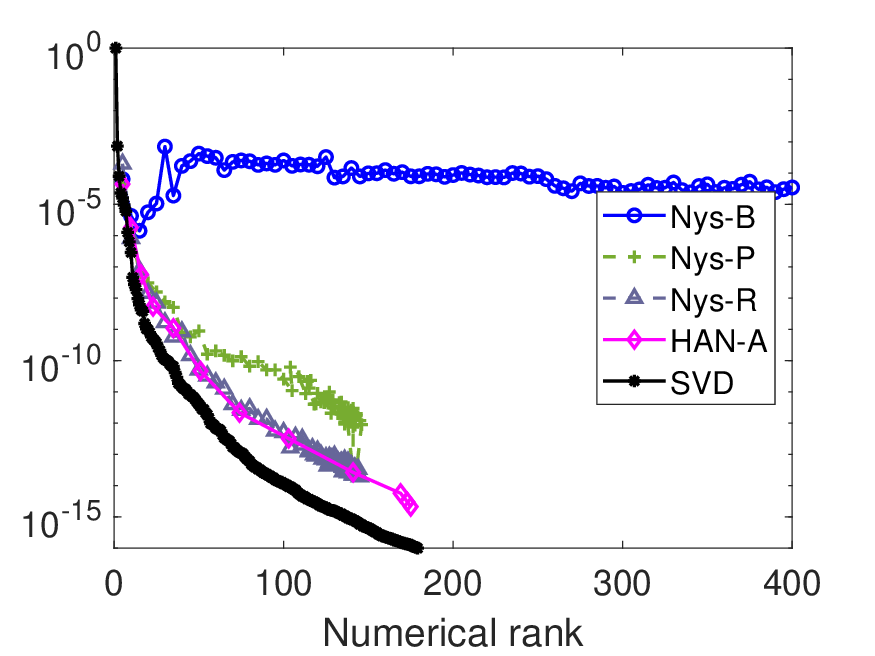} &
\includegraphics[height=0.97in]{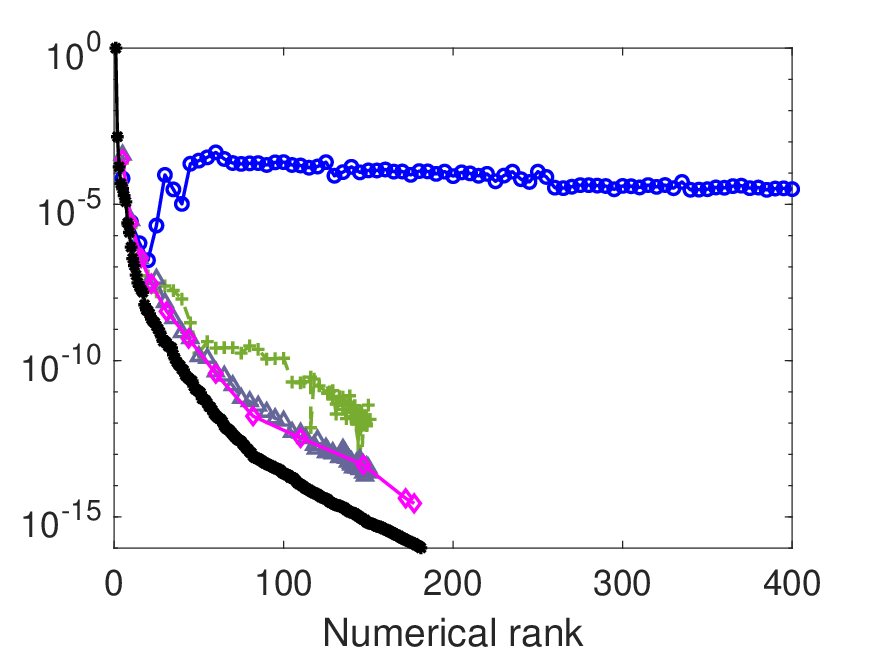} &
\includegraphics[height=0.97in]{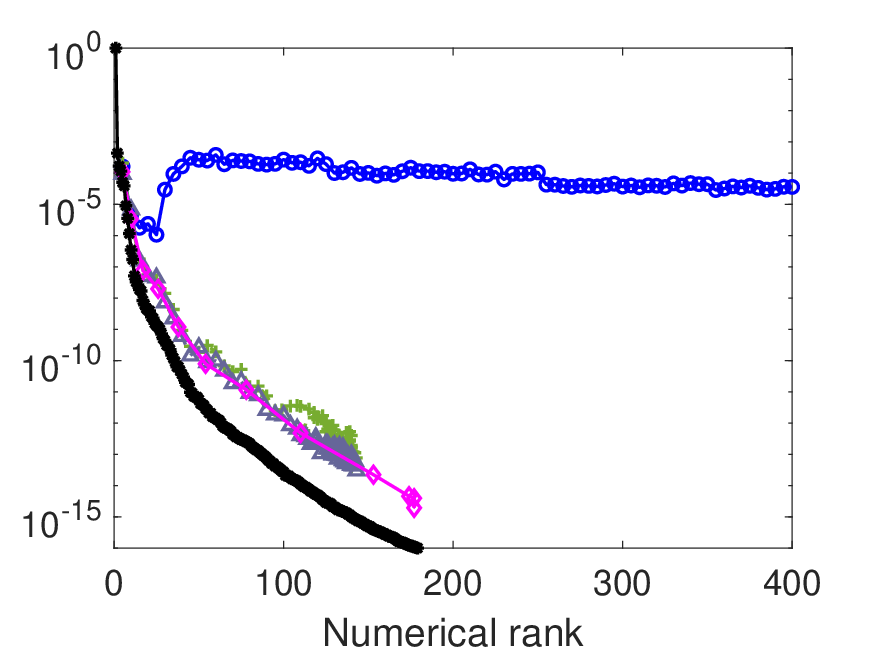} &
\includegraphics[height=0.97in]{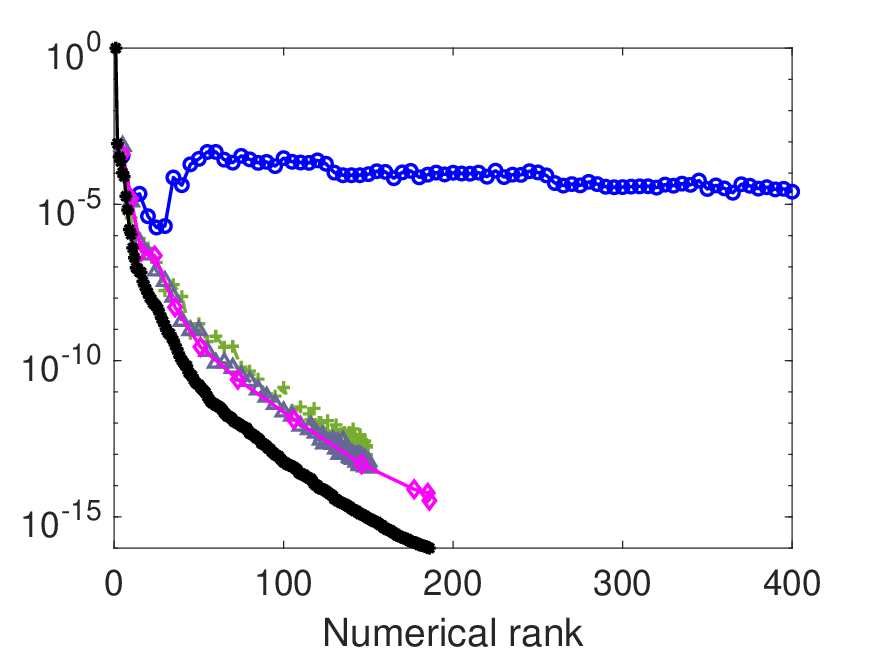}\\
{\small $\kappa(x,y)=\sqrt{\frac{|x-y|^{2}}{\sigma^{2}}+1}$} &
{\small $e^{-\frac{|x-y|^{2}}{\sigma^{2}}}$} & {\small $\sqrt{\frac{|x-y|^{2}}{\sigma^{2}}+1}$} & {\small $e^{-\frac{|x-y|^{2}}{\sigma^{2}}}$}\end{tabular}
\caption{Example \ref{ex:data}: Low-rank approximation errors $\frac{\Vert
A-\tilde{A}\Vert_{2}}{\Vert A\Vert_{2}}$ for high-dimensional tests, where
$\sigma$ in the kernel functions is set to be four times the maximum distance
between the data points and the origin, and the SVD line shows the scaled
singular values.}\label{fig:data}\end{figure}

\section{Conclusions\label{sec:concl}}

This work proposes a set of techniques that can make the Nystr\"{o}m method
reach high accuracies in practice for kernel matrix low-rank approximations. The usual Nystr\"{o}m method is
combined with strong rank-revealing
factorizations to serve as a pivoting strategy. The low-rank basis matrices
are refined through alternating direction row and column pivoting. This is
incorporated into a progressive sampling scheme until a desired accuracy or
numerical rank is reached. A fast subset update strategy further leads to
improved efficiency and also convenient randomized accuracy control. The
design of the resulting HAN schemes is based on some strong heuristics,
as supported by some relevant accuracy and singular value analysis. Extensive
numerical tests show that the schemes can quickly reach high accuracies,
sometimes with quality close to SVDs.

The schemes are useful for low-rank approximations related to kernel matrices
in many numerical computations. They can also be used in rank-structured
methods to accelerate various data analysis tasks. The design of the schemes is fully algebraic and does
not require particular information from the kernel
or the data sets. It remains open to give statistical or deterministic analysis of the decay of the approximation error in the progressive sampling
and refinement steps. We are also attempting a probabilistic study of some
steps in the HAN schemes that may be viewed as a randomized rank-revealing factorization.

\end{document}